\documentclass[10pt,letter,reqno]{amsart}
\usepackage{amssymb,amsthm,bbm,graphicx,placeins,enumitem}
\usepackage{amstext}
\usepackage{natbib}
\usepackage{setspace,mathtools,thmtools,thm-restate}
\usepackage[left=1.25in,right=1.25in,top=1in,bottom=1in]{geometry} 
\usepackage{float}
\usepackage{color,colortbl}
\usepackage{changepage}
\usepackage{booktabs}
\usepackage{array}
\usepackage{multirow}
\usepackage[font=small]{caption}
\usepackage{subcaption}
\usepackage[foot]{amsaddr}

\usepackage{hyperref}
\hypersetup{colorlinks,
linkcolor=blue,
filecolor=blue,
urlcolor=blue,
citecolor=blue}

\allowdisplaybreaks

\numberwithin{equation}{section}

\newcommand{\widebar}[1]{\mkern 1.5mu\overline{\mkern-1.5mu#1\mkern-1.5mu}\mkern 1.5mu}

\DeclareMathOperator*{\argmax}{arg\,max}
\DeclareMathOperator*{\argmin}{arg\,min}

\newcommand{\indicate}[1]{\mathbf{1}_{\{#1\}}}
\newcommand{\xt[1]}{x_t^\textnormal{#1}}

\newcommand{\xstart[1]}{x_t^{\textnormal{#1},*}}

\newtheorem*{thma}{Theorem (Convergence)}

\theoremstyle{definition}
\newtheorem{assumption}{Assumption}[section]
\newtheorem{defn}{Definition}[section]

\DeclareCaptionSubType[Alph]{figure}

\def\Rmax{R_{\text{\textnormal{max}}}}
\def\Vmax{V_{\text{\textnormal{max}}}}

\begin{document}
\title[An ADP Algorithm for Monotone Value Functions]{An Approximate Dynamic Programming Algorithm for Monotone Value Functions}


\begin{abstract}
Many sequential decision problems can be formulated as Markov Decision Processes (MDPs) where the optimal value function (or \emph{cost--to--go} function) can be shown to satisfy a monotone structure in some or all of its dimensions. When the state space becomes large, traditional techniques, such as the backward dynamic programming algorithm (i.e., backward induction or value iteration), may no longer be effective in finding a solution within a reasonable time frame, and thus we are forced to consider other approaches, such as approximate dynamic programming (ADP). We propose a provably convergent ADP algorithm called \emph{Monotone--ADP} that exploits the monotonicity of the value functions in order to increase the rate of convergence. In this paper, we describe a general finite--horizon problem setting where the optimal value function is monotone, present a convergence proof for Monotone--ADP under various technical assumptions, and show numerical results for three application domains: \emph{optimal stopping}, \emph{energy storage/allocation}, and \emph{glycemic control for diabetes patients}. The empirical results indicate that by taking advantage of monotonicity, we can attain high quality solutions within a relatively small number of iterations, using up to two orders of magnitude less computation than is needed to compute the optimal solution exactly.
\end{abstract}

\author[Jiang and Powell]{Daniel R. Jiang$^*$ and Warren B. Powell$^*$}
\address{$^*$Department of Operations Research and Financial Engineering, Princeton University}

\maketitle

\section{Introduction}
\label{sec:intro}
Sequential decision problems are an important concept in many fields, including operations research, economics, and finance. For a small, tractable problem, the \emph{backward dynamic programming} (BDP) algorithm (also known as \emph{backward induction} or \emph{finite--horizon value iteration}) can be used to compute the optimal value function, from which we get an optimal decision making policy \citep{Puterman}. However, the state space for many real--world applications can be immense, making this algorithm very computationally intensive. Hence, we often must turn to the field of approximate dynamic programming, which seeks to solve these problems via approximation techniques. One way to obtain a better approximation is to exploit (problem--dependent) structural properties of the optimal value function, and doing so often accelerates the convergence of ADP algorithms. In this paper, we consider the case where the optimal value function is monotone with respect to a partial order. Although this paper focuses on the theory behind our ADP algorithm and not a specific application, we first point out that our technique can be broadly utilized.  Monotonicity is a very common property, as it is true in many situations that ``more is better.'' To be more precise, problems that satisfy \emph{free disposal} (to borrow a term from economics) or \emph{no holding costs} are likely to contain monotone structure. There are also less obvious ways that monotonicity can come into play, such as environmental variables that influence the stochastic evolution of a primary state variable (e.g., extreme weather can lead to increased expected travel times; high natural gas prices can lead to higher electricity spot prices). The following list is a small sample of real--world applications spanning the literature of the aforementioned disciplines (and their subfields) that satisfy the special property of monotone value functions.
\begin{description}[labelindent=2\parindent]

\item[Operations Research]\hfill
\begin{itemize}
\item The problem of optimal replacement of machine parts is well--studied in the literature (see e.g., \cite{Feldstein1974}, \cite{Pierskalla1976}, and \cite{Rust1987}) and can be formulated as a regenerative optimal stopping problem in which the value function is monotone in the current health of the part and the state of its environment. Section \ref{sec:numerical} discusses this model and provides detailed numerical results.
\item The problem of batch servicing of customers at a service station as discussed in \cite{Papadaki2002} features a value function that is monotone in the number of customers. Similarly, the related problem of multiproduct batch dispatch studied in \cite{Papadaki2003} can be shown to have a monotone value function in the multidimensional state variable that contains the number of products awaiting dispatch.
\end{itemize}
\item[Energy]\hfill
\begin{itemize}
\item In the energy storage and allocation problem, one must optimally control a storage device that interfaces with the spot market and a stochastic energy supply (such as wind or solar). The goal is to reliably satisfy a possibly stochastic demand in the most profitable way. We can show that without holding costs, the value function is monotone in the resource (see \cite{Scott2012} and \cite{Salas2013}). Once again, refer to Section \ref{sec:numerical} for numerical work in this problem class.
\item The value function from the problem of maximizing revenue using battery storage while bidding hourly in the electricity market can be shown to satisfy monotonicity in the resource, bid, and remaining battery lifetime (see \cite{Jiang2013a}).
\end{itemize}
\item[Healthcare]\hfill
\begin{itemize}
\item \cite{Hsih2010} develops a model for optimal dosing applied to glycemic control in diabetes patients. At each decision epoch, one of several treatments (e.g., sensitizers, secretagogues, alpha--glucosidase inhibitors, or peptide analogs) with varying levels of ``strength'' (i.e., ability to decrease glucose levels) but also varying side--effects, such as weight gain, needs to be administered. The value function in this problem is monotone whenever the utility function of the state of health is monotone. See Section \ref{sec:numerical} for the complete model and numerical results.
\item Statins are often used as treatment against heart disease or stroke in diabetes patients with lipid abnormalities. The optimal time for statin initiation, however, is a difficult medical problem due to the competing forces of health benefits and side effects. \cite{Kurt2011} models the problem as an MDP with a value function monotone in a risk factor known as the lipid--ratio (LR).
\end{itemize}
\item[Finance]\hfill
\begin{itemize}
\item The problem of mutual fund cash balancing, described in \cite{Nascimento2010}, is faced by fund managers who must decide on the amount of cash to hold, taking into account various market characteristics and investor demand. The value functions turn out to be monotone in the interest rate and the portfolio's rate of return.
\item The pricing problem for American options (see \cite{Luenberger1998}) uses the theory of optimal stopping and depending on the model of the price process, monotonicity can be shown in various state variables: for example, the current stock price or the volatility (see \cite{Ekstrom2004}).
\end{itemize}
\item[Economics]\hfill
\begin{itemize}
\item \cite{Kaplan2014} models the decisions of consumers after receiving fiscal stimulus payments in order to explain observed consumption behavior. The household has both liquid and illiquid assets (the state variable), in which the value functions are clearly monotone.
\item A classical model of search unemployment in economics describes a situation where at each period, a worker has a decision of accepting a wage offer or continuing to search for employment. The resulting value functions can be shown to be increasing with wage (see Section 10.7 of \cite{Stockey1989} and \cite{McCall1970}).
\end{itemize}
\end{description}

This paper makes the following contributions. We describe and prove the convergence of an algorithm, called \emph{Monotone--ADP (M--ADP)} for learning monotone value functions by preserving monotonicity after each update. We also provide empirical results for the algorithm in the context of various applications in operations research, energy, and healthcare as experimental evidence that exploiting monotonicity dramatically improves the rate of convergence. The performance of \emph{Monotone--ADP} is compared to several established algorithms: kernel--based reinforcement learning \citep{Ormoneit2002}, approximate policy iteration \citep{Bertsekas2011}, asynchronous value iteration \citep{Bertsekas2007}, and $Q$--learning \citep{Watkins1992}.

The paper is organized as follows. Section \ref{sec:litreview} gives a literature review, followed by the problem formulation and algorithm description in Section \ref{sec:mathform} and Section \ref{sec:algorithm}. Next, Section \ref{sec:assumptions} provides the assumptions necessary for convergence and Section \ref{sec:convergence} states and proves the convergence theorem, with several proofs of lemmas and propositions postponed until Appendix \ref{sec:appendix}. Section \ref{sec:numerical} describes numerical experiments over a suite of problems, with the largest one having a seven dimensional state variable and nearly 20 million states per time period. We conclude in Section \ref{sec:conclusion}.

\section{Literature Review}
\label{sec:litreview}
General monotone functions (not necessarily a value function) have been extensively studied in the academic literature. The statistical estimation of monotone functions is known as \emph{isotonic} or \emph{monotone} regression and has been studied as early as 1955; see \cite{Ayer1955} or \cite{Brunk1955}. The main idea of isotonic regression is to minimize a weighted error under the constraint of monotonicity (see \cite{Barlow} for a thorough description). The problem can be solved in a variety of ways, including the \emph{Pool Adjacent Violators Algorithm (PAVA)} described in \cite{Ayer1955}. More recently, \cite{Mammen1991} builds upon this previous research by describing an estimator that combines kernel regression and PAVA to produce a smooth regression function. Additional studies from the statistics literature include: \cite{Mukerjee1988}, \cite{Ramsay1998}, \cite{Dette2006}. Although these approaches are outside the context of dynamic programming, the fact that they were developed and well--studied highlights the pertinence of monotone functions.

From the operations research literature, monotone value functions and conditions for monotone optimal policies are broadly described in \cite{Puterman} [Section 4.7] and some general theory is derived therein. Similar discussions of the topic can be found in \cite{Ross1983}, \cite{Stockey1989}, \cite{Muller1997}, and \cite{Smith2002}. The algorithm that we describe in this paper is first used in \cite{Papadaki2002} as a heuristic to solve the \emph{stochastic batch service problem}, where the value function is monotone. However, the convergence of the algorithm is not analyzed and the state variable is scalar. Finally, in \cite{Papadaki}, the authors prove the convergence of the \emph{Discrete On--Line Monotone Estimation} (DOME) algorithm, which takes advantage of a monotonicity preserving step to iteratively estimate a discrete monotone function. DOME, though, was not designed for dynamic programming and the proof of convergence requires independent observations across iterations, which is an assumption that cannot be made for Monotone--ADP.

Another common property of value functions, especially in resource allocation problems, is convexity/concavity. Rather than using a monotonicity preserving step as Monotone--ADP does, algorithms such as the \emph{Successive Projective Approximation Routine (SPAR)} of \cite{Powell2004}, the \emph{Lagged Acquisition ADP Algorithm} of \cite{Nascimento2009a}, and the \emph{Leveling Algorithm} of \cite{Topaloglu2003} use a concavity preserving step, which is the same as maintaining monotonicity in the slopes. The proof of convergence for our algorithm, Monotone--ADP, uses ideas found in \cite{Tsitsiklis1994a} (later also used in \cite{Bertsekas1996}) and \cite{Nascimento2009a}. Convexity has also been exploited successfully in multistage linear stochastic programs (see, e.g, \cite{Birge1985}, \cite{Pereira1991a}, \cite{Asamov2015}). In our work, we take as inspiration the value of convexity demonstrated in the literature and show that monotonicity is another important structural property that can be leveraged in an ADP setting.

\section{Mathematical Formulation}
\label{sec:mathform}
We consider a generic problem with a time--horizon,  $t=0, 1, 2, \ldots, T$. Let $\mathcal S$ be the state space under consideration, where $|S| < \infty$, and let $\mathcal A$ be the set of actions or decisions available at each time step. Let $S_t \in \mathcal S$ be the random variable representing the state at time $t$ and $a_t \in \mathcal A$ be the action taken at time $t$. For a state $S_t \in \mathcal S$ and an action $a_t \in \mathcal A$, let $C_t(S_t,a_t)$ be a contribution or reward received in period $t$ and $C_T(S_T)$ be the terminal contribution. Let $A_t^\pi:\mathcal S \rightarrow \mathcal A$ be the decision function at time $t$ for a policy $\pi$ from the class $\Pi$ of all admissible policies. Our goal is to maximize the expected total contribution, giving us the following objective function:
\begin{equation*}
\sup_{\pi \in \Pi}\; \mathbf{E}\left [ \sum_{t=0}^{T-1} C_t \big(S_t,A_t^\pi(S_t) \big) + C_T(S_T)\right],
\end{equation*}
where we seek a policy to choose the actions $a_t$ sequentially based on the states $S_t$ that we visit. Let $(W_t)_{t=0}^T$ be a discrete time stochastic process that  encapsulates all of the randomness in our problem; we call it the \emph{information process}. Assume that $W_t \in \mathcal W$ for each $t$ and that there exists a state transition function $f: \mathcal S \times \mathcal A \times \mathcal W \rightarrow \mathcal S$ that describes the evolution of the system. Given a current state $S_t$, an action $a_t$, and an outcome of the information process $W_{t+1}$, the next state is given by
\begin{equation}
S_{t+1} = f(S_t,a_t,W_{t+1}).
\label{transitionf}
\end{equation}

Let $s \in \mathcal S$. The optimal policy can be expressed through a set of optimal value functions using the well--known Bellman's equation:
\begin{equation}
\begin{aligned}
&V^*_t(s) =\sup_{a \in \mathcal A} \Bigl [ C_t(s,a)+\textbf{E}\bigl[ V^*_{t+1}(S_{t+1}) \,| \, S_t=s, \, a_t=a \bigr] \Bigr ] \text{ for } t=0,1,2,\ldots,T-1,\\
&V^*_{T}(s) = C_T(s),
\end{aligned}
\label{bellmangen}
\end{equation}
with the understanding that $S_{t+1}$ transitions from $S_t$ according to (\ref{transitionf}). In many cases, the terminal contribution function $C_T(S_T)$ is zero.
Suppose that the state space $\mathcal S$ is equipped with a partial order, denoted $\preceq$, and the following monotonicity property is satisfied for every $t$:
\begin{equation}
s \preceq s' \quad \Longrightarrow \quad V_t^*(s) \le V_t^*(s').
\label{monogen}
\end{equation}
In other words, the optimal value function $V_t^*$ is \emph{order--preserving} over the state space $\mathcal S$. In the case where the state space is multidimensional (see Section \ref{sec:numerical} for examples), a common example of $\preceq$ is \emph{componentwise inequality}, which we henceforth denote using the traditional $\le$.

A second example that arises very often is the following definition of $\preceq$, which we call the \emph{generalized componentwise inequality}. Assume that each state $s$ can be decomposed into $s = (m,i)$ for some $m \in \mathcal M$ and $i \in \mathcal I$. For two states $s = (m,i)$ and $s' = (m', i')$, we have
\begin{equation}
s \preceq s' \quad \Longleftrightarrow \quad m \le m', \; i = i'.
\label{eq:componentwise}
\end{equation}
In other words, we know that whenever $i$ is held constant, then the value function is monotone in the ``primary'' variable $m$. An example of when such a model would be useful is when $m$ represents the amount of some held resource that we are both buying and selling, while $i$ represents additional \emph{state-of-the-world} information, such as prices of related goods, transport times on a shipping network, or weather information. Depending on the specific model, the relationship between the value of $i$ and the optimal value function may be quite complex and a priori unknown to us. However, it is likely to be obvious that for $i$ held constant, the value function is increasing in $m$, the amount of resource that we own --- hence, the definition (\ref{eq:componentwise}). The following proposition is given in the setting of the generalized componentwise inequality and provides a simple condition that can be used to verify monotonicity in the value function.

\begin{restatable}{prop}{monocondone} Suppose that every $s \in \mathcal S$ can be written as $s = (m,i)$ for some $m \in \mathcal M$ and $i \in \mathcal I$ and let $S_t = (M_t,I_t)$ be the state at time $t$, with $M_t \in \mathcal M$ and $I_t \in \mathcal I$. Let the partial order $\preceq$ on the state space $\mathcal S$ be described by (\ref{eq:componentwise}). Assume the following assumptions hold. 
\begin{enumerate}[label=(\roman*),labelindent=1in]
\item For every $s, s' \in \mathcal S$ with $s \preceq s'$, $a \in \mathcal A$, and $w \in \mathcal W$, the state transition function satisfies
\[
f(s,a,w) \preceq f(s',a,w).
\]
\item For each $t < T$, $s, s' \in \mathcal S$ with $s \preceq s'$, and $a \in \mathcal A$,
\begin{equation*}
C_t(s,a) \le C_t(s',a) \, \mbox{ and } \, C_T(s) \le C_T(s').
\end{equation*}
\item For each $t < T$, $M_t$ and $W_{t+1}$ are independent.
\end{enumerate}
Then, the value functions $V_t^*$ satisfy the monotonicity property of (\ref{monogen}).
\label{mono_cond_one}
\end{restatable}
\begin{proof}
See Appendix \ref{sec:appendix}.
\end{proof}
There are other similar ways to check for monotonicity; for example, see Proposition 4.7.3 of \cite{Puterman} or Theorem 9.11 of \cite{Stockey1989} for conditions on the transition probabilities. We choose to provide the above proposition due to its relevance to our example applications in Section \ref{sec:numerical}.

The most traditional form of Bellman's equation has been given in (\ref{bellmangen}), which we refer to as the \emph{pre--decision state} version. Next, we discuss some alternative formulations from the literature that can be very useful for certain problem classes. A second formulation, called the $Q$--function (or \emph{state--action}) form Bellman's equation, is popular in the field of reinforcement learning, especially in applications of the widely used $Q$--learning algorithm (see \cite{Watkins1992}): 
\begin{equation}
\begin{aligned}
&Q^*_t(s,a) =  \textbf{E}\Bigl[C_t(s,a)+ \max_{a_{t+1} \in \mathcal A} Q_{t+1}^*(S_{t+1},a_{t+1}) \,| \, S_t = s, \, a_t = a \Bigr] \text{ for } t=0,1,2,\ldots,T-1,\\
&Q^*_{T}(s,a) = C_T(s),
\end{aligned}
\label{bellmanq}
\end{equation}
where we must now impose the additional requirement that $\mathcal A$ is a finite set. $Q^*$ is known as the \emph{state--action value function} and the ``state space'' in this case is enlarged to be $\mathcal S \times \mathcal A$.

A third formulation of Bellman's equation is in the context of \emph{post--decision states} (see \cite{Powell2011} for a detailed treatment of this important technique). Essentially, the post--decision state, which we denote $S_t^a$, represents the state after the decision has been made, but before the random information $W_{t+1}$ has arrived (the state--action pair is also a post--decision state). For example, in the simple problem of purchasing additional inventory $x_t$ to the current stock $R_t$ in order to satisfy a next--period stochastic demand, the post--decision state can be written as $R_t+x_t$, while the pre--decision state is $R_t$. It must be the case that $S_t^a$ contains the same information as the state--action pair $(S_t,a_t)$, meaning that regardless of whether we condition on $S_t^a$ or $(S_t,a_t)$, the conditional distribution of $W_{t+1}$ is the same. The attractiveness of this method is that 1) in certain problems, $S_t^a$ is of lower dimension than $(S_t,a_t)$ and 2) when writing Bellman's equation in terms of the post--decision state space (using a redefined value function), the supremum and the expectation are interchanged, giving us some computational advantages. Let $s^a$ be a post--decision state from the post--decision state space $\mathcal S^a$. Bellman's equation becomes
\begin{equation}
\begin{aligned}
&V^{a,*}_t(s^a) =  \mathbf{E}\left[  \sup_{a \in \mathcal A} \, \Bigl[ C_{t+1}(S_{t+1},a)+ V_{t+1}^{a,*}(S_{t+1}^a)\Bigr] \, \Bigl | \, S_t^a = s^a  \right] \text{ for } t=0,1,2,\ldots,T-2,\\
&V^{a,*}_{T-1}(s^a) = \mathbf{E} \bigl[ C_T(S_{T}) \, | \, S_{T-1}^a = s^a \bigr],
\end{aligned}
\label{bellmanpostdec}
\end{equation}
where $V^{*,a}$ is known as the \emph{post--decision value function}.
In approximate dynamic programming, the original Bellman's equation formulation (\ref{bellmangen}) can be used if the transition probabilities are known. When the transition probabilities are unknown, we must often rely purely on \emph{experience} or some form of \emph{black box simulator}. In these situations, formulations (\ref{bellmanq}) and (\ref{bellmanpostdec}) of Bellman's equation, where the optimization is within the expectation, become extremely useful. For the remainder of this paper, rather than distinguishing between the three forms of the value function ($V^*$, $Q^*$, and $V^{a,*}$), we simply use $V^*$ and call it the \emph{optimal value function}, with the understanding that it may be replaced with any of the definitions. Similarly, to simplify notation, we do not distinguish between the three forms of the state space ($\mathcal S$, $\mathcal S \times \mathcal A$, and $\mathcal S^a$) and simply use $\mathcal S$ to represent the domain of the value function (for some $t$).

Let $d = |\mathcal S|$ and $D = (T+1) \, |\mathcal S|$. We view the optimal value function as a vector in $\mathbb R^D$, that is to say, $V^* \in \mathbb R^D$ has a component at $(t,s)$ denoted as $V_t^*(s)$. Moreover, for a fixed $t \le T$, the notation $V_t^* \in \mathbb R^d$ is used to describe $V^*$ restricted to $t$, i.e., the components of $V_t^*$ are $V_t^*(s)$ with $s$ varying over $\mathcal S$. We adopt this notational system for arbitrary value functions $V \in \mathbb R^D$ as well. Finally, we define the \emph{generalized dynamic programming operator} $H : \mathbb R^D \rightarrow \mathbb R^D$, which applies the right hand sides of either (\ref{bellmangen}), (\ref{bellmanq}), or (\ref{bellmanpostdec}) to an arbitrary $V \in \mathbb R^D$, i.e., replacing $V_t^*$, $Q_t^*$, and $V_t^{*.a}$ with $V_t$. For example, if $H$ is defined in the context of (\ref{bellmangen}), then the component of $HV$ at $(t,s)$ is given by
\begin{equation}
(HV)_t(s) = \left\{
	\begin{array}{ll}
		 \sup_{a \in \mathcal A} \Bigl [ C_t(s,a)+\mathbf{E} \bigl[ V_{t+1}(S_{t+1})\,|\,S_t = s, \, a_t=a \bigr] \Bigr ] & \mbox{for } t=0,1,2,\ldots,T-1, \vspace{.6em} \\
		 C_T(s) & \mbox{for } t=T.
	\end{array}
\right.
\label{Hdef}
\end{equation}
For (\ref{bellmanq}) and (\ref{bellmanpostdec}), $H$ can be defined in an analogous way. We now state a lemma concerning useful properties of $H$. Parts of it are similar to Assumption 4 of \cite{Tsitsiklis1994a}, but we can show that these statements always hold true for our more specific problem setting, where $H$ is a generalized dynamic programming operator.
\begin{restatable}{lem}{Hprops}
The following statements are true for $H$, when it is defined using (\ref{bellmangen}), (\ref{bellmanq}), or (\ref{bellmanpostdec}).
\begin{enumerate}[label=(\roman*),labelindent=1in]
\item $H$ is monotone, i.e., for $V,\,V' \in \mathbb R^D$ such that $V \le V'$, we have that $HV \le HV'$ (componentwise).
\item For any $t < T$, let $V,\,V' \in \mathbb R^D$, such that $V_{t+1} \le V_{t+1}'$. It then follows that $(HV)_t \le (HV')_t$.
\item The optimal value function $V^*$ uniquely satisfies the fixed point equation $HV = V$.
\item Let $V \in \mathbb R^D$ and $e$ is a vector of ones with dimension $D$. For any $\eta > 0$,
\[HV-\eta e \le H(V-\eta e) \le H(V+\eta e) \le HV+\eta e.
\]
\end{enumerate}
\label{Hprops}
\end{restatable}
\begin{proof}
See Appendix \ref{sec:appendix}.
\end{proof}


\section{Algorithm}
\label{sec:algorithm}
In this section, we formally describe the Monotone--ADP algorithm. Assume a probability space $(\Omega, \mathcal F, \mathbf P)$ and let $\widebar{V}^n$ be the approximation of $V^*$ at iteration $n$, with the random variable $S_t^n \in \mathcal S$ representing the state that is visited (by the algorithm) at time $t$ in iteration $n$. The observation of the optimal value function at time $t$, iteration $n$, and state $S_t^n$ is denoted $\hat{v}_t^n(S_t^n)$ and is calculated using the estimate of the value function from iteration $n-1$. The raw observation $\hat{v}_t^n(S_t^n)$ is then \emph{smoothed} with the previous estimate $\widebar{V}_t^{n-1}(S_t^n)$, using a stochastic approximation step, to produce the smoothed observation $z_t^n(S_t^n)$. Before presenting the description of the ADP algorithm, some definitions need to be given. We start with $\Pi_M$, the monotonicity preserving projection operator. Note that the term ``projection'' is being used loosely here; the space that we ``project'' onto actually changes with each iteration.
\begin{defn}
For $s^r \in \mathcal S$ and $z^r \in \mathbb R$, let $(s^r,z^r)$ be a \emph{reference point} to which other states are compared. Let $V_t \in \mathbb R^d$ and define the projection operator $\Pi_M : \mathcal S \times \mathbb R \times \mathbb R^d \rightarrow \mathbb R^d$, where the component of the vector $\Pi_M(s^r,z^r,V_t)$ at $s$ is given by 
\begin{equation}
\Pi_M\bigl(s^r,z^r,V_t\bigr)(s) = \left\{
  \begin{array}{ll}
    z^r & \mbox{if } s = s^r, \vspace{.6em} \\
    z^r \vee V_t(s)  & \mbox{if } s^r \preceq s,\;s \ne s^r,
\vspace{.6em} \\
    z^r \wedge V_t(s) & \mbox{if } s^r \succeq s,\;s \ne s^r,
\vspace{.6em} \\
  V_t(s) & \mbox{otherwise.}
  \end{array}
\right.
\label{genprojection]}
\end{equation}
\end{defn}
In the context of the Monotone--ADP algorithm, $V_t$ is the current value function approximation, $(s^r,z^r)$ is the latest observation of the value ($s^r$ is latest visited state), and $\Pi_M(s^r,z^r,V_t)$ is the updated value function approximation. Violations of the monotonicity property of (\ref{monogen}) are corrected by $\Pi_M$ in the following ways:
\begin{itemize}
\item if $z^r \ge V_t(s)$ and $s^r \preceq s$, then $V_t(s)$ is \emph{too small} and is increased to $z^r = z^r \vee V_t(s)$, and
\item if $z^r \le V_t(s)$ and $s^r \succeq s$, then $V_t(s)$ is \emph{too large} and is decreased to $z^r = z^r \wedge V_t(s)$. 
\end{itemize}
See Figure \ref{fig:projection} for an example showing a sequence of two observations and the resulting projections in the Cartesian plane, where $\preceq$ is the componentwise inequality in two dimensions. We now provide some additional motivation for the definition of $\Pi_M$. Because $z_t^n(S_t^n)$ is the \emph{latest} observed value and it is obtained via stochastic approximation (see the Step 2b of Figure \ref{alg:vfa2}), our intuition guides us to ``keep'' this value, i.e., by setting $\widebar{V}_t^n(S_t^n) = z_t^n(S_t^n)$. For $s \in \mathcal S$ and $v \in \mathbb R$, let us define the set 
\[
\mathcal V_{\mathcal M}(s,z) = \bigl \{ V \in \mathbb R^d : V(s) = z, \, V \mbox{ monotone over } \mathcal S \bigr \}
\] which fixes the value at $s$ to be $z$, while restricting to the set of all possible $V$ that satisfy the monotonicity property (\ref{monogen}). Now, to get the approximate value function of iteration $n$ and time $t$, we want to find $\widebar{V}_t^n$ that is close to $\widebar{V}_t^{n-1}$ but also satisfies the monotonicity property:
\begin{equation}
\widebar{V}_t^n \in \argmin \Bigl \{ \bigl \|V_t - \widebar{V}_t^{n-1} \bigr \|_2 : V_t \in \mathcal V_\mathcal{M}\bigl(S_t^n,z_t^n(S_t^n)\bigr) \Bigr\},
\label{prop:argmin}
\end{equation}
where $\| \cdot \|_2$ is the Euclidean norm. Let us now briefly pause and consider a possible alternative, where we do not require $\widebar{V}_t^n(S_t^n) = z_t^n(S_t^n)$. Instead, suppose we introduce a vector $\hat{V}_t^{n-1} \in \mathbb R^d$ such that $\hat{V}_t^{n-1}(s) = \widebar{V}_t^{n-1}(s)$ for $s \ne S_t^n$ and $\hat{V}_t^{n-1}(S_t^n) = z_t^n(S_t^n)$. Next, project $\hat{V}_t^{n-1}$ the space of vectors $V$ that are monotone over $\mathcal S$ to produce $\widebar{V}_t^n$ (this would be a proper projection, where the space does not change). The problem with this approach arises in the early iterations where we have poor estimates of the value function: for example, if $\widebar{V}_t^0(s) = 0$ for all $s$, then $\hat{V}_t^0$ is a vector of mostly zeros and the likely result of the projection, $\widebar{V}_t^1$, would be the \emph{original vector} $\widebar{V}_t^{0}$ --- hence, no progress is made. A potential explanation for the failure of such a strategy is that it is a naive adaptation of the natural approach for a batch framework to a recursive setting.

The next proposition shows that this representation of $\widebar{V}_t^n$ is equivalent to one that is obtained using the projection operator $\Pi_M$.
\begin{restatable}{prop}{Propproj}
The solution to the minimization (\ref{prop:argmin}) can be characterized using $\Pi_M$. Specifically,
\[
\Pi_M \bigl(S_t^n,z_t^n(S_t^n),\widebar{V}_t^{n-1}\bigr) \in \argmin \Bigl \{ \bigl \|V_t - \widebar{V}_t^{n-1} \bigr \|_2 : V_t \in \mathcal V_\mathcal{M}\bigl(S_t^n,z_t^n(S_t^n)\bigr) \Bigr\},
\]
so that we can write $\widebar{V}_t^n = \Pi_M \bigl(S_t^n,z_t^n(S_t^n),\widebar{V}_t^{n-1}\bigr)$.
\label{Propproj}
\end{restatable}
\begin{proof}
See Appendix \ref{sec:appendix}.
\end{proof}

\begin{figure}[h]
	\begin{center}
	\includegraphics[scale=.4]{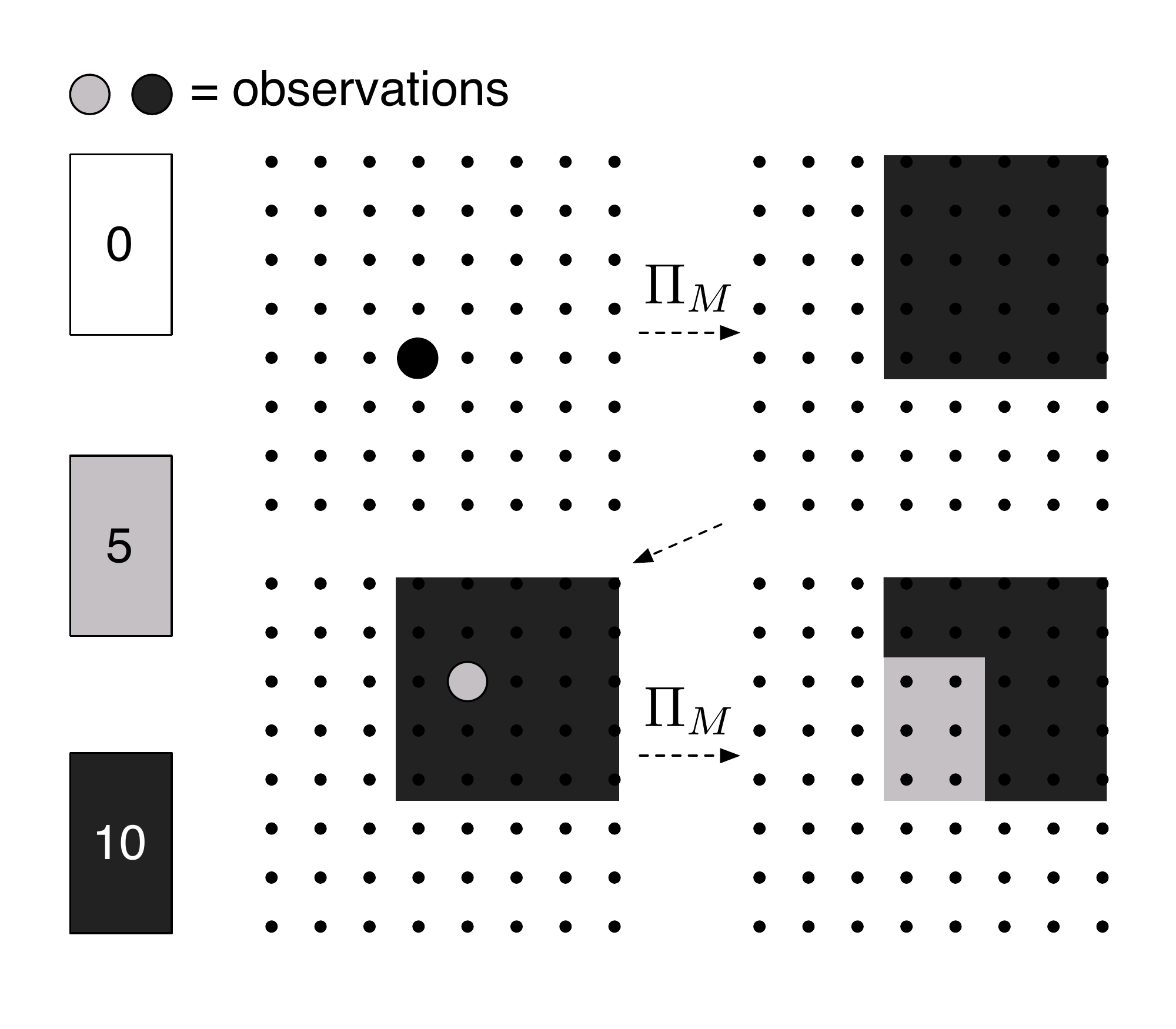}\\
	\end{center}
	\caption{Example Illustrating the Projection Operator $\Pi_M$}
	\label{fig:projection}
\end{figure}

We now introduce, for each $t$, a (possibly stochastic) stepsize sequence $\alpha_t^n \le 1$ used for smoothing in new observations. The algorithm only directly updates values (i.e., not including updates from the projection operator) for states that are visited, so for each $s \in \mathcal S$, let
\begin{equation*}
\alpha_t^n(s) = \alpha_t^{n-1} \, \textbf{1}_{\{s=S_t^n\}}.
\end{equation*}
Let $\hat{v}_t^n \in \mathbb R^d$ be a noisy observation of the quantity $(H\widebar{V}^{n-1})_t$, and let $w_t^n \in \mathbb R^d$ represent the additive noise associated with the observation:
\begin{equation*}
\hat{v}_t^n  = \bigl(H\widebar{V}^{n-1}\bigr)_t + w_t^n.
\end{equation*} 
Although the algorithm is asynchronous and only updates the value for $S_t^n$ (therefore, it only needs $\hat{v}_t^n(S_t^n)$, the component of $\hat{v}_t^n$ at $S_t^n$), it is convenient to assume $\hat{v}_t^n(s)$ and $w_t^n(s)$ are defined for all $s$. We also require a vector $z_t^n \in \mathbb R^d$ to represent the ``smoothed observation'' of the future value, i.e., $z_t^n(s)$ is $\hat{v}_t^n(s)$ smoothed with the previous value $\widebar{V}_t^{n-1}(s)$ via the stepsize $\alpha_t^n(s)$.
Let us denote the history of the algorithm up until iteration $n$ by the filtration $\{\mathcal F^n\}_{n \ge 1}$, where
\begin{equation*}
\mathcal F^n = \sigma \bigl\{(S_{t}^m,\; w_{t}^m)_{\; m\le n,\; t \le T}\bigr\}.
\end{equation*}
A precise description of the algorithm is given in Figure \ref{alg:vfa2}.
\begin{figure}[h] 
\mbox{}\hrulefill\mbox{}
\begin{description}[leftmargin=5.2em,style=nextline]
    \item[Step 0a.]Initialize $\widebar{V}_t^0 \in [0,\Vmax]$ for each $t \le T-1 $ such that monotonicity is satisfied within $\widebar{V}_t^0$, as described in (\ref{monogen}).
    \vspace{.5em}
    \item[Step 0b.]Set $\widebar{V}_{T}^n(s) = C_T(s)$ for each $s \in \mathcal{S}$ and $n \le N$.
    \vspace{.5em}
    \item[Step 0c.]Set $n=1$.
    \vspace{.5em}    
    \item[Step 1.]Select an initial state $S_0^n$.
    \vspace{.5em}
    \item[Step 2.]For $t=0,1, \ldots, (T-1)$:
  	\vspace{.5em}    
    \begin{description}[leftmargin=5.2em,style=nextline]
    	\item[Step 2a.] Sample a noisy observation of the future value:\\
    	\vspace{.5em}
    	\quad $\hat{v}_t^n= \bigl(H\widebar{V}^{n-1}\bigr)_t+w_t^n$.
    	\vspace{.5em}
    	\item[Step 2b.] Smooth in the new observation with previous value at each $s$:\\
    	\vspace{.5em}
    	\quad $z_t^n(s) = \bigl(1-\alpha_t^{n}(s)\bigr) \, \widebar{V}_t^{n-1}(s)+\alpha_t^{n}(s) \, \hat{v}_t^n(s).$
    	\vspace{.5em}
    	\item[Step 2c.] Perform monotonicity projection operator:\\
    	\vspace{.5em}
    	\quad $\widebar{V}_t^n = \Pi_M\bigl(S_t^n, z_t^n(S_t^n),\widebar{V}_t^{n-1}\bigr)$.
    	\vspace{.5em}
    	\item[Step 2d.] Choose the next state $S_{t+1}^n$ given $\mathcal F^{n-1}$.
    \end{description}
    \vspace{.5em}
    \item[Step 3.] If $n < N$, increment $n$ and return to \textbf{Step 1}.
\end{description}
 \mbox{}\hrulefill\mbox{}
\caption{Monotone--ADP Algorithm}
\label{alg:vfa2}
\end{figure}
Notice from the description that if the monotonicity property (\ref{monogen}) is satisfied at iteration $n-1$, then the fact that the projection operator $\Pi_M$ is applied ensures that the monotonicity property is satisfied again at time $n$. Our benchmarking results of Section \ref{sec:numerical} show that maintaining monotonicity in such a way is an invaluable aspect of the algorithm that allows it to produce very good policies in a relatively small number of iterations. Traditional approximate (or asynchronous) value iteration, on which Monotone--ADP is based, is asymptotically convergent but \emph{extremely} slow to converge in practice (once again, see Section \ref{sec:numerical}). As we have mentioned, $\Pi_M$ is not a standard projection operator, as it ``projects'' to a different space on every iteration, depending on the state visited and value observed; therefore, traditional convergence results no longer hold. The remainder of the paper establishes the asymptotic convergence of Monotone--ADP. 

\subsection{Extensions of Monotone--ADP} We now briefly present two possible extensions of Monotone--ADP. First, consider a discounted, infinite horizon MDP. An extension (or perhaps, simplification) to this case can be obtained by removing the loop over $t$ (and all subscripts of $t$ and $T$) and acquiring one observation per iteration, exactly resembling asynchronous value iteration for infinite horizon problems.

Second, we consider possible extensions when representations of the approximate value function other than lookup table are used; for example, imagine we are using basis functions $\{\phi_g\}_{g \in \mathcal G}$ for some feature set $\mathcal G$ combined with a coefficient vector $\theta_{t}^n$ (which has components $\theta_{tg}^n$), giving the approximation
\[
\widebar{V}_t^n(s) = \sum_{g \in \mathcal G} \theta_{tg}^n \, \phi_g(s).
\]
Equation (\ref{prop:argmin}) is the starting point for adapting Monotone--ADP to handle this case. An analogous version of this update might be given by
\begin{equation}
\theta_t^n \in \argmin \Bigl \{ \|\theta_t - \theta_t^{n-1} \|_2 : \widebar{V}^n_t(S_t^n) = z_t^n(S_t^n) \mbox{ and } \widebar{V}^n_t \textnormal{ monotone} \Bigr\},
\label{eq:argminbasis}
\end{equation}
where we have altered the objective to minimize distance in the coefficient space. Unlike (\ref{prop:argmin}), there is, in general, no simple and easily computable solution to (\ref{eq:argminbasis}), but special cases may exist. The analysis of this situation is beyond the scope of this paper and left to future work. In this paper, we consider the finite horizon case using a lookup table representation.

\section{Assumptions}
\label{sec:assumptions}
We begin by providing some technical assumptions that are needed for convergence analysis. The first assumption gives, in more general terms than previously discussed, the monotonicity of the value functions.
\begin{assumption}
The two monotonicity assumptions are as follows.
\begin{enumerate}[label=(\roman*),labelindent=1in]
\item The terminal value function $C_T$ is monotone over $\mathcal S$ with respect to $\preceq$.
\item For any $t < T$ and any vector $V \in \mathbb R^D$ such that $V_{t+1}$ is monotone over $\mathcal S$ with respect to $\preceq$, it is true that $(HV)_t$ is monotone over the state space as well.
\end{enumerate}
\label{ass:Emono}
\end{assumption}
The above assumption implies that for any choice of terminal value function $V^*_T = C_T$ that satisfies monotonicity, the value functions for the previous time periods are monotone as well. Examples of sufficient conditions include monotonicity in the contribution function plus a condition on the transition function, as in $(i)$ of Proposition \ref{mono_cond_one}, or a condition on the transition probabilities, as in Proposition 4.7.3 of \cite{Puterman}. Intuitively speaking, when the statement \emph{``starting with more at $t$ $\Rightarrow$ ending with more at $t+1$''} applies, in expectation, to the problem at hand, Assumption \ref{ass:Emono} is satisfied. One obvious example that satisfies monotonicity occurs in resource or asset management scenarios; oftentimes in these problems, it is true that for any outcome of the random information $W_{t+1}$ that occurs (e.g., random demand, energy production, or profits), we end with more of the resource at time $t+1$ whenever we start with more of the resource at time $t$. Mathematically, this property of resource allocation problems translates to the stronger statement:
\[
S_{t+1}\,|\,S_t=s,\,a_t=a \; \preceq \; S_{t+1}\,|\,S_t=s',\,a_t=a \quad a.s.
\]
for all $a \in \mathcal A$ when $s \preceq s'$. This is essentially the situation that Proposition \ref{mono_cond_one} describes.
\begin{assumption}
For all $s \in \mathcal S$ and $t < T$, the sampling policy satisfies
\begin{equation*}
\sum_{n=1}^\infty \mathbf{P}\bigl(S_t^n = s \,|\, \mathcal F^{n-1}\bigr) = \infty \quad a.s.
\end{equation*}
\label{ioass}
\end{assumption}
By the Extended Borel--Cantelli Lemma (see \citet{Breiman1992}), any scheme for choosing states that satisfies the above condition will visit every state infinitely often with probability one.
\begin{assumption}
Suppose that the contribution function $C_t(s,a)$ is bounded: without loss of generality, let us assume that for all $s \in \mathcal S$, $t < T$, and $a \in \mathcal A$, $
0 \le C_t(s,a)\le C_\textnormal{max}$, for some $C_\textnormal{max} > 0$. Furthermore, suppose that $0 \le C_T(s) \le C_\textnormal{max}$ for all $s \in \mathcal S$ as well. This naturally implies that there exists $V_\textnormal{max} > 0$ such that $0 \le V_t^*(s) \le V_\textnormal{max}$.
\label{Cbounded}
\end{assumption}
The next three assumptions are standard ones made on the observations $\hat{v}_t^n$, the noise $w_t^n$, and the stepsize sequence $\alpha_t^n$; see \citet{Bertsekas1996} (e.g. Assumption 4.3 and Proposition 4.6) for additional details.
\begin{assumption}
The observations that we receive are bounded (by the same constant $V_\textnormal{max}$):
\begin{equation*}
0 \le \hat{v}_t^n(s) \le V_\textnormal{max} \quad a.s.,
\end{equation*}
for all $s \in \mathcal S$.
\label{boundedobs}
\end{assumption}
Note that the lower bounds of zero in Assumptions \ref{Cbounded} and \ref{boundedobs} are chosen for convenience and can be shifted by a constant to suit the application (as is done in Section \ref{sec:numerical}).

\begin{assumption}
The following holds almost surely:
\[
\mathbf{E} \bigl[w_t^{n+1}(s) \, | \, \mathcal F^n \bigr] = 0,
\]
for any state $s \in \mathcal S$.
This property means that $w_t^n$ is a \emph{martingale difference noise} process.

\label{wass}
\end{assumption}
\begin{assumption}
For each $t \le T$, $s \in\mathcal S$, suppose $\alpha_t^n \in [0,1]$ is $\mathcal F^n$--measurable and
\vspace{0.5em}
\begin{enumerate}[label=(\roman*),labelindent=1in]
\item $\displaystyle \sum_{n=1}^\infty \alpha_t^n(s) = \infty \quad a.s.$,
\vspace{0.5em}
\item $\displaystyle \sum_{n=1}^\infty \alpha_t^n(s)^2 < \infty \quad a.s.$\\
\end{enumerate}\label{Stepassumption}
\end{assumption}

\subsection{Remarks on Simulation}
\label{sec:simremarks}
Before proving the theorem, we offer some additional comments regarding the assumptions as they pertain to simulation. If $H$ is defined in the context of (\ref{bellmangen}), then it is not easy to perform Step 2a of Figure \ref{alg:vfa2},
$$\hat{v}_t^n = \bigl(H\widebar{V}^{n-1}\bigr)_t+w_t^n,$$
 such that Assumption \ref{wass} is satisfied. Due to the fact that the supremum is outside of the expectation operator, an upward bias would be present in the observation $\hat{v}_t^n(s)$ \emph{unless} the expectation can be computed exactly, in which case $w_t^n(s) = 0$ and we have
\begin{equation}
\hat{v}_t^n(s) =\sup_{a \in \mathcal A} \Bigl [ C_t(s,a)+\textbf{E}\bigl[ \widebar{V}^{n-1}_{t+1}(S_{t+1}) \,| \, S_t=s, \, a_t=a \bigr] \Bigr ].
\label{eq:vhat1}
\end{equation}
Thus, any approximation scheme used to calculate the expectation inside of the supremum would cause Assumption \ref{wass} to be unsatisfied. When the approximation scheme is a sample mean, the bias disappears asymptotically with the number of samples (see \cite{Kleywegt2002}, which discusses the \emph{sample average approximation} or SAA method). It is therefore possible that although theoretical convergence is not guaranteed, a large enough sample may still achieve decent results \emph{in practice}.

 On the other hand, in the context of (\ref{bellmanq}) and (\ref{bellmanpostdec}), the expectation and the supremum are interchanged. This means that we can trivially obtain an unbiased estimate of $(H\widebar{V}^{n-1})_t$ by sampling \emph{one outcome} of the information process $W^n_{t+1}$ from the distribution $W_{t+1}\,|\,S_t=s$, computing the next state $S_{t+1}^n$, and solving a deterministic optimization problem (i.e., the optimization within the expectation). In these two cases, we would respectively use
 \begin{equation}
 \hat{v}_t^n(s,a) =  C_t(s,a)+ \max_{a_{t+1} \in \mathcal A} \widebar{Q}^{n-1}_{t+1}(S_{t+1}^n,a_{t+1})
 \label{eq:vhat2}
 \end{equation}
 and
 \begin{equation}
 \hat{v}_t^n(s^a) =  \sup_{a \in \mathcal A} \, \Bigl[ C_{t+1}(S^n_{t+1},a)+ \widebar{V}_{t+1}^{a,n-1}(S_{t+1}^{a,n})\Bigr],
 \label{eq:vhat3}
 \end{equation}
where $\widebar{Q}^{n-1}_{t+1}$ is the approximation to $Q_{t+1}^*$, $\widebar{V}_t^{a,n-1}$ is the approximation to $V_t^{a,*}$, and $S_{t+1}^{a,n}$ is the post--decision state obtained from $S_{t+1}^n$ and $a$. Notice that (\ref{eq:vhat1}) contains an expectation while (\ref{eq:vhat2}) and (\ref{eq:vhat3}) do not, making them particularly well--suited for model--free situations, where distributions are unknown and only samples or experience are available. Hence, the best choice of model depends heavily upon the problem domain. 

Finally, we give a brief discussion of the choice of stepsize. There are a variety of ways in which we can satisfy Assumption \ref{Stepassumption}, and here we offer the simplest example. Consider any deterministic sequence $\{a^n\}$ such that the usual stepsize conditions are satisfied:
\[
\sum_{n=0}^\infty a^n = \infty \; \mbox{ and } \; \sum_{n=0}^\infty (a^n)^2 < \infty.
\]
Let $N(s,n,t) = \sum_{m=1}^n\mathbf{1}_{\{s=S_t^m\}}$ be the random variable representing the total number of visits of state $s$ at time $t$ until iteration $n$.
Then, $\alpha_t^n = a^{N(S_t^n,n,t)}$ satisfies Assumption \ref{Stepassumption}.

\section{Convergence Analysis of the Monotone--ADP Algorithm}
\label{sec:convergence}
We are now ready to show the convergence of the algorithm. Note that although there is a significant similarity between this algorithm and the Discrete On--Line Monotone Estimation (DOME) algorithm described in \citet{Papadaki}, the proof technique is very different. The convergence proof for the DOME algorithm cannot be directly extended to our problem due to differences in the assumptions. 

Our proof draws on proof techniques found in \citet{Tsitsiklis1994a} and \citet{Nascimento2009a}. In the latter, the authors prove convergence of a purely exploitative ADP algorithm given a concave, piecewise--linear value function for the lagged asset acquisition problem. We cannot exploit certain properties inherent to that problem, but in our algorithm we assume exploration of all states, a requirement that can be avoided when we are able to assume concavity. Furthermore, a significant difference in this proof is that we consider the case where $\mathcal S$ may not be a total ordering. A consequence of this is that we extend to the case where the monotonicity property covers multiple dimensions (e.g., the relation on $\mathcal S$ is the componentwise inequality), which was not allowed in \cite{Nascimento2009a}.
\begin{thma}
Under Assumptions \ref{ioass}--\ref{Stepassumption}, for each $t \le T$ and $s \in \mathcal S$, the estimate $\widebar{V}_t^n(s)$ produced by the Monotone--ADP Algorithm of Figure \ref{alg:vfa2}, converge to the optimal value function $V_t^*(s)$ almost surely.
\end{thma}
Before providing the proof for this convergence result, we present some preliminary definitions and results. First, we define two deterministic bounding sequences, $U^k$ and $L^k$. The two sequences $U^k$ and $L^k$ can be thought of, jointly, as a sequence of ``shrinking'' rectangles, with $U^k$ being the upper bounds and $L^k$ being the lower bounds. The central idea to the proof is showing that the estimates $\widebar{V}^n$ enter (and stay) in smaller and smaller rectangles, for a fixed $\omega \in \Omega$ (we assume that the $\omega$ does not lie in a discarded set of probability zero). We can then show that the rectangles converge to the point $V^*$, which in turn implies the convergence of $\widebar{V}^n$ to the optimal value function. This idea is attributed to \cite{Tsitsiklis1994a} and is illustrated in Figure \ref{fig:boundingrects}.
\begin{figure}[h]
	\begin{center}
	\includegraphics[scale=.5]{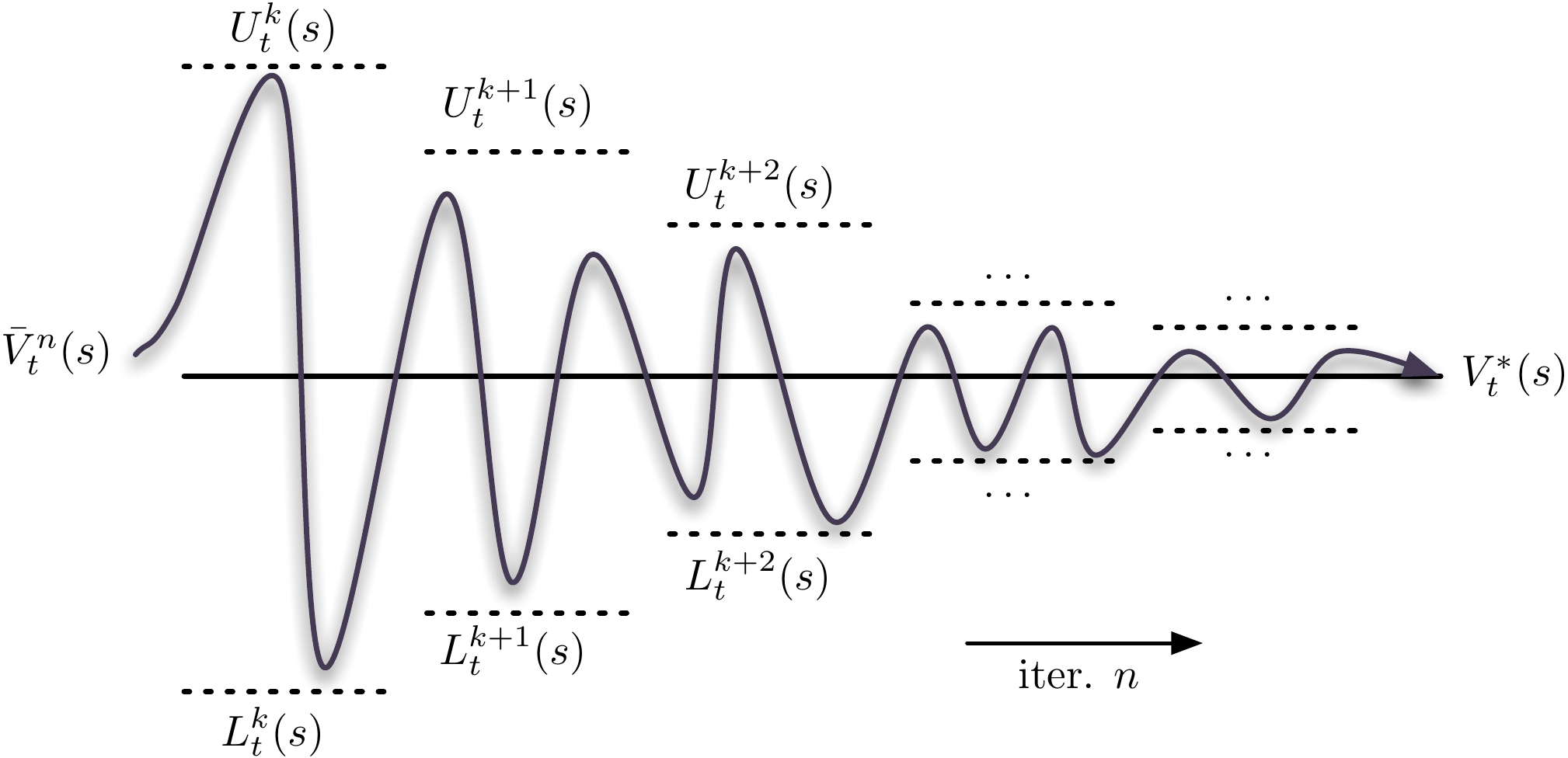}\\
	\end{center}
	\caption{Central Idea of Convergence Proof}
	\label{fig:boundingrects}
\end{figure}

The sequences $U^k$ and $L^k$ are written recursively. Let
\begin{equation}
\begin{aligned}
U^0 &= V^*+V_\textnormal{max} \cdot e,\\
L^0 &= V^*-V_\textnormal{max} \cdot e,
\end{aligned}
\label{Udefn1}
\end{equation}
and let
\begin{equation*}
\begin{aligned}
U^{k+1} &= \frac{U^k+HU^k}{2},\\
L^{k+1} &= \frac{L^k+HL^k}{2}.
\end{aligned}
\end{equation*}

\begin{restatable}{lem}{LemmaOne}
For all $k \ge 0$, we have that
\begin{equation*}
\begin{aligned}
HU^k &\le U^{k+1} \le U^k,\\
HL^k &\ge L^{k+1} \ge L^k.\\
\end{aligned}
\end{equation*}
Furthermore,
\begin{equation}
\begin{aligned}
U^{k} &\longrightarrow V^*,\\
L^{k} &\longrightarrow V^*.\\
\end{aligned}
\label{ULconv}
\end{equation}
\label{lem:1}
\end{restatable}
\begin{proof}
The proof of this lemma is given in \cite{Bertsekas1996} (see Lemma 4.5 and Lemma 4.6). The properties of $H$ given in Proposition \ref{Hprops} are used for this result.
\end{proof}
\begin{restatable}{lem}{LemmaTwo}
The bounding sequences satisfy the monotonicity property; that is, for $k \ge 0$, $t \le T$, $s \in \mathcal S$, $s' \in \mathcal S$ such that $s \preceq s'$, we have
\begin{equation*}
\begin{aligned}
U^k_t(s) \le U^k_t(s'),\\
L^k_t(s) \le L^k_t(s').
\end{aligned}
\end{equation*}
\label{lem:monoU}
\end{restatable}
\begin{proof}
See Appendix \ref{sec:appendix}.
\end{proof}
We continue with some definitions pertaining to the projection operator $\Pi_M$. A ``$-$'' in the superscript signifies ``the value $s$ is too small'' and the ``$+$'' signifies ``the value of $s$ is too large.''
\begin{defn}
For $t<T$ and $s \in \mathcal S$, let $\mathcal N_t^-(s)$ be a random set representing the iterations for which $s$ was increased by the projection operator. Similarly, let $\mathcal N_t^+(s)$ represent the iterations for which $s$ was decreased:
\begin{align*}
\mathcal N_t^{\Pi-}(s) &= \bigl\{n: s \ne S_t^n \mbox{ and } \widebar{V}_t^{n-1}(s) < \widebar{V}_t^n(s)\bigr\},\\
\mathcal N_t^{\Pi+}(s) &= \bigl\{n: s \ne S_t^n \mbox{ and }\widebar{V}_t^{n-1}(s) > \widebar{V}_t^n(s) \bigr\}.
\end{align*}
\end{defn}
\begin{defn}
For $t < T$ and $s\in\mathcal S$, let $N^{\Pi-}(t,s)$ be the last iteration for which the state $s$ was increased by $\Pi_M$ at time $t$. 
\begin{equation*}
N^{\Pi-}_t(s) = \max \, \mathcal N_t^-(s).
\end{equation*}
Similarly, let
\begin{equation*}
N^{\Pi+}_t(s) = \max \, \mathcal N_t^+(s).
\end{equation*}
Note that $N^{\Pi-}_t(s) = \infty$ if $|\mathcal N_t^-(s)| = \infty$ and $N^{\Pi+}_t(s) = \infty$ if $|\mathcal N_t^+(s)| = \infty$.
\end{defn}
\begin{defn}
Let $N^\Pi$ be large enough so that for iterations $n \ge N^\Pi$, any state increased (decreased) finitely often by the projection operator $\Pi_M$ is no longer affected by $\Pi_M$. In other words, if some state is increased (decreased) by $\Pi_M$ on an iteration after $N^\Pi$, then that state is increased (decreased) by $\Pi_M$ infinitely often. We can write:
\begin{equation*}
\begin{aligned}
N^\Pi = \max\; \Bigl( \bigl \{N^{\Pi-}_t(s)&: t < T,\;s \in \mathcal S,\;N^{\Pi-}_t(s) < \infty \bigr \} \, \cup \\ 
&\bigl \{N^{\Pi+}_t(s): t < T,\;s \in \mathcal S,\;N^{\Pi+}_t(s) < \infty \bigr \} \Bigr) + 1.
\end{aligned}
\end{equation*}
\end{defn}

We now define, for each $t$, two random subsets $\mathcal S_t^-$ and $\mathcal S_t^+$ of the state space $\mathcal S$ where $\mathcal S_t^-$ contains states that are increased by the projection operator $\Pi_M$ finitely often and $\mathcal S_t^+$ contains states that are decreased by the projection operator finitely often. The role that these two sets play in the proof is as follows: 
\begin{itemize}
\item We first show convergence for states that are projected finitely often ($s \in \mathcal S_t^-$ or $s \in \mathcal S_t^+$).
\item Next, relying on the fact that convergence already holds for states that are projected finitely often, we use an induction--like argument to extend the property to states that are projected infinitely often ($s \in \mathcal S \setminus \mathcal S_t^-$ or $s \in \mathcal S \setminus \mathcal S_t^+$). This step requires the definition of a tree structure that arranges the set of states and its partial ordering in an intuitive way.
\end{itemize}

\begin{defn}
For $t < T$, define
\begin{equation*}
\mathcal S_t^- = \bigl\{s \in \mathcal S: N^{\Pi-}_t(s) < \infty \bigr\},
\end{equation*}
and
\begin{equation*}
\mathcal S_t^+ = \bigl\{s \in \mathcal S: N^{\Pi+}_t(s) < \infty \bigr\},
\end{equation*}
to be random subsets of states that are projected finitely often.
\end{defn}


\begin{restatable}{lem}{notempty}
The random sets $\mathcal S_t^-$ and $\mathcal S_t^+$ are almost surely nonempty.
\label{lem:notempty}
\end{restatable}
\begin{proof}
See Appendix \ref{sec:appendix}.
\end{proof}

We now provide several remarks regarding the projection operator $\Pi_M$. The value of a state $s$ can only be increased by $\Pi_M$ if we visit a ``smaller'' state, i.e., $S_t^n \preceq s$. This statement is obvious from the second condition of (\ref{genprojection]}). Similarly, the value of the state can only be decreased by $\Pi_M$ if the visited state is ``larger,'' i.e., $S_t^n \succeq s$. Intuitively, it can be useful to imagine that, in some sense, the values of states can be ``pushed up'' from the ``left'' and ``pushed down'' from the ``right.''

Finally, due to our assumption that $\mathcal S$ is only a partial ordering, the update process (from $\Pi_M$) becomes more difficult to analyze than in the total ordering case. To facilitate the analysis of the process, we introduce the notions of \emph{lower (upper) immediate neighbors} and \emph{lower (upper) update trees}.
\begin{defn}
For $s = (m,i)\in \mathcal S$, we define the set of \emph{lower immediate neighbors} $\mathcal S_L(s)$ in the following way:
\begin{equation*}
\begin{aligned}
\mathcal S_L(s) = \bigl\{s' \in \mathcal S : \quad &s' \preceq s,\\
& s' \ne s,\\
&\nexists \, s'' \in \mathcal S,\,s'' \ne s,\,s'' \ne s',\,s' \preceq s'' \preceq s \bigr\}.
\end{aligned}
\end{equation*}
In other words, there does not exist $s''$ in between $s'$ and $s$. The set of \emph{upper immediate neighbors} $\mathcal S_U(s)$ is defined in a similar way:
\begin{equation*}
\begin{aligned}
\mathcal S_U(s) = \{s' \in \mathcal S:\quad &s' \succeq s,\\
& s' \ne s,\\
&\nexists \, s'' \in \mathcal S,\,s'' \ne s,\,s'' \ne s',\,s' \succeq s'' \succeq m\}.
\end{aligned}
\end{equation*}
\label{defn:lower}
\end{defn}
The intuition for the next lemma is that if some state $s$ is increased by $\Pi_M$, then it must have been caused by visiting a lower state. In particular, \emph{either} the visited state was one of the lower immediate neighbors \emph{or} one of the lower immediate neighbors was also increased by $\Pi_M$. In either case, one of the lower immediate neighbors has the same value as $s$. This lemma is crucial later in the proof.
\begin{restatable}{lem}{lowerimmediate}
Suppose the value of $s$ is increased by $\Pi_M$ on some iteration $n$:
\begin{equation*}
s \ne S_t^n \, \mbox{ and } \, \widebar{V}_t^{n-1}(s) < \widebar{V}_t^n(s).
\end{equation*}
Then, there exists another state $s' \in \mathcal S_L(s)$ (in the set of lower immediate neighbors) whose value is equal to the newly updated value:
\begin{equation*}
\widebar{V}_t^n(s') = \widebar{V}_t^{n}(s).
\end{equation*}
\label{lem:lowerimmediate}
\end{restatable}
\begin{proof}
See Appendix \ref{sec:appendix}.
\end{proof}
\begin{defn}
Consider some $\omega \in \Omega$. Let $s \in \mathcal S \setminus \mathcal S_t^-$\hspace{-0.1em}, meaning that $s$ is increased by $\Pi_M$ infinitely often: $|\mathcal N_t^-(s)|=\infty$. A \emph{lower update tree} $T_t^-(s)$ is an organization of the states in the set $\mathcal L =\{s' \in \mathcal S: s' \preceq s\}$, where the value of each node is an element of $\mathcal L$. The tree $T_t^-(s)$ is constructed according to the following rules. 
\begin{enumerate}[label=(\roman*),labelindent=1in]
\item The root node of $T_t^-(s)$ has value $s$.
\item Consider an arbitrary node $j$ with value $s_j$.
\begin{enumerate}
\item If $s_j \in \mathcal S \setminus \mathcal S_t^-$, then for each $s_{jc} \in \mathcal S_L(s_j)$, add a child node with value $s_{jc}$ to the node $j$.
\item If $s_j \in \mathcal S_t^-$, then $j$ is a leaf node (it does not have any child nodes).
\end{enumerate}
\end{enumerate}
The tree $T_t^-(s)$ is unique and can easily be built by starting with the root node and successively applying the rules. The \emph{upper update tree} $T_t^+(s)$ is defined in a completely analogous way.
\end{defn}
Note that the lower update tree is random and we now argue that for each $\omega$, it is well--defined. We observe that it cannot be the case for some state $s$ to be an element of $\mathcal S \setminus \mathcal S_t^-$ while $\mathcal S_L(s') = \{\}$, as for it to be increased infinitely often, there must exist at least one ``lower'' state whose observations cause the monotonicity violations. Using this fact along with the finiteness of $\mathcal S$ and Lemma \ref{lem:notempty}, which states that $\mathcal S_t^-$ is nonempty, it is clear that all paths down the tree reach a leaf node (i.e., an element of $\mathcal S_t^-$). The reason for discontinuing the tree at states in $\mathcal S_t^-$ is that our convergence proof employs an induction--like argument up the tree, starting with states in $\mathcal S_t^-$. Lastly, we remark that it is possible for multiple nodes to have the same value. As an illustrative example, consider the case with $\mathcal S = \{0,1,2\}^2$ with $\preceq$ being the componentwise inequality. Assume that for a particular $\omega \in \Omega$, $s = (s_x,s_y) \in \mathcal S_t^-$ if and only if $s_x=0$ or $s_y=0$ (lower boundary of the square). Figure \ref{trees} shows the realization of the lower update tree at evaluated at the state $(2,2)$.
\begin{figure}[h]
	\begin{center}
	\includegraphics[scale=.5]{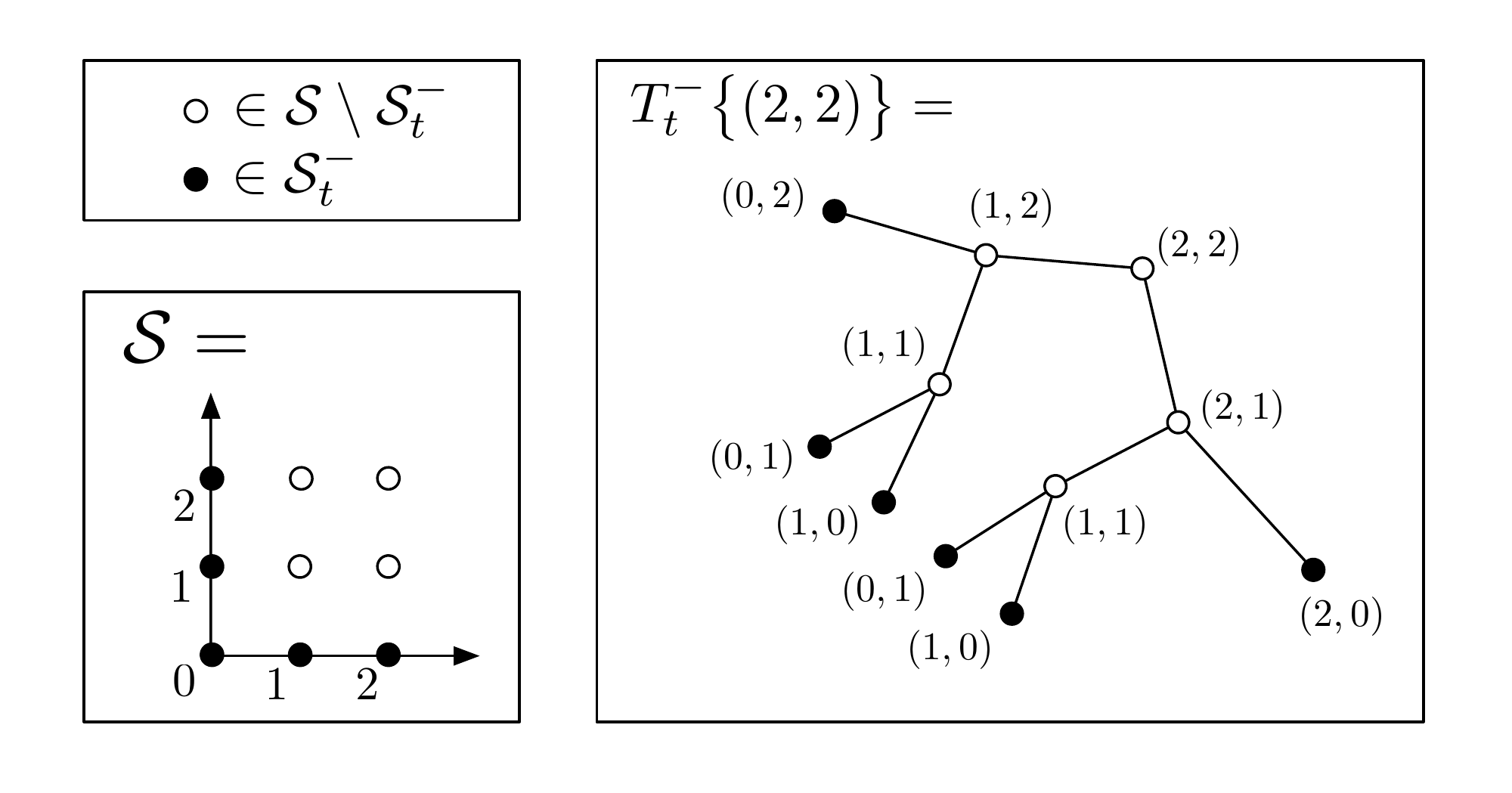}\\
	\end{center}
	\caption{Illustration of the Lower Update Tree}
	\label{trees}
\end{figure}

The next lemma is a useful technical result used in the convergence proof.
\begin{restatable}{lem}{alphaprod}
For any $s \in \mathcal S$,
\begin{equation*}
\displaystyle
\lim_{m\rightarrow \infty} \left[ \, \prod_{n=1}^m \bigl(1-\alpha_t^n(s)\bigr) \, \right]= 0 \quad a.s.
\end{equation*}
\label{lem:alphaprod}
\end{restatable}
\begin{proof}
See Appendix \ref{sec:appendix}.
\end{proof}
With these preliminaries in mind (other elements will be defined as they arise), we begin the convergence analysis.
\begin{proof}[Proof of Theorem]
As previously mentioned, to show that the sequence $\widebar{V}_t^n(s)$ (almost surely) converges to $V^*_t(s)$ for each $t$ and $s$, we need to argue that $\widebar{V}_t^n(s)$ eventually enters every rectangle (or ``interval,'' when we discuss a specific component of the vector $\widebar{V}^n$), defined by the sequence  $L_t^k(s)$ and $U_t^k(s)$. Recall that the estimates of the value function produced by the algorithm are indexed by $n$ and the bounding rectangles are indexed by $k$. Hence, we aim to show that for each $k$, we have that for $n$ sufficiently large, it is true that $\forall \, s\in \mathcal S$,
\begin{equation}
L_t^k(s) \le \widebar{V}_t^n(s) \le U_t^k(s).
\label{boundrect}
\end{equation}
Following this step, an application of (\ref{ULconv}) in Lemma \ref{lem:1} completes the proof. We show the second inequality of (\ref{boundrect}) and remark that the first can be shown in a completely symmetric way. The goal is then to show that $\exists \, N_t^k < \infty$ $a.s.$ such that $\forall \, n \ge N_t^k$ and $\forall \, s\in \mathcal S$,
\begin{equation}
\widebar{V}_t^n(s) \le U_t^k(s).
\label{toprove}
\end{equation}
Choose $\omega \in \Omega$. For ease of presentation, the dependence of the random variables on $\omega$ is omitted. We use backward induction on $t$ to show this result, which is the same technique used in \cite{Nascimento2009a}. The inductive step is broken up into two cases, $s \in \mathcal S_t^-$ and $s \in \mathcal S \setminus \mathcal S_t^-$.

\subsubsection*{Base case, $t=T$.}
Since for all $s \in \mathcal S$, $k$, and $n$, we have that (by definition) $\widebar{V}_T^n(s) = U_T^k(s) = 0$, we can arbitrarily select $N_T^k$. Suppose that for each $k$, we choose $N_T^k = N^\Pi$, allowing us to use the property of $N^\Pi$ that if $s \in \mathcal S_t^-$, then the estimate of the value at $s$ is no longer affected by $\Pi_M$ on iterations $n \ge N^\Pi$.
\subsubsection*{Induction hypothesis, $t+1$} Assume for $t+1 \le T$ that $\forall \, k \ge 0$,
$\exists \, N_{t+1}^k < \infty$ such that $N_{t+1}^k \ge N^{\Pi}$ and $\forall \, n \ge N_{t+1}^k$, we have that $\forall \, s\in \mathcal S$,
\begin{equation*}
\widebar{V}_{t+1}^n(s) \le U_{t+1}^k(s).
\end{equation*}
%


\subsubsection*{Inductive step from $t+1$ to $t$} The remainder of the proof concerns this inductive step and is broken up into two cases, $s \in \mathcal S_t^-$ and $s \in \mathcal S \setminus \mathcal S_t^-$. For each $s$, we show the existence of a state dependent iteration $\tilde{N}_t^k(s) \ge N^\Pi$, such that for $n \ge \tilde{N}_t^k(s)$, (\ref{toprove}) holds. The \emph{state independent} iteration $N_t^k$ is then taken to be the maximum of $\tilde{N}_t^k(s)$ over $s$.

\subsubsection*{Case 1: $s \in \mathcal S_t^-$} To prove this case, we induct forwards on $k$. Note that we are still inducting backwards on $t$, so the induction hypothesis for $t+1$ still holds. The inductive step is proved in essentially the same manner as Theorem 2 of \cite{Tsitsiklis1994a}.

\subsubsection*{Base case, $k=0$ (within induction on $t$)}
By Assumption \ref{Cbounded} and (\ref{Udefn1}), we have that $U_t^0(s) \ge V_\textnormal{max}$.
But by Assumption \ref{boundedobs}, the updating equation (Step 2b of Figure \ref{alg:vfa2}), and the initialization of $\widebar{V}_t^0(s) \in [0,\Vmax]$, we can easily see that $\widebar{V}_t^n(s) \in [0,V_\textnormal{max}]$ for any $n$ and $s$. Therefore, $\widebar{V}_t^n(s) \le U_t^0(s)$, for any $n$ and $s$, so we can choose $N_t^{0}$ arbitrarily. Let us choose $\tilde{N}_t^{0}(s) = N_{t+1}^0$, and since $N_{t+1}^0$ came from the induction hypothesis for $t+1$, it is also true that $\tilde{N}_t^{0}(s) \ge N^\Pi$.

\subsubsection*{Induction hypothesis, $k$ (within induction on $t$)} Assume for $k \ge 0$ that $\exists \, \tilde{N}_{t}^{k}(s) < \infty$ such that $\tilde{N}_{t}^{k}(s) \ge N_{t+1}^k \ge N^{\Pi}$ and $\forall \, n \ge \tilde{N}_{t}^{k}(s)$, we have 
\begin{equation*}
\widebar{V}_{t}^n(s) \le U_{t}^k(s).
\end{equation*}

Before we begin the inductive step from $k$ to $k+1$, we define some additional sequences and state a few useful lemmas.
\begin{defn}
The \emph{positive incurred noise}, since a starting iteration $m$, is represented by the sequence $W_t^{n,m}(s)$. For $s \in \mathcal S$, it is defined as follows:
\begin{equation*}
\begin{aligned}
W_t^{m,m}(s) &= 0,\\
W_t^{n+1,m}(s) &= \Bigl[\bigl(1-\alpha_t^n(s)\bigr)  \, W_t^{n,m}(s)+\alpha_t^n(s) \, w_t^{n+1}(s)\Bigr]^+ \quad \textnormal{for } n \ge m.
\end{aligned}
\end{equation*}
\end{defn}
The term $W_t^{n+1,m}(s)$ is only updated from $W_t^{n,m}(s)$ when $s=S_t^n$, i.e., on iterations where the state is visited by the algorithm, due to the fact that the stepsize $\alpha_t^n(s) = 0$ whenever $s \ne S_t^n$.
\begin{restatable}{lem}{Wzero}
For any starting iteration $m \ge 0$ and any state $s \in \mathcal S$, under Assumptions \ref{boundedobs}, \ref{wass}, and \ref{Stepassumption}, $W_t^{n,m}(s)$ asymptotically vanishes:
\begin{equation*}
\lim_{n\rightarrow \infty} W_t^{n,m}(s) = 0 \quad a.s.
\end{equation*}
\label{lem:W}
\end{restatable}
\begin{proof}
The proof is analogous to that of Lemma 6.2 in \cite{Nascimento2009a}, which uses a martingale convergence argument.
\end{proof}
To reemphasize the presence of $\omega$, we note that the following definition and the subsequent lemma both use the realization $\tilde{N}_t^{k}(s)(\omega)$ from the $\omega$ chosen at the beginning of the proof.
\begin{defn}
The other auxiliary sequence that we need is $X_t^{n}(s)$ which applies the smoothing step to $(HU^k)_t(s)$. For any state $s \in \mathcal S$, let
\begin{equation*}
\begin{aligned}
X_t^{\tilde{N}_t^{k}(s)}(s) &= U_t^k(s),\\
X_t^{n+1}(s) &= \bigl(1-\alpha_t^n(s)\bigr) \, X_t^{n}(s) + \alpha_t^n(s) \, \bigl(HU^k\bigr)_t(s) \quad \textnormal{for } n \ge \tilde{N}_t^{k}(s).
\end{aligned}
\end{equation*}
\end{defn}
\begin{restatable}{lem}{XW}
For $n \ge \tilde{N}_t^{k}(s)$ and state $s \in \mathcal S_t^-$\hspace{-0.1em},
\begin{equation*}
\widebar{V}_t^n(s) \le X_t^{n}(s) + W_t^{n,\tilde{N}_t^{k}(s)}(s).
\end{equation*}
\label{lem:XW}
\end{restatable}
\begin{proof}
See Appendix \ref{sec:appendix}.
\end{proof}
\subsubsection*{Inductive step from $k$ to $k+1$} If $U_t^k(s)=\bigl(HU^k\bigr)_t(s)$, then by Lemma \ref{lem:1}, we see that $U_t^k(s)=U_t^{k+1}(s)$ so $\widebar{V}_t^n \le U_t^k(s) \le U_t^{k+1}(s)$ for any $n \ge \tilde{N}_t^{k}(s)$ and the proof is complete. Since we know that $(HU^k)_t(s) \le U_t^k(s)$ by Lemma \ref{lem:1}, we can now assume that $s \in \mathcal K$, where
\begin{equation*}
\mathcal K = \Bigl\{s' \in \mathcal S: \bigl(HU^k\bigr)_t(s') < U_t^k(s') \Bigr\}.
\end{equation*}
In this case, we can define
\begin{equation*}
\displaystyle
\delta_t^k = \min_{s\in\mathcal S_t^- \cap \, \mathcal K}\left(\frac{U_t^k(s)-\bigl(HU^k\bigr)_t(s)}{4}\right) > 0.
\end{equation*}
Choose $\tilde{N}_t^{k+1}(s) \ge \tilde{N}_t^{k}(s)$ such that
\begin{equation*}
\prod_{n=\tilde{N}_t^{k}(s)}^{\tilde{N}_t^{k+1}(s)-1} \bigl(1-\alpha_t^n(s)\bigr) \le \frac 14
\end{equation*}
and for all $n \ge \tilde{N}_t^{k+1}(s)$,
\begin{equation*}
W_t^{n,\tilde{N}_t^{k}(s)}(s) \le \delta_t^k.
\end{equation*}
$\tilde{N}_t^{k+1}(s)$ clearly exists because both sequences converge to zero, by Lemma \ref{lem:alphaprod} and Lemma \ref{lem:W}. Recursively using the definition of $X_t^n(s)$, we get that
\begin{equation*}
X_t^n(s) = \beta_t^n(s) \, U_t^k(s) + \bigl(1-\beta_t^n(s) \bigr) \, \bigl(HU^k\bigr)_t(s),
\end{equation*}
where $\beta_t^n(s) = \prod_{l=\tilde{N}_t^{k}(s)}^{n-1}(1-\alpha_t^l(s))$. Notice that for $n \ge \tilde{N}_t^{k+1}(s)$, we know that $\beta_t^n(s) \le \frac 14$, so we can write
\begin{align}
X_t^n(s) &= \beta_t^n(s) \, U_t^k(s) + \bigl(1-\beta_t^n(s)\bigr) \,\bigl(HU^k\bigr)_t(s)\nonumber\\
&= \beta_t^n(s) \, \Bigl[ U_t^k(s) - \bigl(HU^k\bigr)_t(s) \Bigr] + \bigl(HU^k\bigr)_t(s)\nonumber\\
&\le \frac 14 \, U_t^k(s) + \frac 34 \, \bigl(HU^k\big)_t(s)\nonumber\\
&=\frac 12 \, \Bigl[  U_t^k(s) + \bigl(HU^k\bigr)_t(s)\Bigr]-\frac 14 \, \Bigl[U_t^k(s) - \bigl(HU^k\bigr)_t(s)\Bigr]\nonumber\\
&\le U_t^{k+1}(s) - \delta_t^k. \label{eq:Xdelta}
\end{align}
We can apply Lemma \ref{lem:XW} and (\ref{eq:Xdelta}) to get
\begin{equation*}
\begin{aligned}
\widebar{V}_t^n(s) &\le X_t^n(s) + W_t^{n,\tilde{N}_t^{k}(s)}(s)\\
&\le (U_t^{k+1}(s) - \delta_t^k)+\delta_t^k\\
&= U_t^{k+1}(s),
\end{aligned}
\end{equation*}
for all $n \ge \tilde{N}_t^{k+1}(s)$.
Thus, the inductive step from $k$ to $k+1$ is complete.
\subsubsection*{Case 2: $s \in \mathcal S \setminus \mathcal S_t^-$} Recall that we are still in the inductive step from $t+1$ to $t$ (where the hypothesis was the existence of $N_{t+1}^k$). As previously mentioned, the proof for this case relies on an induction--like argument over the tree $T_t^-(s)$.
The following lemma is the core of our argument, and the proof is provided below.
\begin{restatable}{lem}{treeinduction}
Consider some $k \ge 0$ and a node $j$ of $T_t^-(s)$ with value $s_j \in \mathcal S \setminus \mathcal S_t^-$ and let the $C_j \ge 1$ child nodes of $j$ be denoted by the set $\{s_{j,1},s_{j,2},\ldots,s_{j,C_j}\}$. Suppose that for each $s_{j,c}$ where $1 \le c \le C_j$, we have that $\exists \, \tilde{N}_t^{k}(s_{j,c}) < \infty$ such that $\forall \, n \ge \tilde{N}_t^{k}(s_{j,c})$,
\begin{equation}
\widebar{V}_t^n(s_{j,c}) \le U_t^k(s_{j,c}).
\label{cond:1}
\end{equation}
Then, $\exists \, \tilde{N}_t^{k}(s_{j}) < \infty$ such that $\forall \, n \ge \tilde{N}_t^{k}(s_{j})$,
\begin{equation*}
\widebar{V}_t^n(s_j) \le U_t^k(s_j).
\end{equation*}
\label{treeinductionlem}
\end{restatable}

\begin{proof} First, note that by the induction hypothesis, part $(ii)$ of Lemma $\ref{Hprops}$, and Lemma \ref{lem:1}, we have the inequality
\begin{equation}
\bigl(H\widebar{V}^n\bigr)_t(s) \le \bigl(HU^k\bigr)_t(s) \le U_t^k(s).
\label{ineq:1_2}
\end{equation}
We break the proof into several steps.\\
\emph{Step 1.} Let us consider the iteration $\tilde{N}$ defined by
\begin{equation*}
\displaystyle
\tilde{N} = \min \left(n \in \mathcal N_t^{\Pi-}(s_j): n \ge \max_{c} \tilde{N}_t^{k}(s_{j,c}) \right),
\end{equation*}
which exists because $s_j \in \mathcal S \setminus \mathcal S_t^-$ and is increased infinitely often. This means that $\Pi_M$ increased the value of state $s_j$ on iteration $\tilde{N}$. As the first step, we show that $\forall \, n \ge \tilde{N}$,
\begin{equation}
\widebar{V}_t^n(s_j) \le U_t^k(s_j)+W_t^{n,\tilde{N}}(s_j),
\label{step1}
\end{equation}
using an induction argument.\\
\emph{Base case, $n = \tilde{N}$.} 
Using Lemma \ref{lem:lowerimmediate}, we know that for some $c \in \{1,2,\ldots,C_j\}$, we have:
\begin{equation*}
\begin{aligned}
\widebar{V}_t^n(s_j) = \widebar{V}_t^n(s_{j,c}) &\le U_t^k(s_{j,c})\\
&\le U_t^k(s_j)+W_t^{n,\tilde{N}}(s_j).
\end{aligned}
\end{equation*}
The fact that $\tilde{N} \ge \tilde{N}_t^{k}(s_{j,c})$ for every $c$ justifies the first inequality and the second inequality above follows from the monotonicity within $U^k$ (see Lemma \ref{lem:monoU}) and the fact that $W_t^{\tilde{N},\tilde{N}}(s_j)=0$.\vspace{0.2em}\\
\emph{Induction hypothesis, $n$.} Suppose (\ref{step1}) is true for $n$ where $n \ge \tilde{N}$.\vspace{0.2em}\\
\emph{Inductive step from $n$ to $n+1$.} Consider the following two cases:
\begin{enumerate}[label=(\Roman*),labelindent=1in]
\item Suppose $n+1 \in \mathcal N_t^{\Pi-}(s_j)$. The proof for this is exactly the same as for the base case, except we use $W_t^{n+1,\tilde{N}}(s_j) \ge 0$ to show the inequality. Again, this step depends heavily on Lemma \ref{lem:lowerimmediate} and on the fact that every child node represents a state that satisfies (\ref{cond:1}).
\item Suppose $n+1 \not \in \mathcal N_t^{\Pi-}(s_j)$. There are again two cases to consider:
\begin{enumerate}[label=(\Alph*),labelindent=1in]
\item Suppose $S_t^{n+1} = s_j$. Then,
\begin{align}
\widebar{V}_t^{n+1}(s_j) &= z_t^{n+1}(s_j)\nonumber\\
&=\bigl(1-\alpha_t^{n+1}(s_j)\bigr) \, \widebar{V}_t^n(s_j) + \alpha_t^{n+1}(s_j) \, \hat{v}_t^{n+1}(s_j) \nonumber\\
&\le\begin{aligned}[t]\bigl(1 &-\alpha_t^{n+1}(s_j)\bigr) \, \Bigl(U_t^k(s_j) + W_t^{n,\tilde{N}}(s_j)\Bigr) \\
&+ \alpha_t^{n+1}(s_j)\Bigl[ \bigl(H\widebar{V}^n\bigr)_t(s_j) +w_t^{n+1}(s_j) \Bigr]\end{aligned} \nonumber\\
&\le U_t^k(s_j)+W_t^{n+1,\tilde{N}}(s_j),\nonumber
\end{align}
where the first inequality follows from the induction hypothesis for $n$ and the second inequality follows by (\ref{ineq:1_2}).
\item Suppose $S_t^{n+1} \ne s_j$. This means the stepsize $\alpha_t^{n+1}(s_j) = 0$, which in turn implies the noise sequence remains unchanged: $W_t^{n+1,\tilde{N}}(s_j) = W_t^{n,\tilde{N}}(s_j)$. Because the value of $s_j$ is not increased at $n+1$,
\begin{align*}
\widebar{V}_t^{n+1}(s_j) &\le \widebar{V}_t^n(s_j)\\
&\le U_t^k(s_j)+W_t^{n,\tilde{N}}(s_j)\\
&= U_t^k(s_j)+W_t^{n+1,\tilde{N}}(s_j).
\end{align*}
\end{enumerate}
\end{enumerate}
\emph{Step 2.} By Lemma \ref{lem:W}, we know that $W_t^{n,\tilde{N}}(s_j) \rightarrow 0$ and thus, $\exists \, \tilde{N}_t^{k,\epsilon}(s_j) < \infty$ such that $\forall \, n \ge \tilde{N}_t^{k,\epsilon}(s_j)$,
\begin{equation}
\widebar{V}_t^n(s_j) \le U_t^k(s_j) + \epsilon.
\label{Neps}
\end{equation}
Let $\epsilon = U_t^k(s_j)-V_t^*(s_j) > 0$. Since $U_t^k(s_j) \searrow V_t^*(s_j)$, we also have that $\exists \, k' > k$ such that,
\begin{equation*}
U_t^{k'}(s_j)-V_t^*(s_j) < \epsilon/2.
\end{equation*}
Combining with the definition of $\epsilon$, we have
\begin{equation*}
U_t^k(s_j)-U_t^{k'}(s_j) > \epsilon/2.
\end{equation*}
Applying (\ref{Neps}), we know that $\exists \, \tilde{N}_t^{k'\!,\epsilon/2}(s_j) < \infty$ such that $\forall \, n \ge \tilde{N}_t^{k'\!,\epsilon/2}(s_j)$,
\begin{align}
\widebar{V}_t^n(s_j) &\le U_t^{k'}(s_j) + \epsilon/2\nonumber \\
&\le U_t^k(s_j)-\epsilon/2+\epsilon/2 \nonumber \\
&\le U_t^k(s_j). \nonumber
\end{align}
Therefore, we can choose $\tilde{N}_t^{k}(s_{j}) = N_t^{k'\!,\epsilon/2}(s_j)$ to conclude the proof.
\end{proof}
We now present a simple algorithm that incorporates the use of Lemma \ref{treeinductionlem} in order to obtain $\tilde{N}_t^k(s)$ when $s \in \mathcal S \setminus \mathcal S_t^-$. Denote the \emph{height} (longest path from root to leaf) of $T_t^-(s)$ by $H^-_t(s)$. The \emph{depth} of a node $j$ is the length of the path from the root node to $j$.
\begin{description}[leftmargin=4.2em,style=nextline]
\item[Step 0.] Set $h = H_t^-(s)-1$. The child nodes of any node of depth $h$ are leaf nodes that represent states in $\mathcal S_t^-$ --- the conditions of Lemma \ref{treeinductionlem} are thus satisfied by \emph{Case 1}.
\item[Step 1.] Consider all nodes $j_h$ (with value $s_h$) of depth $h$ in $T_t^-(s)$. An application of Lemma \ref{treeinductionlem} results in $\tilde{N}_t^k(s_h)$ such that $\widebar{V}_t^n(s_h) \le U_t^k(s_h)$ for all $n \ge \tilde{N}_t^k(s_h)$.
\item[Step 2.] If $h = 0$, we are done. Otherwise, decrement $h$ and note that once again, the conditions of Lemma \ref{treeinductionlem} are satisfied for any node of depth $h$. Return to \textbf{Step 1}.
\end{description}
At the completion of this algorithm, we have the desired $\tilde{N}_t^k(s)$ for $s \in \mathcal S \setminus \mathcal S_t^-$. By its construction, we see that $\tilde{N}_t^k(s) \ge N^\Pi$, and the final step of the inductive step from $t+1$ to $t$ is to define $N_t^k = \max_{s \in \mathcal S} \tilde{N}_t^k(s)$. The proof is complete.
\end{proof}

\section{Numerical Results}
\label{sec:numerical}
Theoretically, we have shown that Monotone--ADP is an asymptotically convergent algorithm. In this section, we discuss its empirical performance. There are two main questions that we aim to answer using numerical examples:
\begin{enumerate}
\item How much does the monotonicity preservation operator, $\Pi_M$, increase the rate of convergence, compared to other popular approximate dynamic programming or reinforcement learning algorithms?
\item How much computation can we save by solving a problem to near--optimality using Monotone--ADP compared to solving it to full optimality using backward dynamic programming?
\end{enumerate}

To provide insight into these questions, we compare Monotone--ADP against four ADP algorithms from the literature (kernel--based reinforcement learning, approximate policy iteration with polynomial basis functions, asynchronous value iteration, and $Q$--learning) across a set of fully benchmarked problems from three distinct applications (optimal stopping, energy allocation/storage, and glycemic control for diabetes patients). Throughout our numerical work, we assume that the model is known and thus compute the observations $\hat{v}_t^n$ using (\ref{eq:vhat1}). For Monotone--ADP, asynchronous value iteration, and $Q$--learning, the sampling policy we use is the $\varepsilon$--greedy exploration policy (explore with probability $\varepsilon$, follow current policy otherwise) --- we found that across our range of problems, choosing a relatively large $\varepsilon$ (e.g., $\varepsilon \in  [0.5,1]$) generally produced the best results. These same three algorithms also require the use of a stepsize, and in all cases we use the adaptive stepsize derived in \cite{George2006}. Moreover, note that since we are interested in comparing the performance of approximation algorithms against optimal benchmarks, it is necessary to sacrifice a bit of realism and discretize the distribution of $W_{t+1}$ so that an exact optimal solution can be obtained using backward dynamic programming (BDP). However, this certainly is not necessary in practice; Monotone--ADP handles continuous distributions of $W_{t+1}$ perfectly well, especially if (\ref{eq:vhat2}) and (\ref{eq:vhat3}) are used.

Before moving on to the applications, let us briefly introduce each of the algorithms that we use in the numerical work. For succinctness, we omit step--by--step descriptions of the algorithms and instead point the reader to external references.

\subsection*{Kernel--Based Reinforcement Learning (KBRL)} Kernel regression, which dates back to \cite{Nadaraya1964}, is arguably the most widely used nonparametric technique in statistics. \cite{Ormoneit2002} develops this powerful idea in the context of approximating value functions using an approximate value iteration algorithm --- essentially, the Bellman operator is replaced with a kernel--based approximate Bellman operator. For our finite--horizon case, the algorithm works backwards from $t=T-1$ to produce kernel--based approximations of the value function at each $t$, using a fixed--size batch of observations each time. The most attractive feature of this algorithm is that no structure whatsoever needs to be known about the value function, and in general, a decent policy can be found. One major drawback, however, is the fact that KBRL cannot be implemented in a recursive way; the number of observations used per time--period needs to be specified beforehand, and if the resulting policy is poor, then KBRL needs to completely start over with a larger number of observations. A second major drawback is \emph{bandwidth selection} --- in our empirical work, the majority of the time spent with this algorithm was focused on tuning the bandwidth, with guidance from various ``rules of thumb,'' such as the one found in \cite{Scott1992}. Our implementation uses the popular \emph{Gaussian kernel}, given by
\[
K(s,s')  = \frac{1}{2\pi} \, \exp\left(\frac{\|s -s' \|_2^2}{2b}\right),
\]
where $s$ and $s'$ are two states and $b$ the bandwidth (tuned separately for each problem). The detailed description of the algorithm can be found in the original paper, \cite{Ormoneit2002}.
\subsection*{Approximate Policy Iteration with Polynomial Basis Functions (API)} Based on the exact policy iteration algorithm (which is well--known to possess finite time convergence), certain forms of approximate policy iteration has been applied successfully in a number of real applications (\cite{Bertsekas2011} provides an excellent survey). The basis functions that we employ are all possible monomials up to degree 2 over all dimensions of the state variable (i.e., we allow interactions between every pair of dimensions). Due to the fact that there are typically a small number of basis functions, policies produced by API are very fast to evaluate, regardless of the size of the state space. The exact implementation of our API algorithm, specialized for finite--horizon problems, is given in \cite[Section 10.5]{Powell2011}. One drawback of API that we observed is an inadequate exploration of the state space for certain problems; even if the initial state $S_0$ is fully randomized, the coverage of the state space in a much later time period, $t' \gg 0$, may be sparse. To combat this issue, we introduced artificial exploration to time periods $t'$ with poor coverage by adding simulations that started at $t'$ rather than 0. The second drawback is that we observed what we believe to be \emph{policy oscillations}, where the policies go from good to poor in an oscillating manner. This issue is not well understood in the literature but is discussed briefly in \cite{Bertsekas2011}.
\subsection*{Asynchronous Value Iteration (AVI)} This algorithm is an elementary algorithm of approximate dynamic programming/reinforcement learning. As the basis for Monotone--ADP, we include it in our comparisons to illustrate the utility of $\Pi_M$. In fact, AVI can be recovered from Monotone--ADP by simply removing the monotonicity preservation step; it is a lookup table based algorithm where only one state (per time period) is updated at every iteration. More details can be found in \cite[Section 4.5]{Sutton1998}, \cite[Section 2.6]{Bertsekas2007}, or \cite[Section 10.2]{Powell2011}.
\subsection*{$Q$--learning (QL)} Due to \cite{Watkins1992}, this reinforcement--learning algorithm estimates the values of state--action pairs, rather than just the state, in order to handle the model--free situation. Its crucial drawback however is it can only be applied in problems with very small action spaces --- for this reason, we only show results for $Q$--learning in the context of our optimal stopping application, where the size of the action space is two. In order to make the comparison between $Q$--learning and the algorithms as fair as possible, we improve the performance of $Q$--learning by taking advantage of our \emph{known model} and compute the expectation at every iteration (as we do in AVI). This slight change from the original formulation, of course, does not alter its convergence properties.
\subsection*{Backward Dynamic Programming (BDP)} This is the well--known, standard procedure for solving finite--horizon MDPs. Using significant computational resources, we employ BDP to obtain the optimal benchmarks in each of the example applications in this section. A description of this algorithm, which is also known as \emph{backward induction} or \emph{finite--horizon value iteration}, can be found in \cite{Puterman}.

\subsection{Evaluating Policies}
We follow the method in \cite{Powell2011} for evaluating the policies generated by the algorithms. By the principle of dynamic programming, for particular value functions $V \in \mathbb R^D$, the decision function at time $t$, $A_t:\mathcal S \rightarrow \mathcal A$ can be written as
\begin{equation*}
A_t(s) = \argmax_{a \in \mathcal A} \; \Bigl[C_t(s,a) + \mathbf{E}\bigl[ V_{t+1}(S_{t+1}) \,|\, S_t=s,\,a_t=a \bigr]\Bigr].
\end{equation*}
To evaluate the \emph{value of a policy} (i.e., the expected contribution over the time horizon) using simulation, we take a test set of $L=1000$ sample paths, denoted $\hat{\Omega}$, compute the contribution for each $\omega \in \hat{\Omega}$ and take the empirical mean:
\begin{equation*}
F(V,\omega) = \sum_{t=0}^T C_t \bigl(S_t(\omega), A_t(S_t(\omega)) \bigr) \quad \mbox{and} \quad \bar{F}(V) = L^{-1} \sum_{\omega \in \hat{\Omega}} F(V,\omega).
\end{equation*}
For each problem instance, we compute the optimal policy using backward dynamic programming. We then compare the performance of approximate policies generated by each of the above approximation algorithms to that of the optimal policy (given by $V_0(S_0)$), as a percentage of optimality.


We remark that although a complete numerical study is not the main focus of this paper, the results below do indeed show that Monotone--ADP provides benefits in each of these nontrivial problem settings. 

\subsection{Regenerative Optimal Stopping}
We now present a classical example application from the fields of operations research, mathematics, and economics: regenerative optimal stopping (also known as \emph{optimal replacement} or \emph{optimal maintenance}; see \cite{Pierskalla1976} for a review). The optimal stopping model described in this paper is inspired by that of \cite{Feldstein1974}, \cite{Rust1987}, and \cite{Kurt2010a} and is easily generalizable to higher dimensions. We consider the decision problem of a firm that holds a depreciating asset that it needs to either sell or replace (known formally as \emph{replacement investment}), with the depreciation rate determined by various other economic factors (giving a multidimensional state variable). The competing forces can be summarized to be the revenue generated from the asset (i.e., production of goods, tax breaks) versus the cost of replacement and the financial penalty when the asset's value reaches zero.

Consider a depreciating asset whose value over discrete time indices $t \in \{ 0,1,\ldots,T\}$ is given by a process $\{X_t\}_{t=0}^T$ where $X_t \in \mathcal X = \{0,1,\ldots,X_\textnormal{max}\}$. Let $\{Y_t\}_{t=0}^T$ with $Y_t = (Y_t^i)_{i=1}^{n-1} \in \mathcal Y$ describe external economic factors that affect the depreciation process of the asset. Each component $Y_t^i \in \{0,1,\ldots,Y^i_\textnormal{max}\}$ contributes to the overall depreciation of the asset. The asset's value either remains the same or decreases during each time period $t$. Assume that for each $i$, higher values of the factor $Y_t^i$ correspond to more positive influences on the value of the asset. In other words, the probability of its value depreciating between time $t$ and $t+1$ increases as $Y_t^i$ decreases.

Let $S_t = (X_t,Y_t) \in \mathcal S = \mathcal X \times \mathcal Y$ be the $n$--dimensional state variable. When $X_t > 0$, we earn a revenue of $P$ for some $P>0$, and when the asset becomes worthless (i.e., when $X_t=0$), we suffer a penalty of $-F$ for some $F>0$. At each stage, we can either choose to replace the asset by taking action $a_t=1$ for some cost $r(X_t,Y_t)$, which is nonincreasing in $X_t$ and $Y_t$, or do nothing by taking action $a_t=0$ (therefore, $\mathcal A = \{0,1\}$). Note that when the asset becomes worthless, we are forced to pay the penalty $F$ in addition to the replacement cost $r(X_t,Y_t)$. Therefore, we can specify the following contribution function:
\begin{equation*}
C_t(S_t,a_t) = P \cdot \indicate{X_t > 0} - F \cdot \indicate{X_t=0} - r(X_t,Y_t)\, \left(1-\indicate{a_t=0} \, \indicate{X_t>0}\right).
\end{equation*}

Let $f^-:\mathcal X \times \mathcal Y \rightarrow [0,1]$ be a nonincreasing function in all of its arguments. $X_t$ obeys the following dynamics. If $X_t = 0$ or if $a_t=1$, then $X_{t+1} = X_\textnormal{max}$ with probability 1 (regeneration or replacement). Otherwise, $X_t$ decreases with probability $f^-(X_t,Y_t)$ or stays at the same level with probability $1-f^-(X_t,Y_t)$. The transition function is written:
\begin{equation*} 
\begin{aligned}
X_{t+1} =  \Bigl(X_t &\cdot \indicate{U_{t+1} > f(X_t,Y_t)} + [X_t - \epsilon_{t+1}^X]^+ \cdot \indicate{U_{t+1} \le f(X_t,Y_t)}  \Bigr) \, \indicate{a_t=0} \, \indicate{X_t >0} \\
&+X_\textnormal{max} \, \left (1-\indicate{a_t=0} \, \indicate{X_t>0} \right),
\end{aligned}
\end{equation*}
where $U_{t+1}$ are i.i.d.\ uniform random variables over the interval $[0,1]$ and $\epsilon_{t+1}^X$ are i.i.d.\ \emph{discrete} uniform random variables over $\{1,2,\ldots,\epsilon_\textnormal{max}\}$. The first part of the transition function covers the case where we wait ($a_t=0$) and the asset still has value ($X_t>0$); depending on the outcome of $U_{t+1}$, its value either remains at its current level or depreciates by some random amount $\epsilon_{t+1}^X$. The second part of the formula reverts $X_{t+1}$ to $X_\textnormal{max}$ whenever $a_t=1$ or $X_t=0$. 

Let $Y_t^i$ evolve stochastically such that if $a_t=0$ and $X_t>0$, $Y_{t+1}^i \le Y_t^i$ with probability 1. Otherwise, the external factors also reset: $Y_{t+1}^i=Y^i_\textnormal{max}$:
\begin{equation*}
Y_{t+1}^i = [Y^i_t - \epsilon_{t+1}^i]^+ \cdot \indicate{a_t=0} \, \indicate{X_t >0} + Y^i_\textnormal{max} \cdot \left (1-\indicate{a_t=0} \, \indicate{X_t>0} \right),
\end{equation*}
where $\epsilon_{t+1}^i$ are i.i.d. (across $i$ and $t$) Bernoulli with a fixed parameter $p^i$. Thus, we take the information process to be $W_{t+1} = (U_{t+1}, \epsilon_{t+1}^X,\epsilon^1_{t+1},\epsilon_{t+1}^2,\ldots,\epsilon_{t+1}^n)$, which is independent of $S_t$. Moreover, we take $C_T(x,y) = 0$ for all $(x,y) \in \mathcal X \times \mathcal Y$. The following proposition establishes that Assumption \ref{ass:Emono} is satisfied.

\begin{restatable}{prop}{replmono}
Under the regenerative optimal stopping model, define the Bellman operator $H$ as in (\ref{Hdef}) and let $\preceq$ be the componentwise inequality over all dimensions of the state space. Then, Assumption \ref{ass:Emono} is satisfied. In particular, this implies that the optimal value function is monotone: for each $t \le T$, $V_t^*(X_t,Y_t)$ is nondecreasing in both $X_t$ and $Y_t$ (i.e., in all $n$ dimensions of the state variable $S_t$).
\label{prop:replmono}
\end{restatable}
\begin{proof}
See Appendix \ref{sec:appendix}.
\end{proof}


\subsubsection{Parameter Choices} In the numerical work of this paper, we consider five variations of the problem, where the dimension varies across $n=3$, $n=4$, $n=5$, $n=6$, and $n=7$, and the labels assigned to these are \texttt{R3}, \texttt{R4}, \texttt{R5}, \texttt{R6}, and \texttt{R7}, respectively. The following set of parameters are used across all of the problem instances. We use $X_\textnormal{max} = 10$ and $Y_\textnormal{max}^i=10$, for $i=1,2,\ldots,n-1$ over a finite time horizon $T=25$. The probability of the $i$--th external factor $Y_t^i$ decrementing, $p_i$, is set to $p_i=i/(2n)$, for $i=1,2,\ldots,n-1$. Moreover, the probability of the asset's value $X_t$ depreciating is given by:
\[
f^-(X_t,Y_t) = 1 - \frac{X_t^2 + \|Y_t \|_2^2}{X_\textnormal{max}^2 + \|Y_\textnormal{max} \|_2^2},
\]
and we use $\epsilon_\textnormal{max}=5$. The revenue is set to be $P=100$ and the penalty to be $F=1000$. Finally, we let the replacement cost be:
\[
r(X_t,Y_t) = 400 + \frac{2}{n}\, \bigl(X_\textnormal{max}^2- X_t^2 + \|Y_\textnormal{max} \|_2^2  - \|Y_t \|_2^2 \bigr),
\]
which ranges from 400 to 600. All of the policies that we compute assume an initial state of $S_0=(X_\textnormal{max},Y_\textnormal{max})$. For each of the five problem instances, Table \ref{table:sizes1} gives the cardinalities of the state space, effective state space (i.e., $(T+1)\, |\mathcal S|$), and action space, along with the computation time required to solve the problem exactly using backward dynamic programming.
In the case of \texttt{R7}, we have an effective state space of nearly half a billion, which requires over a week of computation time to solve exactly.

\begin{table}[h]
\centering
\scriptsize
\begin{tabular}{@{}crrrr@{}}\toprule

\multicolumn{1}{c}{\textbf{Label}} & \multicolumn{1}{c}{\textbf{State Space}} & \multicolumn{1}{c}{\textbf{Eff. State Space}} & \multicolumn{1}{c}{\textbf{Action Space}} & \multicolumn{1}{c}{\textbf{CPU (Sec.)}}\\
\midrule
\texttt{R3} & 1{,}331 & 33{,}275 & 2 & 49\\
\texttt{R4} & 14{,}641 & 366{,}025 & 2 & 325\\
\texttt{R5} & 161{,}051 & 4{,}026{,}275 & 2 & 3{,}957\\
\texttt{R6} & 1{,}771{,}561 & 44{,}289{,}025 & 2 & 49{,}360 \\
\texttt{R7} & 19{,}487{,}171 & 487{,}179{,}275 & 2 & 620{,}483\\
\bottomrule
\end{tabular}
\vspace{1em}
\caption{Basic Properties of Regenerative Optimal Stopping Problem Instances}
\label{table:sizes1}
\end{table}
\subsubsection{Results}
Figure \ref{fig:Riter} displays the empirical results of running Monotone--ADP and each of the aforementioned ADP/RL algorithms on the optimal stopping instances \texttt{R3}--\texttt{R7}. Due to the vast difference in size of the problems (e.g., \texttt{R7} is 14{,}000 times larger than \texttt{R3}), each problem was run for a different number of iterations. First, we point out that AVI and QL barely make any progress, even in the smallest instance \texttt{R3}. However, this observation only partially attests to the value of the simple $\Pi_M$ operator, for it is not entirely surprising as AVI and QL only update one state (or one state--action) per iteration. The fact that Monotone--ADP also outperforms both KBRL and API (in the area of 10\%--15\%) makes a stronger case for Monotone--ADP because in both KBRL and API, there is a notion of generalization to the entire state space. As we mentioned earlier, besides the larger optimality gap, the main concerns with KBRL and API are, respectively, bandwidth selection and policy oscillations.
\begin{figure}[!ht]
        \centering
        \begin{subfigure}[b]{0.32\textwidth}
                \centering
                \includegraphics[width=\textwidth]{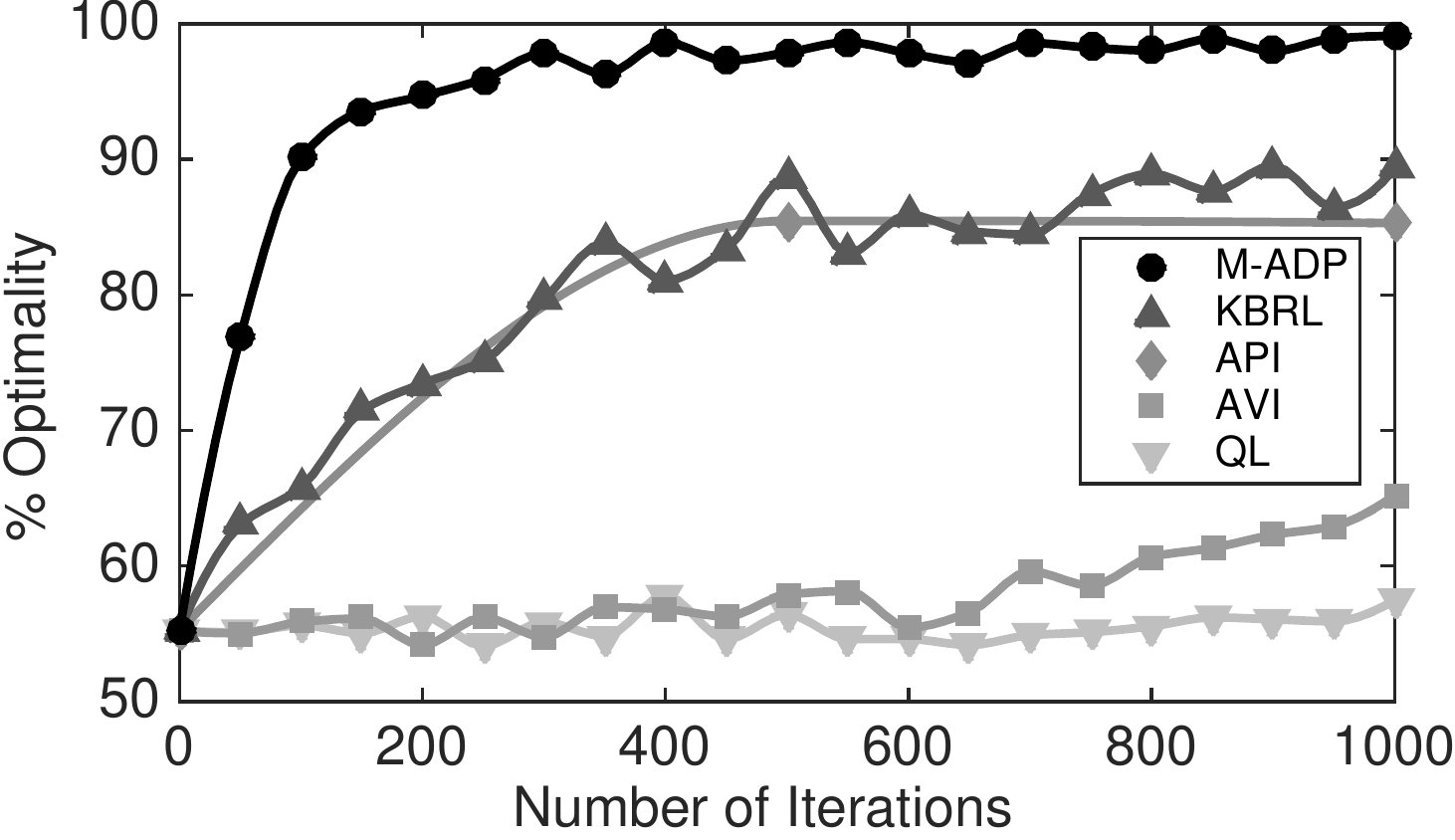}
                \caption{Instance \texttt{R3}}
        \end{subfigure}
        \begin{subfigure}[b]{0.32\textwidth}
                \centering
                \includegraphics[width=\textwidth]{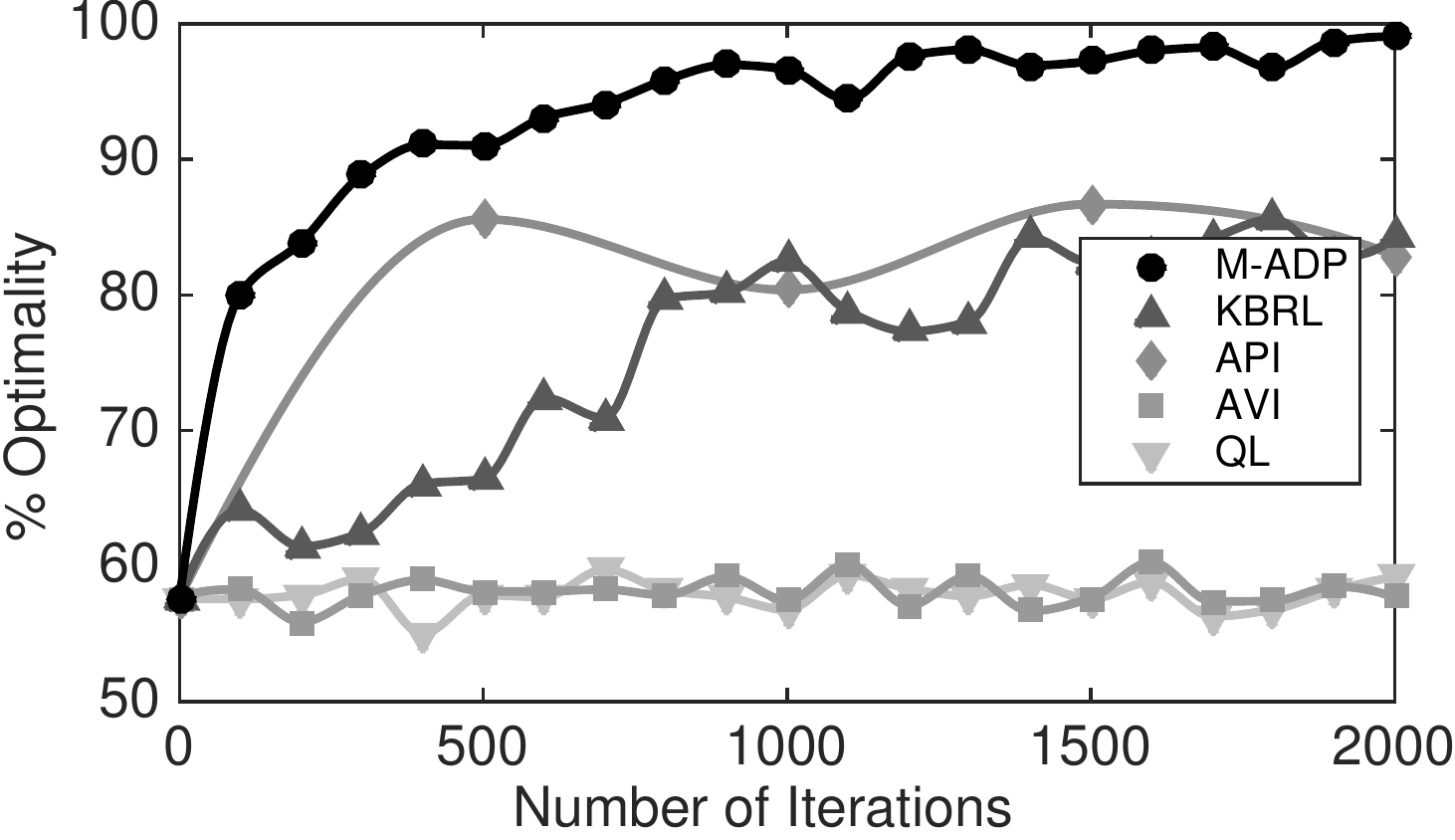}
                \caption{Instance \texttt{R4}}
        \end{subfigure}
        \begin{subfigure}[b]{0.32\textwidth}
                \centering
                \includegraphics[width=\textwidth]{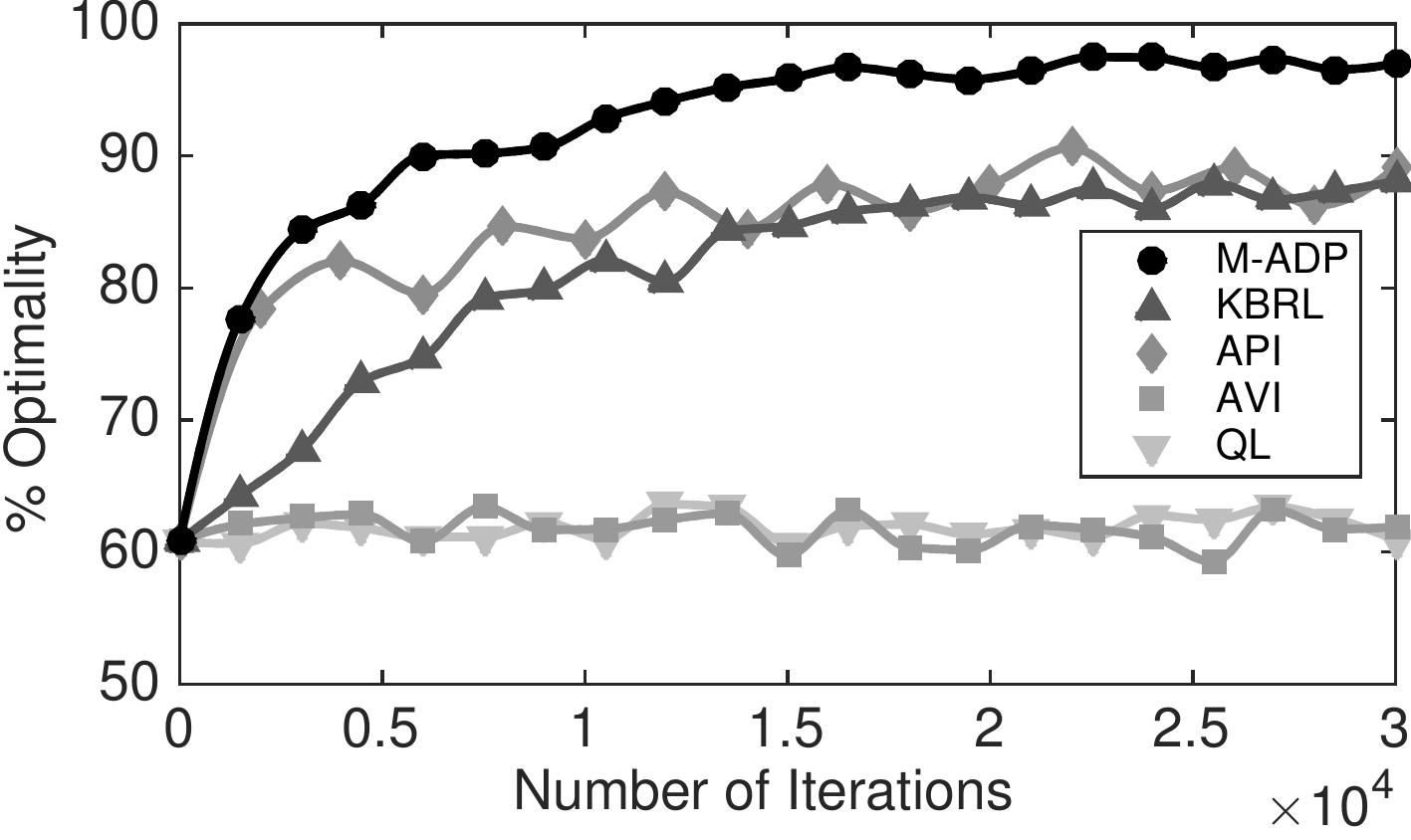}
        		\caption{Instance \texttt{R5}}
        \end{subfigure}\\
       
        \begin{subfigure}[b]{0.32\textwidth}
                \centering
                \includegraphics[width=\textwidth]{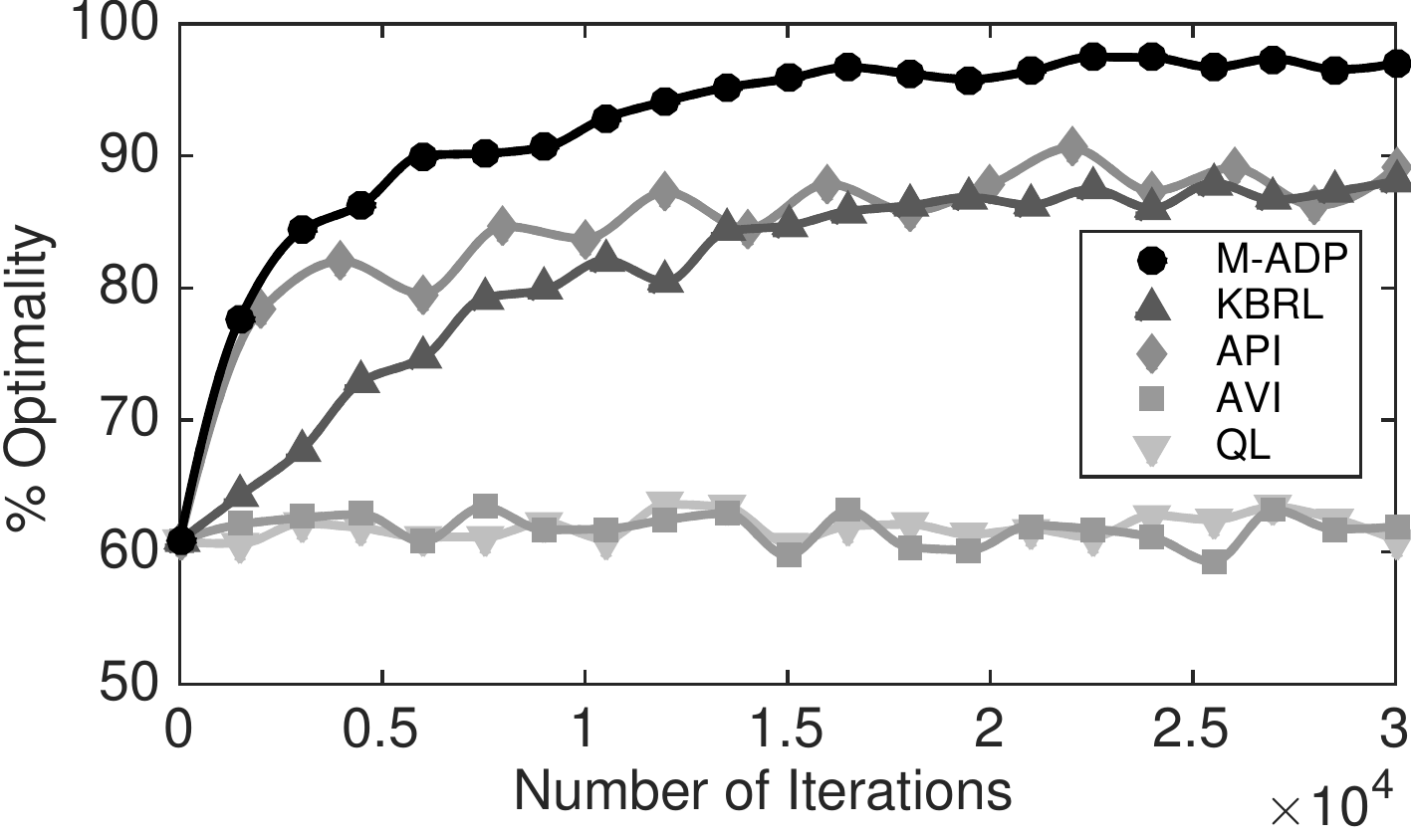}
                \caption{Instance \texttt{R6}}
        \end{subfigure}
        \begin{subfigure}[b]{0.32\textwidth}
                \centering
                \includegraphics[width=\textwidth]{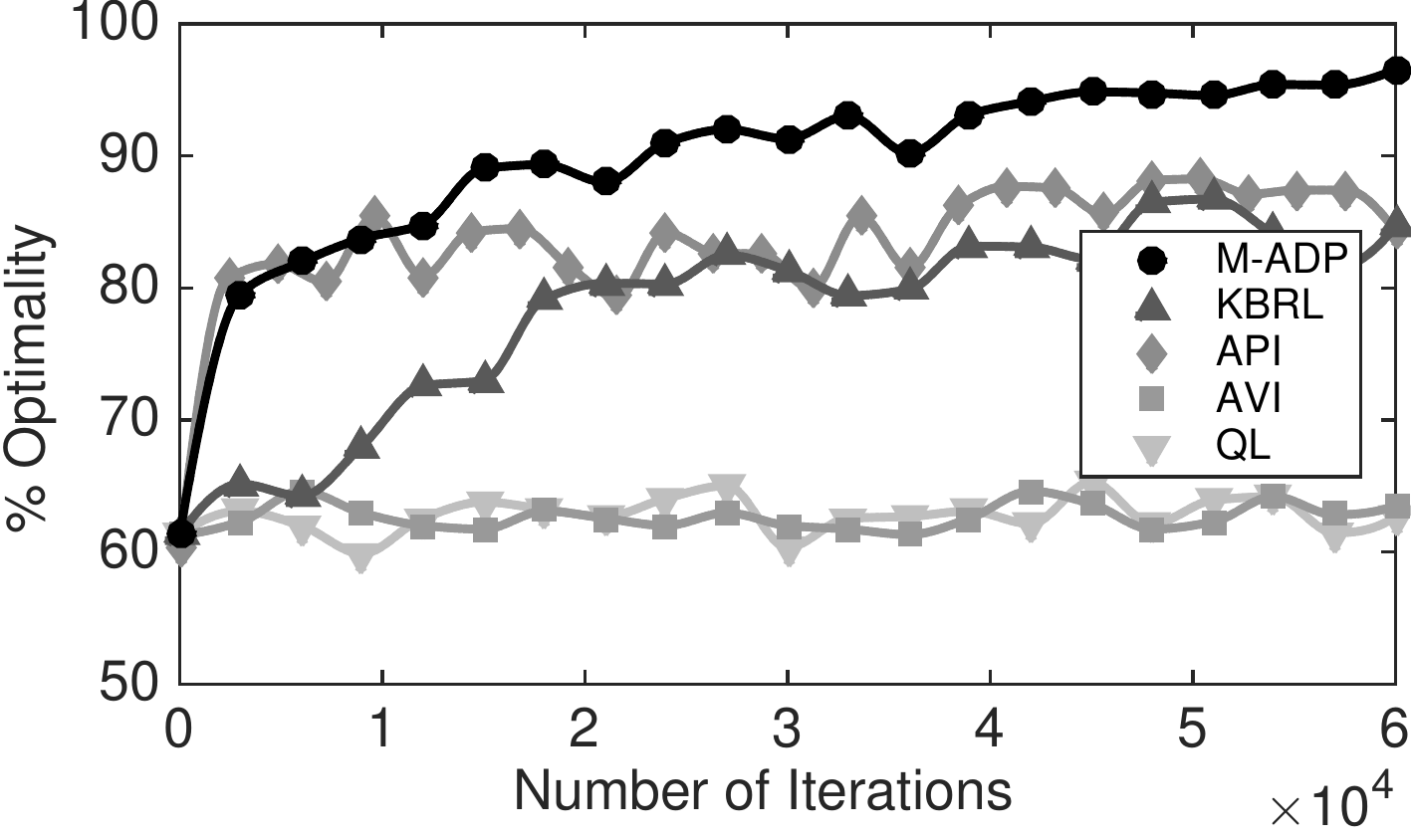}
                \caption{Instance \texttt{R7}}
        \end{subfigure}
        \caption{Comparison of Monotone--ADP to Other ADP/RL Algorithms}
        \label{fig:Riter}
\end{figure}

Question (2) concerns the computation requirement for Monotone--ADP. The optimality of the Monotone--ADP policies versus the computation times needed for each are shown on the semilog plots in Figure \ref{fig:Rcpu} below. In addition, the single point to the right represents the amount of computation needed to produce the exact optimal solution using BDP. The horizontal line represents 90\% optimality (near--optimal). The plots show that for every problem instance, we can reach near--optimality using Monotone--ADP with around an \emph{order of magnitude} less computation than if we used BDP to compute the true optimal. In terms of a percentage, Monotone--ADP required 5.3\%, 5.2\%, 4.5\%, 4.3\%, and 13.1\% of the optimal solution computation time to reach a near--optimal solution in each of the respective problem instances. From Table \ref{table:sizes1}, we see that for larger problems, the amount of computation needed for an exact optimal policy is unreasonable for any real--world application. Combined with the fact that it is extremely easy to find examples of far more complex problems (the relatively small $X_\textnormal{max}$ and $Y_\textnormal{max}^i$ makes this example still somewhat tractable), it should be clear that exact methods are not a realistic option, even given the attractive theory of finite time convergence. 

\begin{figure}[!ht]
        \centering
        \begin{subfigure}[b]{0.32\textwidth}
                \centering
                \includegraphics[width=\textwidth]{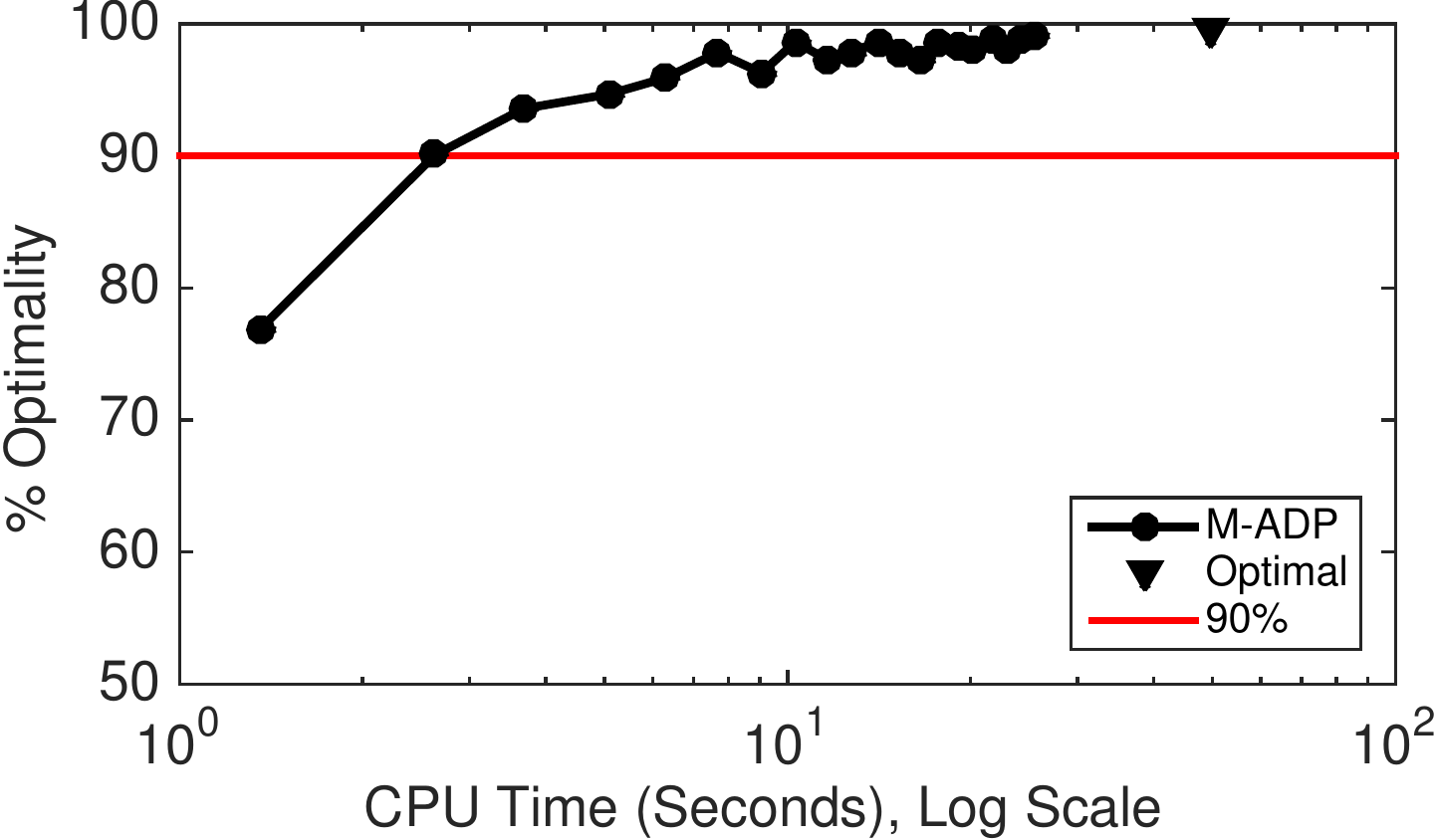}
                \caption{Instance \texttt{R3}}
        \end{subfigure}
        \begin{subfigure}[b]{0.32\textwidth}
                \centering
                \includegraphics[width=\textwidth]{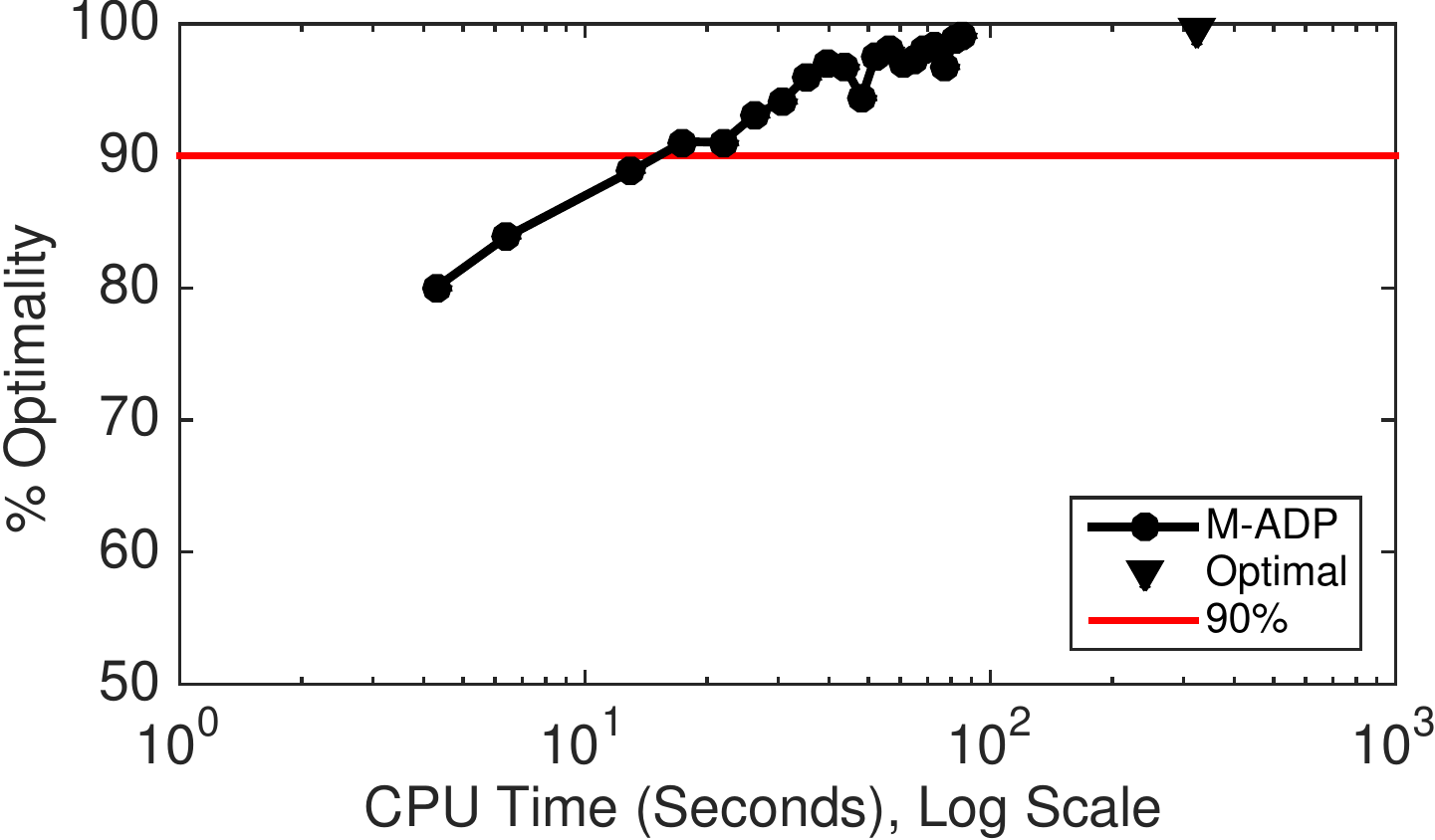}
                \caption{Instance \texttt{R4}}
        \end{subfigure}
        \begin{subfigure}[b]{0.32\textwidth}
                \centering
                \includegraphics[width=\textwidth]{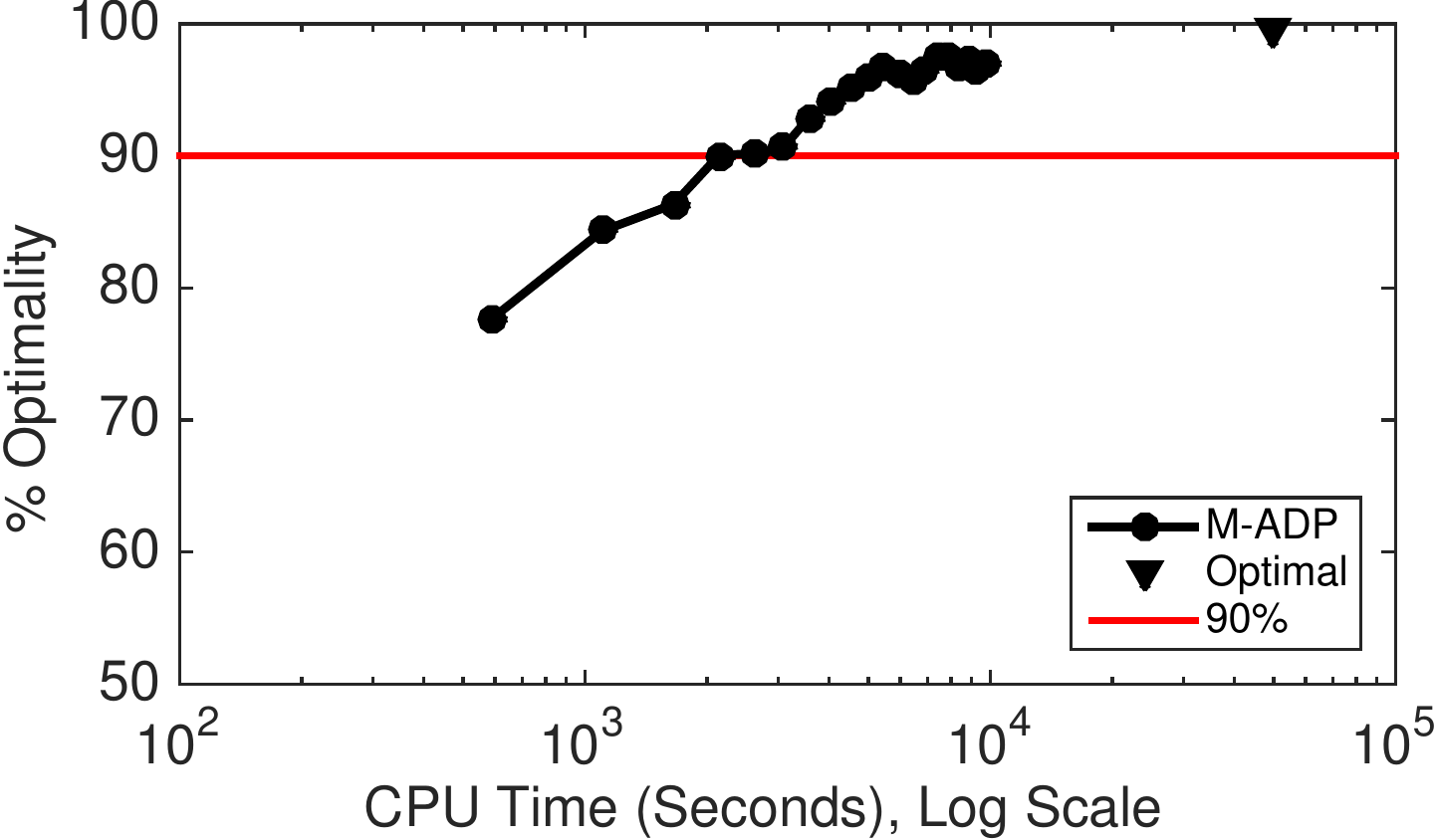}
        		\caption{Instance \texttt{R5}}
        \end{subfigure}\\
       
        \begin{subfigure}[b]{0.32\textwidth}
                \centering
                \includegraphics[width=\textwidth]{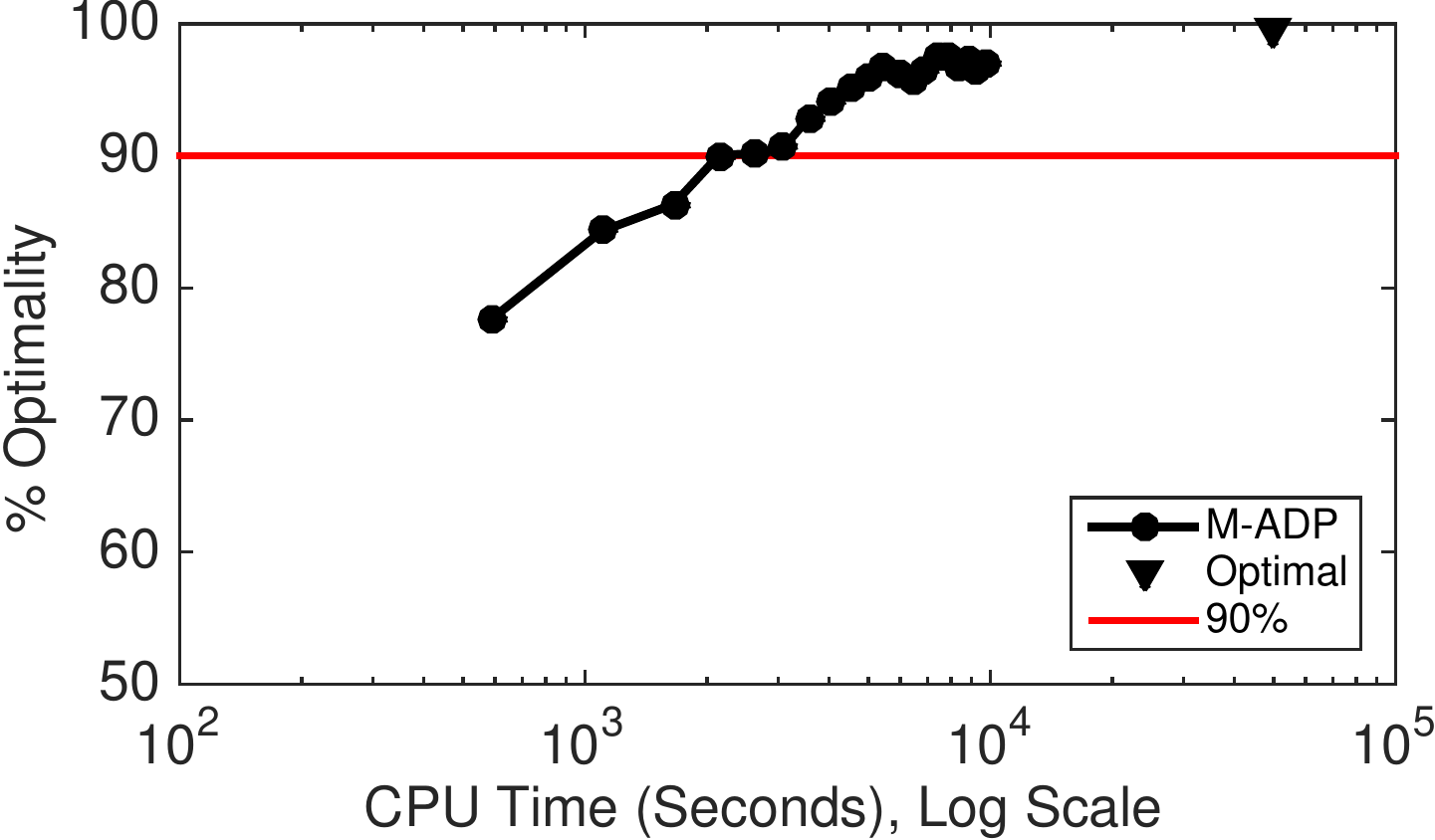}
                \caption{Instance \texttt{R6}}
        \end{subfigure}
        \begin{subfigure}[b]{0.32\textwidth}
                \centering
                \includegraphics[width=\textwidth]{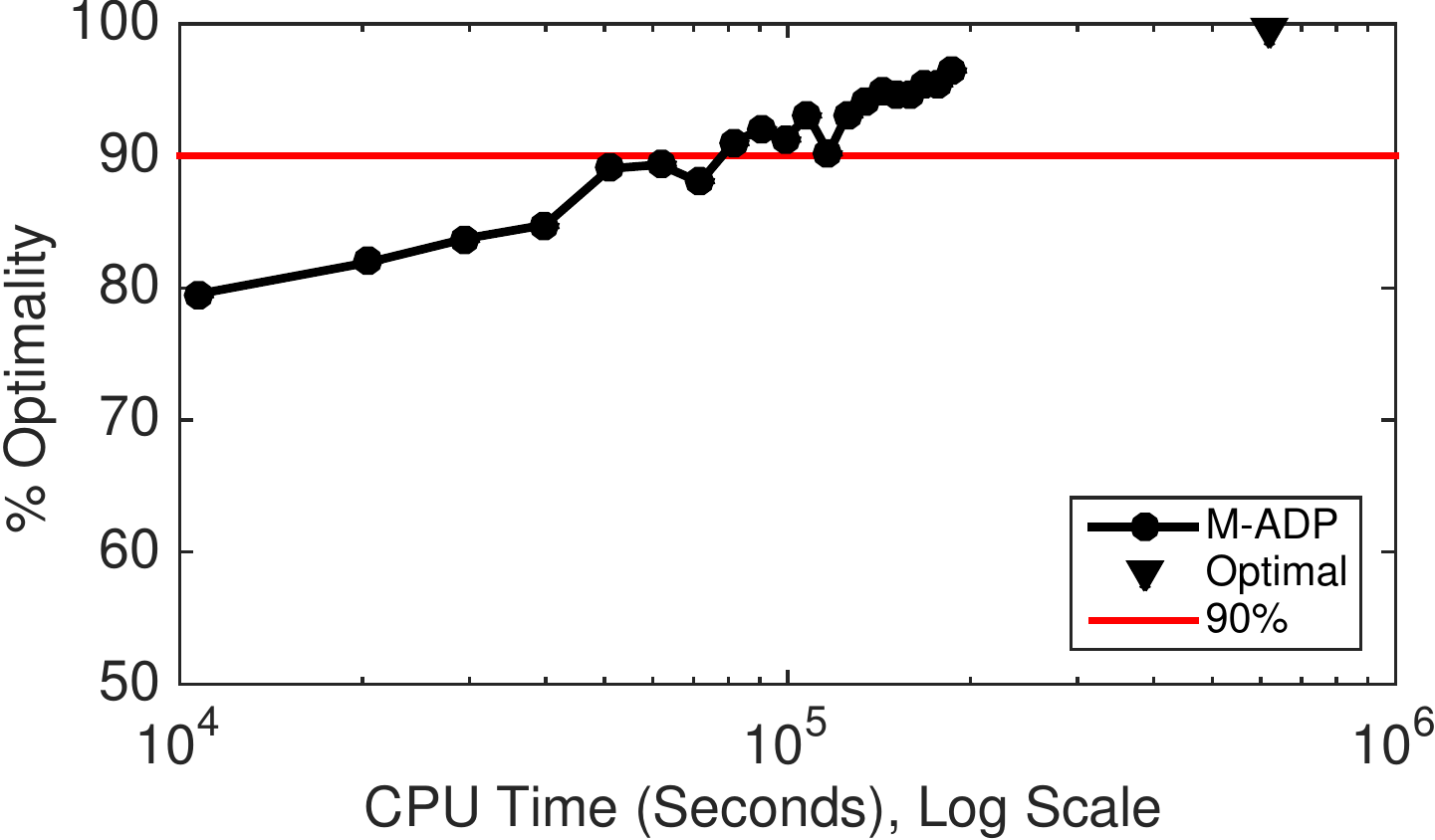}
                \caption{Instance \texttt{R7}}
        \end{subfigure}
        \caption{Computation Times (Seconds) of Monotone--ADP vs. Backward Dynamic Programming}
        \label{fig:Rcpu}
\end{figure}

\subsection{Energy Storage and Allocation}
The recent surge of interest in renewable energy leads us to present a second example application in area energy storage. The goal of this specific problem is to optimize revenues while satisfying demand, in the presence of 1) a storage device, such as a battery, and 2) a (stochastic) renewable source of energy, such as wind or solar. Our action or decision vector is an \emph{allocation decision}, containing five dimensions, that describe how the energy is transferred within our network, consisting of nodes for the storage device, the spot market, demand, and the source of renewable generation (see Figure \ref{fig:storage_diagram}). Similar problems from the literature that share the common theme of storage include \cite{Secomandi2010}, \cite{Carmona2010}, and \cite{Kim2011}, to name a few.
\begin{figure}[h]
	\begin{center}
	\includegraphics[scale=.45]{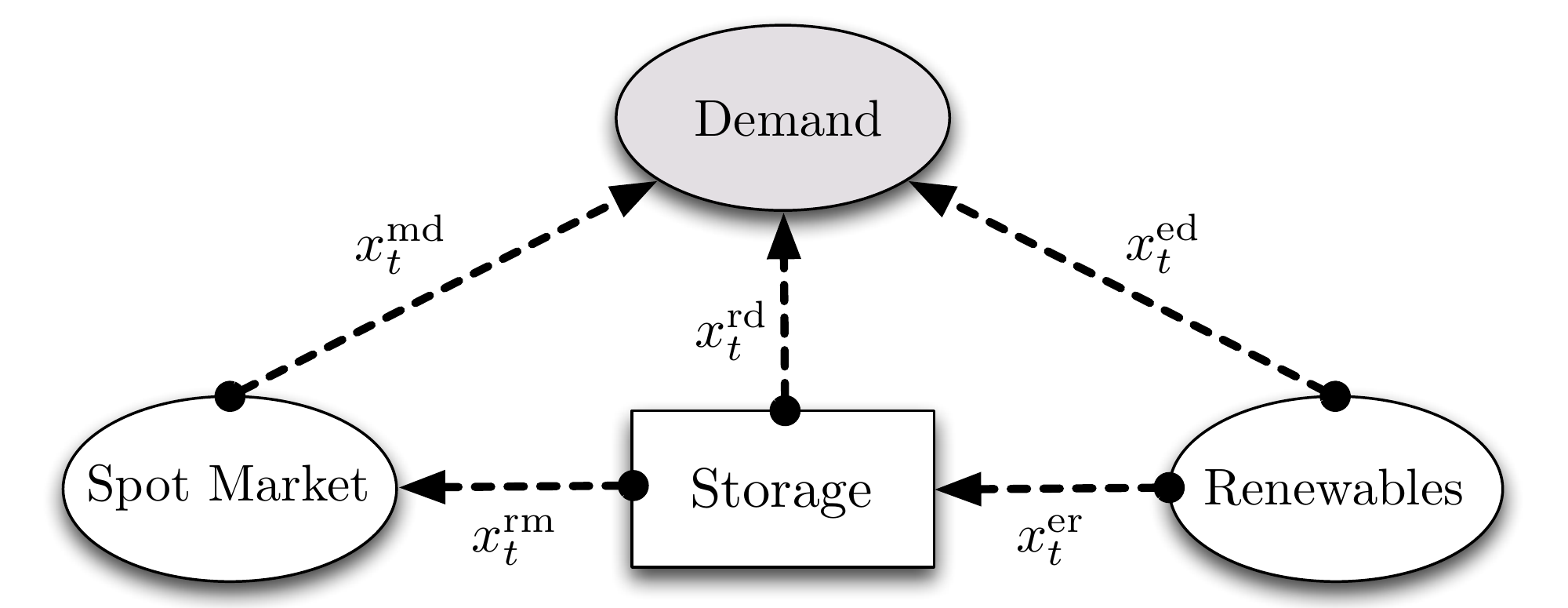}\\
	\end{center}
	\caption{Network Diagram for the Energy Storage Problem}
	\label{fig:storage_diagram}
\end{figure}

Let the state variable $S_t = (R_t, E_t, P_t, D_t) \in \mathcal S$, where $R_t$ is the amount of energy in storage at time $t$, $E_t$ is the amount of renewable generation available at time $t$, $P_t$ is the price of energy on the spot market at time $t$, and $D_t$ is the demand that needs to be satisfied at time $t$. We define the discretized state space $\mathcal S$ by allowing $(R_t, E_t, P_t, D_t)$ to take on all integral values contained within the hyper--rectangle 
\[
[0,R_\textnormal{max}] \times [E_\textnormal{min},E_\textnormal{max}] \times [P_\textnormal{min},P_\textnormal{max}] \times [D_\textnormal{min},D_\textnormal{max}],
\]
where $R_\textnormal{max} \ge 0$, $E_\textnormal{max} \ge E_\textnormal{min} \ge 0$, $P_\textnormal{max} \ge P_\textnormal{min} \ge 0$, and $D_\textnormal{max} \ge D_\textnormal{min} \ge 0$.
Let $\gamma_c$ and $\gamma_d$ be the maximum rates of charge and discharge from the storage device, respectively. The decision vector is given by (refer again to Figure \ref{fig:storage_diagram} for the meanings of the components)
\[
x_t = (\xt[ed],\xt[md],\xt[rd],\xt[er],\xt[rm])^\mathsf{T} \in \mathcal X(S_t),
\]
where the feasible set $\mathcal X(S_t)$ is defined by $x_t \in \mathbb N^5$ intersected with the following, mostly intuitive, constraints:
\begin{align}
(\xt[ed], \xt[md], \xt[rd],\xt[er],\xt[rm])^\mathsf{T} &\ge 0,\label{eq:constraint1}\\
\xt[ed]+ \xt[rd]+\xt[md] &= D_t,\label{eq:constraint2}\\
\xt[rd]+\xt[rm] &\le R_t,\nonumber\\
\xt[er]+\xt[ed] &\le E_t,\nonumber\\
\xt[rd]+\xt[rm] &\le \gamma^d,\nonumber\\
\xt[er] &\le \min \{ \Rmax - R_t, \gamma_c \}\label{eq:constraint6}.
\end{align}
Note that whenever the energy in storage combined with the amount of renewable generation is not enough to satisfy demand, the remainder is purchased from the spot market. The contribution function is given by 
\[
C_t(S_t,x_t) = P_t \, \bigl(D_t + \xt[rm] - \xt[md]\bigr).
\]
To describe the transition of the state variable, first define $\phi = (0,0,-1,1 ,-1)^\mathsf{T}$ to be the flow coefficients for a decision $x_t$ with respect to the storage device. We also assume that the dependence on the past for the stochastic processes $\{E_t\}_{t=0}^T$, $\{P_t\}_{t=0}^T$, and $\{D_t\}_{t=0}^T$ is at most Markov of order one. Let $[\,\cdot\,]^{a}_b = \min(\max(\,\cdot\,,b),a)$. Thus, we can write the transition function for $S_t$ using the following set of equations: 
\begin{equation}
R_{t+1} = R_t + \phi^\mathsf{T}x_t \quad \mbox{and} \quad 
\begin{aligned}
E_{t+1} &= \bigl[ E_t + \hat{E}_{t+1} \bigr]_{E_\textnormal{min}}^{E_\textnormal{max}}\,,\\
P_{t+1} &= \bigl[ P_t + \hat{P}_{t+1} \bigr]_{P_\textnormal{min}}^{P_\textnormal{max}}\,,\\
D_{t+1} &= \bigl[ D_t + \hat{D}_{t+1} \bigr]_{D_\textnormal{min}}^{D_\textnormal{max}}\,,
\end{aligned}
\label{eq:storage_trans}
\end{equation}
where the information process $W_{t+1}=(\hat{E}_{t+1},\hat{P}_{t+1},\hat{D}_{t+1})$ is independent of $S_t$ and $x_t$ (the precise processes used are given in Section \ref{sec:storage_params}). Note that the combined transition function is monotone in the sense of $(i)$ of Proposition \ref{mono_cond_one}, where $\preceq$ is the componentwise inequality.
\begin{restatable}{prop}{storagemono}
Under the energy storage model, define the Bellman operator $H$ as in (\ref{Hdef}), with $\mathcal A$ replaced with the state dependent $\mathcal X(s)$, and let $\preceq$ be the componentwise inequality over all dimensions of the state space. Then, Assumption \ref{ass:Emono} is satisfied. In particular, this implies that the optimal value function is monotone: for each $t \le T$, $V_t^*(R_t,E_t,P_t,D_t)$ is nondecreasing in $R_t$, $E_t$, $P_t$, and $D_t$.
\label{prop:storagemono}
\end{restatable}
\begin{proof}
See Appendix \ref{sec:appendix}.
\end{proof}
\subsubsection{Parameter Choices}
\label{sec:storage_params} In our experiments, a continuous distribution $D$ with density $f_D$ is discretized over a set $\mathcal X_D$ by assigning each outcome $x \in \mathcal X_D$ the probability $f_D(x) / \sum_{x' \in \mathcal X_D} f_D(x')$. We consider two instances of the energy storage problem for $T=25$: the first is labeled \texttt{S1} and has a smaller storage device and relatively low variance in the change in renewable supply $\hat{E}_{t+1}$, while the second, labeled \texttt{S2}, uses a larger storage device and has relatively higher variance in $\hat{E}_{t+1}$. We take $R_\textnormal{min} = 0$ with $R_\textnormal{max} = 30$ for \texttt{S1} and $R_\textnormal{max} = 50$ for \texttt{S2}, and we set $\gamma_c=\gamma_d=5$ for \texttt{S1} and \texttt{S2}. The stochastic renewable supply has characteristics given by $E_\textnormal{min} = 1$ and $E_\textnormal{max}=7$, with $\hat{E}_{t+1}$ being i.i.d. discrete uniform random variables over $\{0,\pm1\}$ for \texttt{S1} and $\hat{E}_{t+1}$ being i.i.d. $\mathcal N(0,3^2)$ discretized over the set $\{0, \pm 1, \pm 2, \ldots, \pm 5\}$. For both cases, we have $P_\textnormal{min} = 30$ and $P_\textnormal{max}=70$ with $\hat{P}_{t+1} = \epsilon_{t+1}^P + \mathbf{1}_{\{U_{t+1} < \,p\}} \, \epsilon_{t+1}^J$, in order to simulate price spikes (or jumps). The noise term $\epsilon_{t+1}^P$ is $\mathcal N(0,2.5^2)$ discretized over $\{0,\pm1,\pm2,\ldots,\pm8\}$; the noise term $\epsilon_{t+1}^J$ is $\mathcal N(0,50^2)$ discretized over $\{0,\pm1,\pm2,\ldots,\pm40\}$; and $U_{t+1}$ is $\mathcal U(0,1)$ with $p=0.031$. Lastly, for the demand process, we take $D_\textnormal{min} = 0$, $D_\textnormal{max}=7$, and $D_t+\hat{D}_{t+1} = \bigl\lfloor 3-4\,\sin(2\pi\,(t+1)/T) \bigr\rfloor + \epsilon^D_{t+1}$, where $\epsilon^D_{t+1}$ is $\mathcal N(0,2^2)$ discretized over $\{0,\pm1,\pm2\}$, in order to roughly model the seasonality that often exists in observed energy demand. For both problems, we use an initial state of an empty storage device and the other dimensions of the state variable set to their minimum values: $S_t = (R_\textnormal{min}, E_\textnormal{min}, P_\textnormal{min}, D_\textnormal{min})$. Table \ref{table:sizes2} summarizes the sizes of these two problems. Since the size of the action space is not constant over the state space, we report the average, i.e., $|\mathcal S|^{-1} \sum_{s} |\mathcal X(s)|$. The maximum size of the feasible set over the state space is also given (and is the same for both \texttt{S1} and \texttt{S2}). Finally, we remark that the greater than linear increase in computation time for \texttt{S2} compared to \texttt{S1} is due to the larger action space and the larger number of random outcomes that need to be evaluated.

\begin{table}[h]
\centering
\scriptsize
\begin{tabular}{@{}crrrr@{}}\toprule
\multicolumn{1}{c}{\textbf{Label}} & \multicolumn{1}{c}{\textbf{State Space}} & \multicolumn{1}{c}{\textbf{Eff. State Space}} & \multicolumn{1}{c}{\textbf{Action Space}} & \multicolumn{1}{c}{\textbf{CPU (Sec.)}}\\
\midrule
\texttt{S1} & 71{,}176 & 1{,}850{,}576 & 165, Max: 623 & 41{,}675\\
\texttt{S2} & 117{,}096 & 3{,}044{,}496 & 178, Max: 623 & 115{,}822\\
\bottomrule
\end{tabular}
\vspace{1em}
\caption{Basic Properties of Energy Storage Problem Instances}
\label{table:sizes2}
\end{table}

\subsubsection{Results}
For this problem, we did not implement QL because of the impracticality of working with state--action pairs in problem domains with vastly larger action space than optimal stopping. The state--dependent action space also introduces implementation difficulties. With regard to KBRL, API, and AVI, the results for energy storage tell a similar story as before, as seen in Figure \ref{fig:Siter}. 
\begin{figure}[!ht]
        \centering
        \begin{subfigure}[b]{0.35\textwidth}
                \centering
                \includegraphics[width=\textwidth]{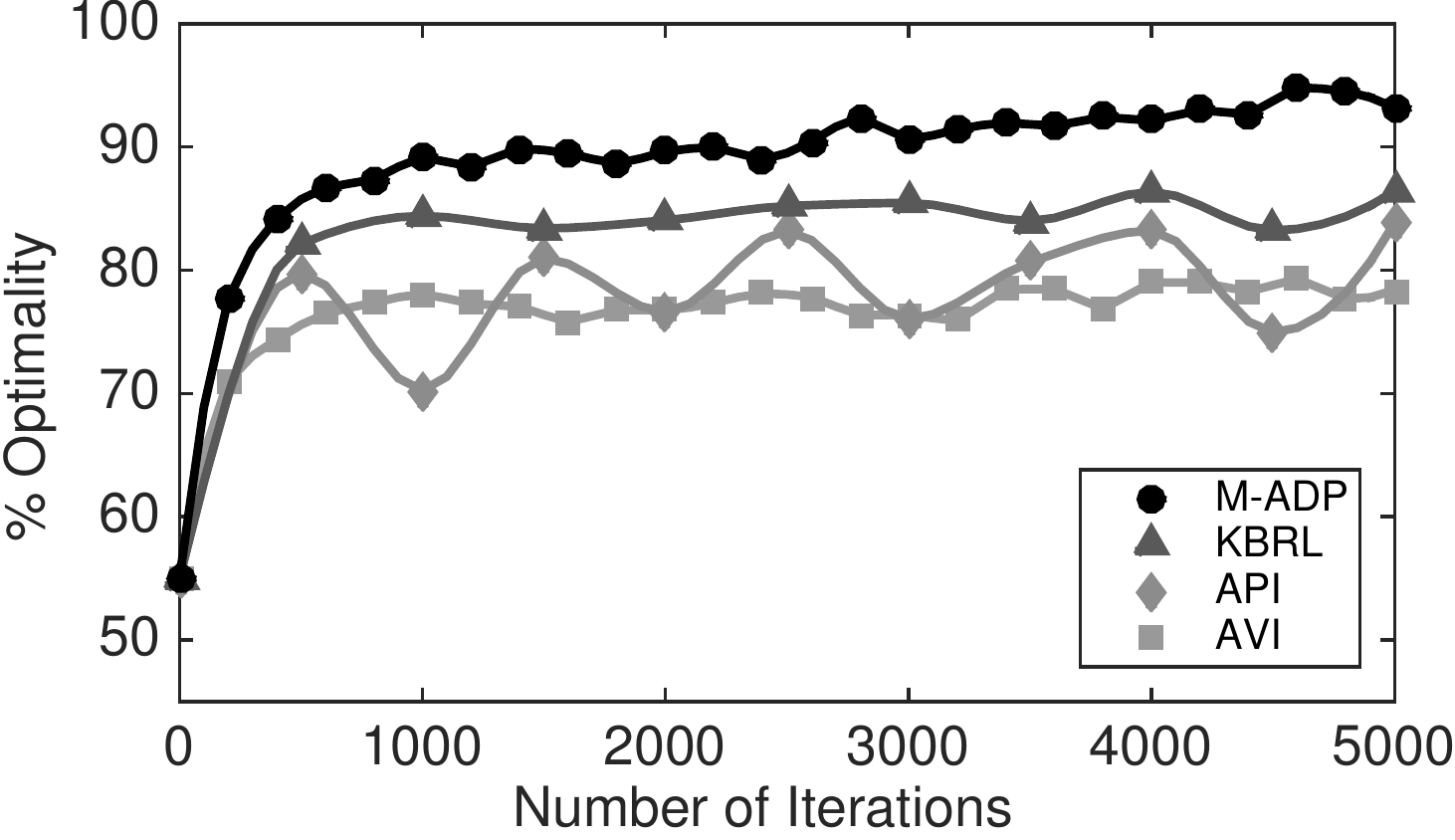}
                \caption{Instance \texttt{S1}}
        \end{subfigure}
        \begin{subfigure}[b]{0.35\textwidth}
                \centering
                \includegraphics[width=\textwidth]{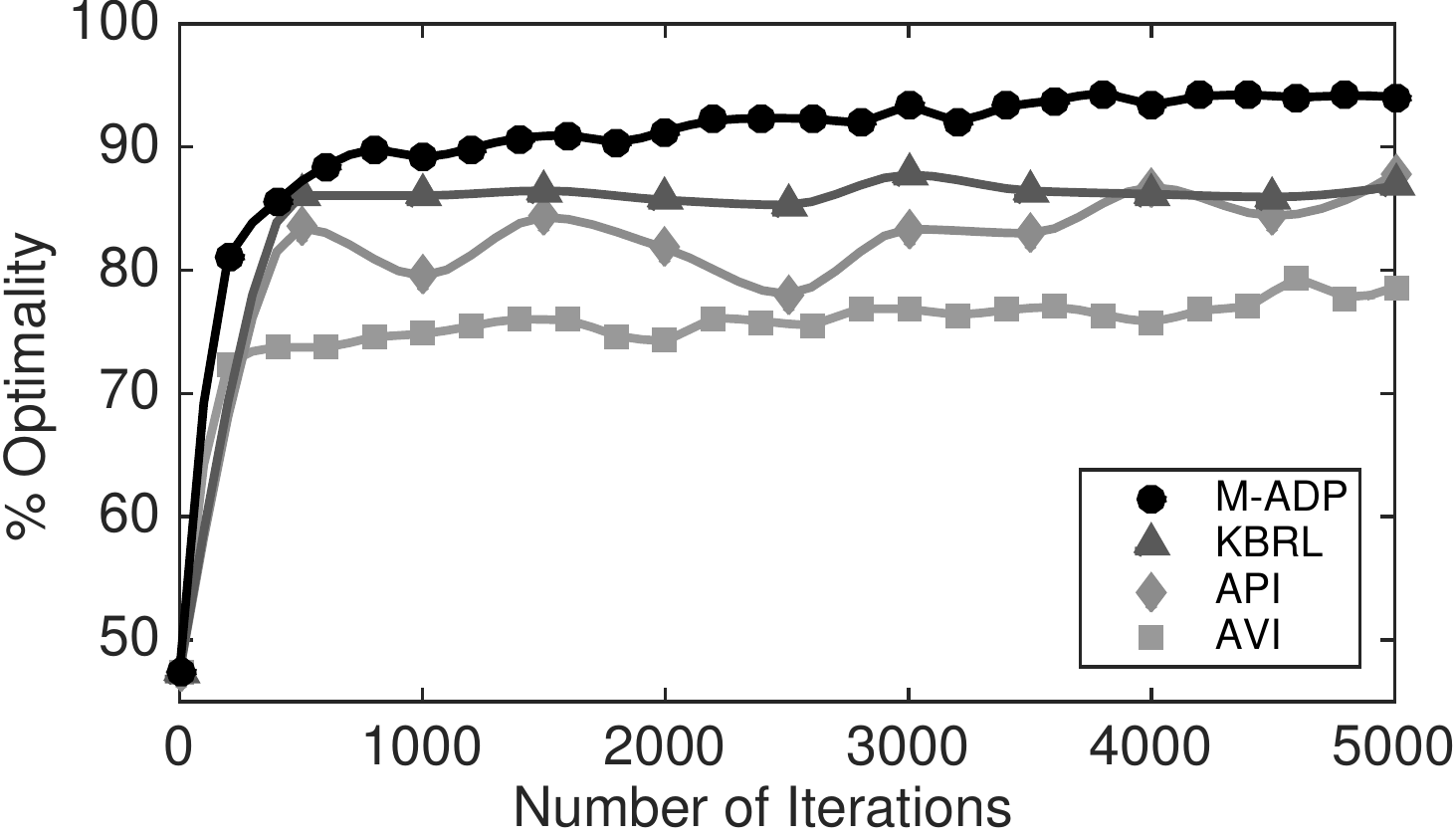}
                \caption{Instance \texttt{S2}}
        \end{subfigure}
        \caption{Comparison of Monotone--ADP to Other ADP/RL Algorithms}
        \label{fig:Siter}
\end{figure}
The computational gap between Monotone--ADP and BDP, however, has increased even further. As illustrated in Figure \ref{fig:Scpu}, the ratio of the amount of computation time for Monotone--ADP to reach near--optimality to the amount of computation needed for BDP stands at 1.9\% for \texttt{S1} and 0.7\% for \texttt{S2}, reaching two orders of magnitude.
\begin{figure}[!ht]
        \centering
        \begin{subfigure}[b]{0.35\textwidth}
                \centering
                \includegraphics[width=\textwidth]{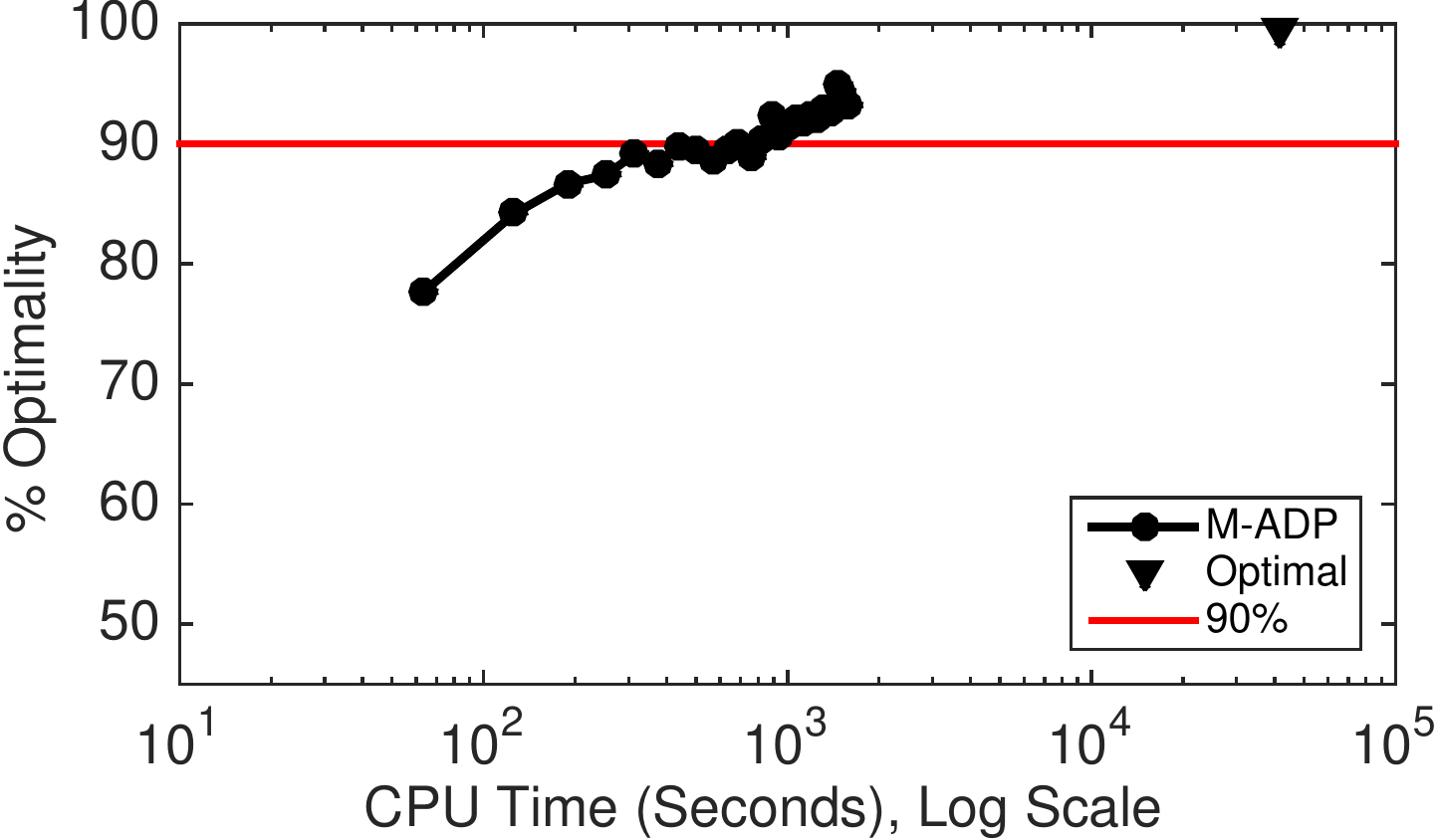}
                \caption{Instance \texttt{S1}}
        \end{subfigure}
        \begin{subfigure}[b]{0.35\textwidth}
                \centering
                \includegraphics[width=\textwidth]{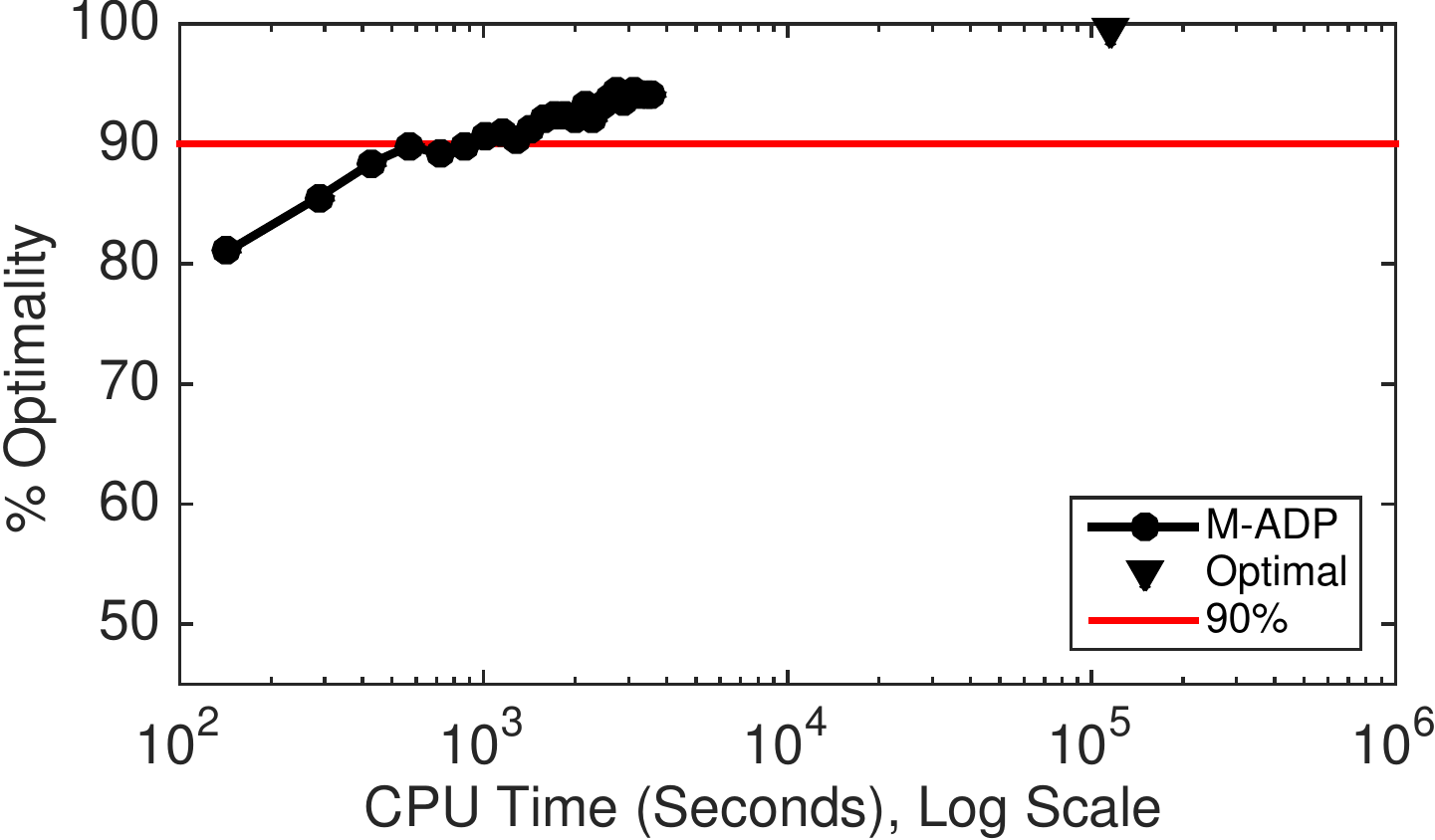}
                \caption{Instance \texttt{S2}}
        \end{subfigure}
        \caption{Computation Times (Seconds) of Monotone--ADP vs. Backward Dynamic Programming}
        \label{fig:Scpu}
\end{figure}

\subsection{Glycemic Control for Diabetes} Our final application is in the area of healthcare, concerning optimal treatment decisions for glycemic control (i.e., the management of blood glucose) in Type 2 diabetes patients over a time horizon $t \in \{0,1,\ldots, T\}$. The model is based primarily on the work \cite{Hsih2010} (with a few extensions) and also exhibits similarities to the ones described in \cite{Kurt2011}, \cite{Mason2012}, and \cite{Mason2014}. Because diabetes is a chronic disease, its patients require long--term, personalized treatment plans that take into account a variety of factors, such as measures of blood glucose, negative side effects, and medical costs. The idea of the model is to maximize a utility function (often chosen to be related to \emph{quality--adjusted life years} or QALYs in the literature) over time. The action at each time period is a treatment decision that has a stochastic effect on the patient's state of health, represented via the state variable $S_t = (H_t^a, H_t^b, H_t^c, H_t^d)$, where $H_t^a \in \mathcal H^a = \{H_\textnormal{min}^a, \ldots, H_\textnormal{max}^a \}$ and $H_t^b\in \mathcal H^b= \{H_\textnormal{min}^b, \ldots, H_\textnormal{max}^b\}$ are measures of blood glucose, $H_t^c \in \mathcal H^c= \{H_\textnormal{min}^c, \ldots, H_\textnormal{max}^c \}$ is a patient's BMI (body mass index), and $H_t^d \in \mathcal H^d = \{H_\textnormal{min}^d, \ldots, H_\textnormal{max}^d \}$ is the severity of side effects (e.g., gastrointestinal discomfort or hypoglycemia). More specifically, let $H_t^a$ be the FPG (fasting plasma glucose) level, a short term indicator of blood glucose and $H_t^b$ be the $\textnormal{HbA}_{\textnormal{1c}}$ (glycated hemoglobin) level, an indicator of average blood glucose over a longer term of a few months. The set of available treatment decisions is $\mathcal X = \{0,1,2,3,4\}$. Below is basic summary of the various treatments:
\begin{itemize}
\item \textbf{No treatment}, $a_t = 0$. It is often the case that the usual recommendations of diet and exercise can be sufficient when the patient's glucose levels are in the normal range; there is no cost and no risk of increasing the severity of side effects.
\item \textbf{Insulin sensitizers}, $a_t = 1$. These drugs increase the patient's sensitivity to insulin; the two most common types are called \emph{biguanides} and \emph{thiazolidinediones} (TZDs). 
\item \textbf{Secretagogues}, $a_t = 2$. Pancreatic $\beta$--cells are responsible for the release of insulin in the human body. Secretagogues act directly on $\beta$--cells to increase insulin secretion. 
\item \textbf{Alpha--glucosidase inhibitors}, $a_t = 3$. As the name suggests, this treatment disrupt the enzyme \emph{alpha--glucosidase}, which break down carbohydrates into simple sugars, and therefore decreases the rate at which glucose is absorbed into the bloodstream.
\item \textbf{Peptide analogs}, $a_t = 4$. By acting on certain \emph{incretins}, which are hormones that affect the insulin production of pancreatic $\beta$--cells, this type of medication is able to regulate blood glucose levels.
\end{itemize}
We make the assumption that the patient under consideration has levels of $H_t^a$, $H_t^b$, $H_t^c$, and $H_t^d$ \emph{higher} (i.e., worse) than the normal range (as is typical of a diabetes patient) and only model this regime. Therefore, assume that we have \emph{nonincreasing} utility functions $u_t^i : \mathcal H^i \rightarrow \mathbb R$ for $i \in \{a,b,c,d\}$ and that the cost of treatment is $P \ge 0$. The contribution function at time $t$ is
\[
C_t(S_t,a_t) = u_t^a(H_t^a) + u_t^b(H_t^b) + u_t^b(H_t^b) +u_t^d(H_t^d) - P \cdot \mathbf{1}_{\{a_t \ne 0\}}.
\]
Furthermore, the information process in this problem is the stochastic effect of the treatment on the patient. This effect is denoted $W_{t+1} = (\hat{H}_{t+1}^a,\hat{H}_{t+1}^b,\hat{H}_{t+1}^c,\hat{H}_{t+1}^d)$ with the transitions given by $\bigl[ H_t^i + \hat{H}^i_{t+1} \bigl]_{H^i_\textnormal{min}}^{H^i_\textnormal{max}}$ for each $i \in \mathcal \{a,b,c,d\}$. Of course, the distribution of $W_{t+1}$ depends on the treatment decision $x_t$, but we assume it is independent of the state variable $S_t$.

Notice that in this problem, the contribution function is nonincreasing with respect to the state variable, reversing the monotonicity in the value function as well.
\begin{restatable}{prop}{glycemicmono}
Under the glycemic control model, define the Bellman operator $H$ as in (\ref{Hdef}), with $\mathcal A$ replaced with the set of treatment decisions $\mathcal X$, and let $\preceq$ be the componentwise inequality over all dimensions of the state space. Then, Assumption \ref{ass:Emono} is satisfied, with the direction of $\preceq$ reversed. In particular, this implies that the optimal value function is monotone: for each $t \le T$, $V_t^*(H_t^a,H_t^b,H_t^c,H_t^d)$ is nonincreasing in $H_t^a$, $H_t^b$, $H_t^c$, and $H_t^d$.
\label{prop:glycemicmono}
\end{restatable}
\begin{proof}
Proposition \ref{mono_cond_one} can be applied, with the direction of the inequalities reversed.
\end{proof}

\subsubsection{Parameter Choices} We remark that the parameters used in our experimental work are not realistic or estimated from data, but chosen so that the resulting MDP is interesting and not too easy to solve (we found that the original parameters from \cite{Hsih2010} created a problem that Monotone--ADP could solve to near optimality in as little as 50 iterations, making for a weak comparison). We consider two variations of the glycemic control problem, labeled \texttt{G1} and \texttt{G2}, where the only difference is that for \texttt{G1}, we assume there is no cost of treatment, i.e., $P=0$, and for \texttt{G2}, we set $P=2$. This comparison is made to illustrate that a trivial, seemingly inconsequential change to the problem can create dramatic difficulties for certain ADP algorithms, as we see in the next section.

The finite time horizon for glycemic control is chosen as $T=12$ (time is typically measured in units of a few months for a problem such as this). The lower and upper bounds of the state variable are given by
\begin{equation*}
\begin{aligned}
\bigl(H^a_\textnormal{min}, H^b_\textnormal{min}, H^c_\textnormal{min}, H^d_\textnormal{min}\bigr) &=(68,4,19,0),\\
\bigl(H^a_\textnormal{max}, H^b_\textnormal{max}, H^c_\textnormal{max}, H^d_\textnormal{max}\bigr) &=(300,20,50,10).
\end{aligned}
\end{equation*} 
For each health indicator $i \in \{a,b,c\}$, the utility function is taken to have the form $u_t^i(h) = k^i \, \log(H_\textnormal{max}^i-h)$ for all $t$. The values of the constant are $k^a = 4.586$, $k^b=7.059$, and $k^c = 5.771$. Additionally, let the utility function for side effects be given by $u_t^d(h) = -10h$.

Next, we assume that the distribution of $(\hat{H}_{t+1}^a,\hat{H}_{t+1}^b,\hat{H}_{t+1}^c)^\mathsf{T}$ conditional on $x_t = x$ is multivariate normal for all $t$, with mean vector $\mu^x$ and covariance matrix $\Sigma^x$:
\[
\mu^x = \left[ \begin{array}{c}
\mu^{x,a} \\
\mu^{x,b} \\
\mu^{x,c} \end{array} \right] \quad \mbox{and} \quad
\Sigma^x = \left[ \begin{array}{ccc}
\sigma^{x,aa} & \sigma^{x,ab} & \sigma^{x,ac} \\
\sigma^{x,ab} & \sigma^{x,bb} & \sigma^{x,bc} \\
\sigma^{x,ac} & \sigma^{x,bc} & \sigma^{x,cc} \end{array} \right].
\]
Discretization of these continuous distributions is performed in the same way as described in Section \ref{sec:storage_params}. The set of values onto which we discretize is a hyperrectangle, where dimension $i$ takes values between $\mu^{x,i} \pm 3 \, \sqrt{\sigma^{x,ii}}$, for $i \in \{a,b,c\}$. The distribution of $\hat{H}_{t+1}^d$ conditional on $x_t=x$ (change in side effect severity for treatment $x$) is a discrete distribution that takes values $\hat{h}^{x,d}$ and with probabilities $p^{x,d}$ (both vectors), for all $t$. The numerical values of these parameters are given in Table \ref{table:glycemic_params}.
\begin{table}[h]
\centering
\scriptsize
\begin{tabular}{@{}crrrrrrrrrrr@{}}\toprule
\textbf{Treatment} & \multicolumn{1}{c}{$\mu^{x,a}$} & \multicolumn{1}{c}{$\mu^{x,b}$} & \multicolumn{1}{c}{$\mu^{x,c}$} & \multicolumn{1}{c}{$\sigma^{x,aa}$} & \multicolumn{1}{c}{$\sigma^{x,bb}$} & \multicolumn{1}{c}{$\sigma^{x,cc}$} & \multicolumn{1}{c}{$\sigma^{x,ab}$} & \multicolumn{1}{c}{$\sigma^{x,ac}$} & \multicolumn{1}{c}{$\sigma^{x,bc}$} & \multicolumn{1}{c}{$\hat{h}^{x,d}$} & \multicolumn{1}{c}{$p^{x,d}$}\\
\midrule
$x_t = 0$ & 30 & 3 & 2 & 25 & 8 & 8 & 0.8 & 0.5 & 0.2 & $[-1,0]$ & $[0.8,0.2]$\\
$x_t = 1$ & $-25$ & $-1$ & 3 & 100 & 16 & 25 & 1.2 & 0.5 & 0.2 & $[0,1,2]$ & $[0.8,0.1,0.1]$\\
$x_t = 2$ & $-45$ & $-3$ & 5 & 100 & 16 & 25 & 1.2 & 0.5 & 0.1 & $[0,1,2]$ & $[0.75,0.15,0.1]$\\
$x_t = 3$ & $-10$ & $-1$ & $-1$ & 81 & 10 & 16 & 0.6 & 0.5 & 0.5 & $[0,1,2]$ & $[0.8,0.1,0.1]$\\
$x_t = 4$ & $-10$ & $-1$ & $-4$ & 81 & 10 & 16 & 1.2 & 0.5 & 0.5 & $[0,1,2]$ & $[0.7,0.2,0.1]$\\
\bottomrule
\end{tabular}
\vspace{1em}
\caption{Parameter Values for Glycemic Control Problem}
\label{table:glycemic_params}
\end{table}
Lastly, as we did for the previous two problems, Table \ref{table:sizes3} shows state space and computation time information for the glycemic control problem. The initial state is set to $S_0 = (H_\textnormal{max}^a,H_\textnormal{max}^b,H_\textnormal{max}^c,H_\textnormal{min}^d)$, to represent an unhealthy diabetes patient who has not undergone any treatment (and therefore, no side effects).
\begin{table}[h]
\centering
\scriptsize
\begin{tabular}{@{}crrrr@{}}\toprule
\multicolumn{1}{c}{\textbf{Label}} & \multicolumn{1}{c}{\textbf{State Space}} & \multicolumn{1}{c}{\textbf{Eff. State Space}} & \multicolumn{1}{c}{\textbf{Action Space}} & \multicolumn{1}{c}{\textbf{CPU (Sec.)}}\\
\midrule
\texttt{G1}/\texttt{G2} & 1{,}312{,}256 & 17{,}059{,}328 & 5 & 201{,}925\\
\bottomrule
\end{tabular}
\vspace{1em}
\caption{Basic Properties of Glycemic Control Problem Instances}
\label{table:sizes3}
\end{table}

\subsubsection{Results}
In the numerical work for glycemic control, we show a slightly different algorithmic phenomenon. Recall that in \texttt{G1}, there no cost of treatment, and consequently, the contribution function is independent of the treatment decision. It turns out that after relatively few iterations, \emph{all} of the ADP algorithms are able to learn that there is \emph{some value to be gained} by applying treatment. Figure \ref{subfig:GiterA} shows that they end up achieving policies between 75\% and 90\% of optimality, with Monotone--ADP outperforming KBRL by roughly 10\% and AVI (the worst performing) by only 15\%. What if we add in a seemingly minor treatment cost of $P=2$? Figure \ref{subfig:GiterB}, on the other hand, shows a dramatic failure of AVI: it never improves beyond 10\% of optimality. API shows similar behavior, and KBRL performed slightly worse (approx. 5\%) than it did on \texttt{G1} and reached a plateau in the middle iterations. The reason for AVI's poor performance is that it updates values too slowly from the initialization of $\widebar{V}^0$ (usually constant over $\mathcal S$) --- in other words, the future value term of Bellman's equation, $\mathbf{E} \bigl[ \widebar{V}^n_{t+1}(S_{t+1})\,|\,S_t = s, \, a_t=a \bigr]$, is unable to compensate for the treatment cost of $P=2$ quickly enough. We therefore conclude that it can be crucially important to update large swathes of states at a time, while observing the structural behavior. Even though KBRL and API do generalize to the entire state space, our observations from comparing \texttt{G1} and \texttt{G2} point to the additional value gained from using structural information.
\begin{figure}[!ht]
        \centering
        \begin{subfigure}[b]{0.35\textwidth}
                \centering
                \includegraphics[width=\textwidth]{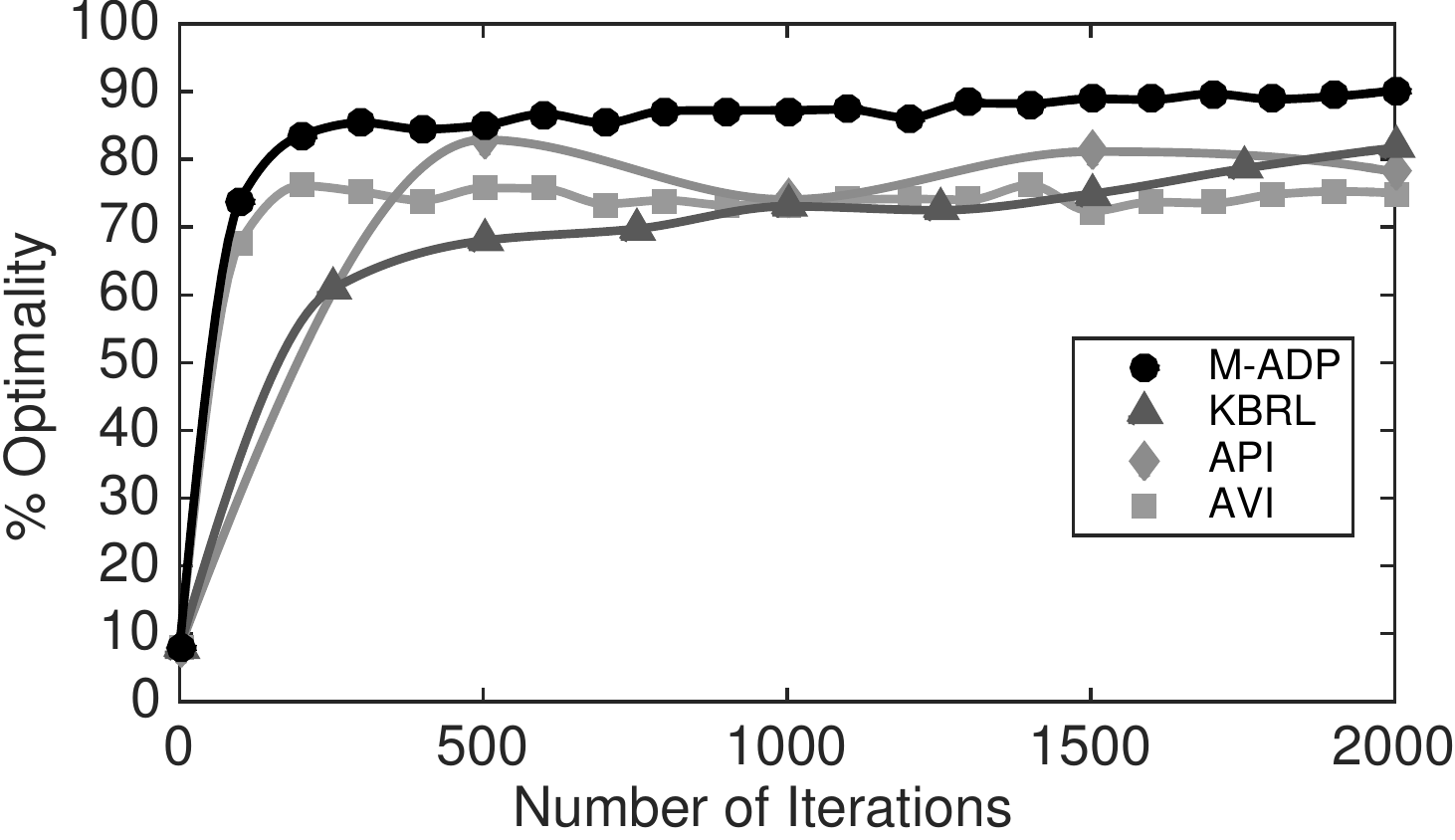}
                \caption{Instance \texttt{G1}}
                \label{subfig:GiterA}
        \end{subfigure}
        \begin{subfigure}[b]{0.35\textwidth}
                \centering
                \includegraphics[width=\textwidth]{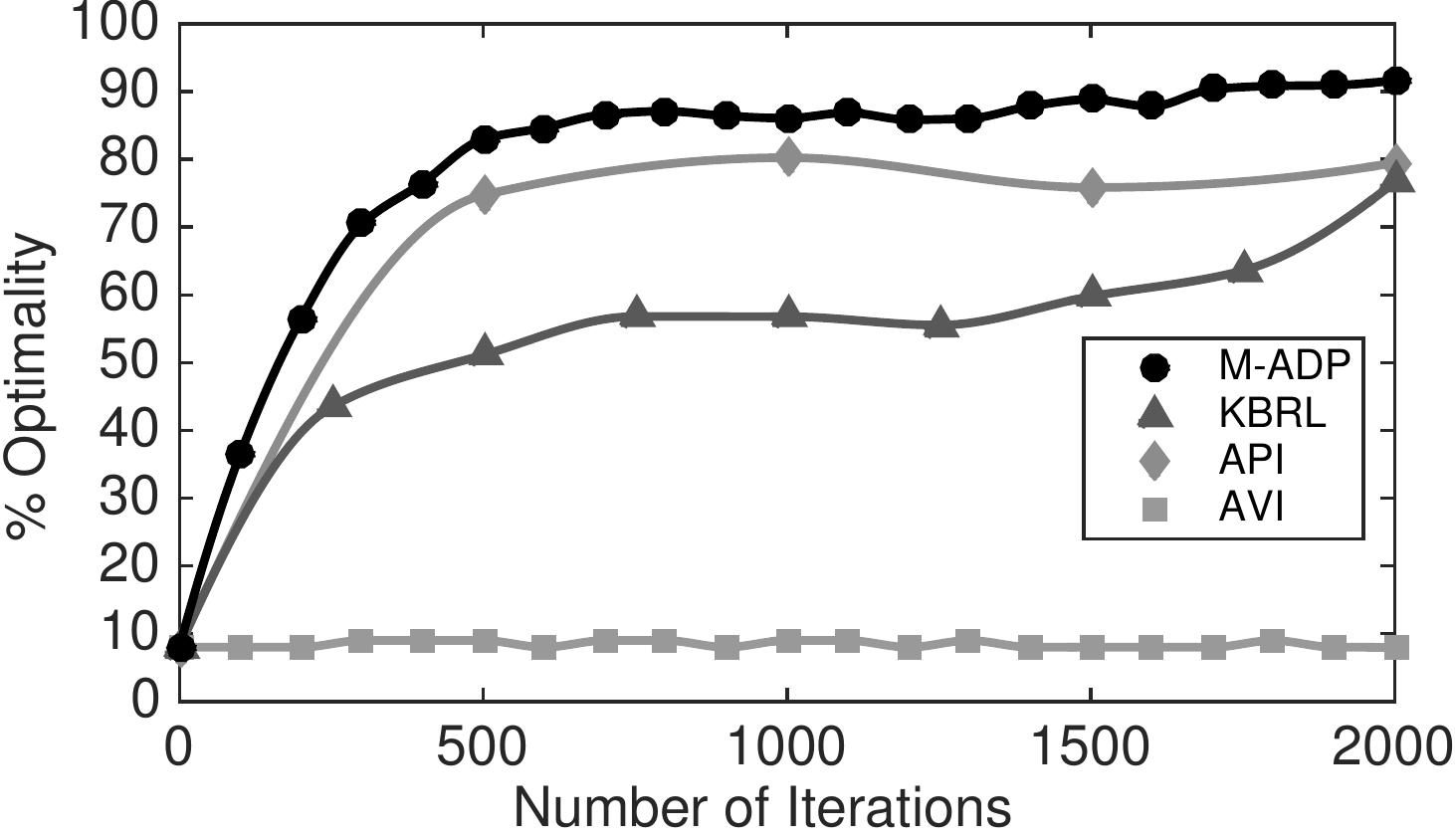}
                \caption{Instance \texttt{G2}}
                \label{subfig:GiterB}
        \end{subfigure}
        \caption{Comparison of Monotone--ADP to Other ADP/RL Algorithms}
\end{figure}
The computation time results of Figure \ref{fig:Gcpu} are similar to that of the previous examples. Monotone--ADP once again produces near--optimal solutions with significantly less computation: a ratio of 1.5\% for \texttt{G1} and 1.2\% for \texttt{G2}.
\begin{figure}[!ht]
        \centering
        \begin{subfigure}[b]{0.35\textwidth}
                \centering
                \includegraphics[width=\textwidth]{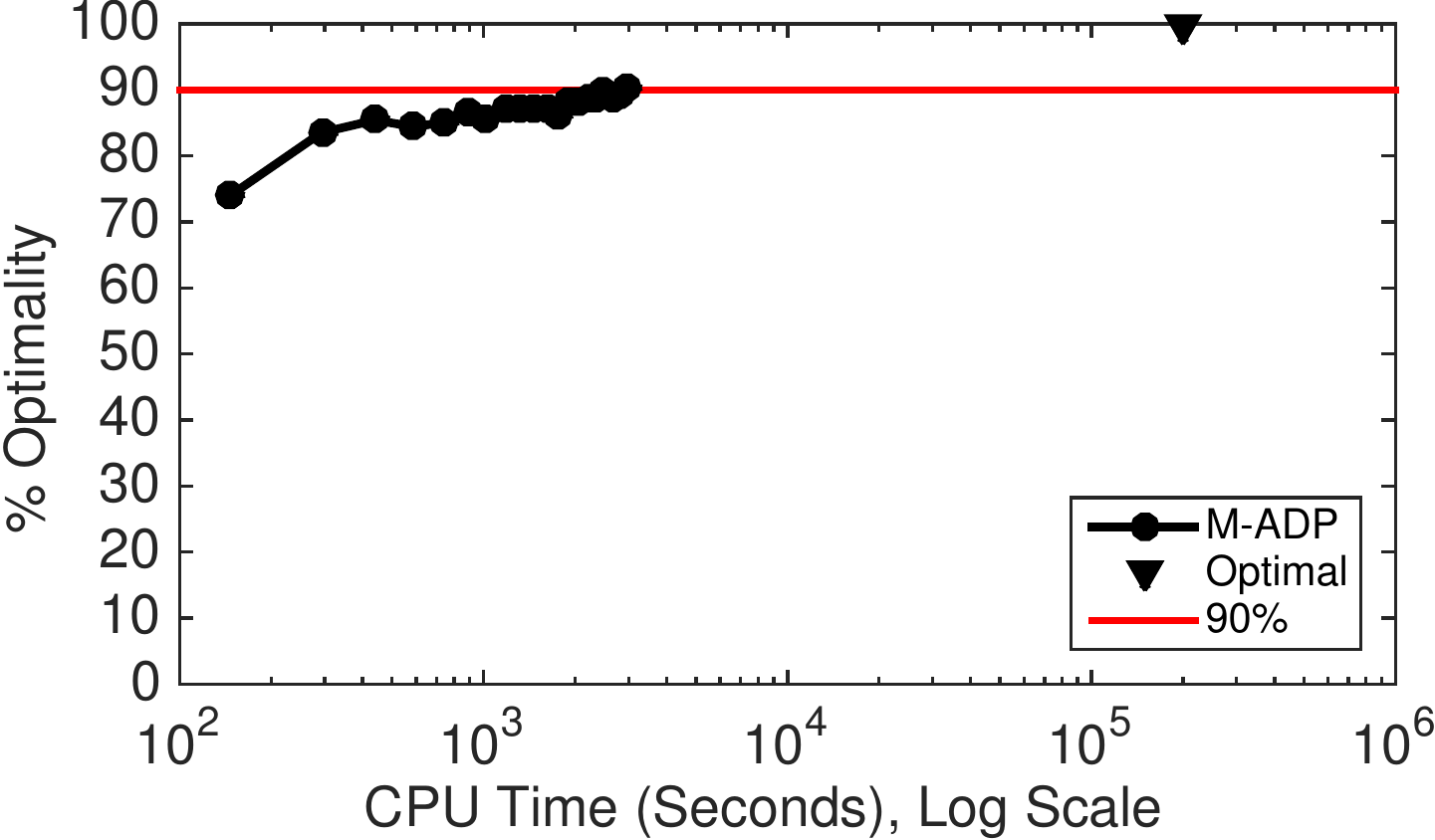}
                \caption{Instance \texttt{G1}}
        \end{subfigure}
        \begin{subfigure}[b]{0.35\textwidth}
                \centering
                \includegraphics[width=\textwidth]{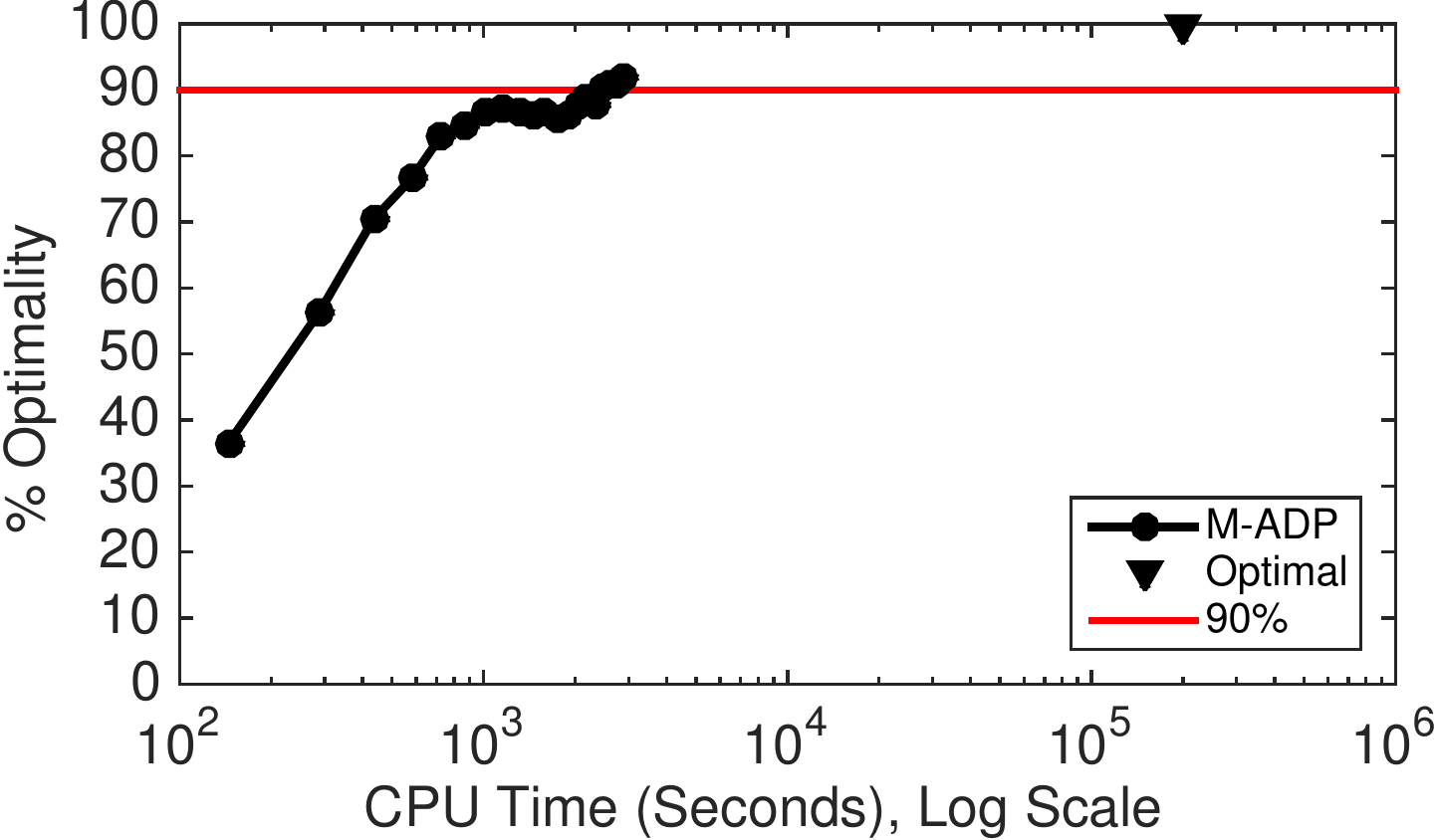}
                \caption{Instance \texttt{G2}}
        \end{subfigure}
        \caption{Computation Times (Seconds) of Monotone--ADP vs. Backward Dynamic Programming}
        \label{fig:Gcpu}
\end{figure}

\section{Conclusion}
\label{sec:conclusion}
In this paper, we formulated a general sequential decision problem with the property of a monotone value function. We then formally described an ADP algorithm, Monotone--ADP, first proposed in \cite{Papadaki2002} as a heuristic for one dimensional state variables, that exploits the structure of the value function by performing a monotonicity preserving operation at each iteration to increase the information gained from each random observation. The algorithm can be applied in the context of three common formulations of Bellman's equation, the pre--decision version, post--decision version and the $Q$--function (or state--action) version. Under several technical assumptions, we prove the almost sure convergence of the estimates produced by Monotone--ADP to the optimal value function. The proof draws upon techniques used in \cite{Tsitsiklis1994a} and \cite{Nascimento2009a}. However, in \cite{Nascimento2009a}, where concavity was assumed, pure exploitation could be used, but only in one dimension. This paper requires a full exploration policy, but exploits monotonicity in multiple dimensions, dramatically accelerating the rate of convergence. We then presented three example applications: regenerative optimal stopping, energy storage/allocation, and glycemic control for diabetes patients. Our empirical results show that in these problem domains, Monotone--ADP outperforms several popular ADP/RL algorithms (kernel--based reinforcement learning, approximate policy iteration, asynchronous value iteration, and $Q$--learning) by exploiting the additional structural information. Moreover, it can produce near--optimal policies using up to \emph{two orders of magnitude} less computational resources than backward dynamic programming. In an application where the optimal solution cannot be computed exactly due to a large state space, we expect that utilizing monotonicity can bring significant advantages. 
\clearpage
\appendix
\section{Proofs}
\label{sec:appendix}
\monocondone*
\begin{proof}
This is easily shown using backwards induction starting from the base case of $V^*_T$, which satisfies (\ref{monogen}) by definition. Consider two states $s = (m,i)$ and $s' = (m',i')$ with $s \preceq s'$ (note that this implies $i=i'$). Given that $V_{t+1}^*$ satisfies (\ref{monogen}), applying $(i)$, $(iii)$, and the monotonicity of the conditional expectation, we see that for any $a \in \mathcal A$,
\begin{align*}
\mathbf{E} \Bigl [  V^*_{t+1}\bigl(f(s,a,W_{t+1})\bigr)\,|\, S_t=s, \, a_t=a\Bigr] &= \mathbf{E} \Bigl [  V^*_{t+1}\bigl(f(s,a,W_{t+1})\bigr)\,|\, I_t=i, \, a_t=a \Bigr]\nonumber\\
&\le  \mathbf{E} \Bigl [  V^*_{t+1}\bigl(f(s',a,W_{t+1})\bigr)\,|\, I_t=i', \, a_t=a \Bigr]\nonumber\\
&= \mathbf{E} \Bigl [  V^*_{t+1}\bigl(f(s',a,W_{t+1})\bigr)\,|\, S_t=s', \, a_t=a \Bigr]. 
\end{align*}
Summing with the inequality in $(ii)$, we have for any $a \in \mathcal A$,
\[
C_t(s,a) + \mathbf{E} \bigl [  V^*_{t+1}(S_{t+1})\,|\, S_t=s, \, a_t=a \bigr] \le C_t(s',a) + \mathbf{E} \bigl [  V^*_{t+1}(S_{t+1})\,|\, S_t=s', \, a_t=a \bigr],
\]
and to complete the proof, we take a supremum over $\mathcal A$ to see that $V_t^*(s) \le V_t^*(s')$.
\end{proof}

\Hprops*
\begin{proof}
The following arguments apply to all three definitions of $H$. $(i)$ is true due to the monotonicity of the supremum (or maximum) and expectation operators. $(ii)$ follows by the definition of $H$ (for each of the three cases) and $(iii)$ follows from the well--known fact that our finite horizon MDP has a unique solution. $(iii)$ is easily deduced by applying the definition of $H$.
\end{proof}

\Propproj*
\begin{proof}
Let $\widebar{V}_t^\Pi = \Pi_M\bigl(S_t^n, z_t^n(S_t^n), \widebar{V}_t^{n-1}\bigr)$ and consider an arbitrary feasible solution $\tilde{V}_t \in V_\mathcal{M}\bigl(S_t^n,z_t^n(S_t^n)\bigr)$ to (\ref{prop:argmin}). To prove the statement of the proposition, we need to argue that
\[
\bigl \| \widebar{V}_t^\Pi - \widebar{V}_t^{n-1} \bigr \|_2 \le \bigl \| \tilde{V}_t - \widebar{V}_t^{n-1} \bigr \|_2.
\]
We do so by verifying that for each state $s \in \mathcal S$,
\begin{equation}
\bigl|\widebar{V}^\Pi_{t}(s)-\widebar{V}^{n-1}_{t}(s)\bigr|^2 \le \bigl|\tilde{V}_{t}(s)-\widebar{V}^{n-1}_{t}(s)\bigr|^2.
\label{toverify}
\end{equation}
There are four cases:
\begin{enumerate}
\item If $s$ and $S_t^n$ are incomparable (neither $s \preceq S_t^n$ nor $S_t^n \preceq s$) then monotonicity does not apply and $\widebar{V}_t^\Pi(s) = \widebar{V}_t^{n-1}(s)$, trivially satisfying (\ref{toverify}).
\item If $s = S_t^n$, then the definition of the feasible region gives $\widebar{V}_t^\Pi(s) = \tilde{V}_t(s)$. Once again, (\ref{toverify}) is satisfied.
\item Consider the case where $s \succeq S_t^n$ and $s \ne S_t^n$. First, if monotonicity is not violated with the new observation $z_t^n(S_t^n)$, then $\Pi_M$ does not alter the value at $s$. Therefore, $\widebar{V}_t^\Pi(s) = \widebar{V}_t^{n-1}(s)$ and (\ref{toverify}) holds. Now suppose monotonicity is violated, meaning that $\widebar{V}_{t}^{n-1}(s) \le z_t^n(S_t^n)$. After applying $\Pi_M$, we have that $\widebar{V}_{t}^\Pi(s) = z_t^n(S_t^n)$. Since $\tilde{V}_t$ is in the feasible set, it must be the case that $\tilde{V}_{t}(s) \ge z_t^n(S_t^n)$, since $\tilde{V}_t(S_t^n) = z_t^n(S_t^n)$. Combining these relationships, it is clear that (\ref{toverify}) holds.
\item The case where $s \preceq S_t^n$, and $s \ne S_t^n$ is handled in an analogous way.
\end{enumerate}
Since this holds for any feasible solution $\tilde{V}_t$, the proof is complete.
\end{proof}

\LemmaTwo*
\begin{proof}
The proof is by induction on $k$. We note that by definition and (\ref{monogen}), $U^0$ and $L^0$ satisfy this property. By the definition of $H$ and Assumption \ref{ass:Emono}, it is easy to see that if $U^k$ satisfies the property, then $HU^k$ does as well. Thus, by the definition of $U^{k+1}$ being the average of the two, we see that $U^{k+1}$ also satisfies the monotonicity property. 
\end{proof}
\notempty*
\begin{proof}
Consider $\mathcal S_t^-$ (the other case is symmetric). Since $\mathcal S$ is finite, there exists a state $\underline{s}$ such that there is no state $s' \in \mathcal S$ where $s' \preceq \underline{s}$. An increase from the projection operator must originate from a violation of monotonicity during an observation of a state $s'$ where $s' \preceq \underline{s}$ and $s' \ne \underline{s}$, but such a state does not exist. Thus, $\underline{s} \in \mathcal S_t^-$.
\end{proof}
\lowerimmediate*
\begin{proof}
Define:
\begin{equation*}
A = \{s'': s'' \preceq s,\, s'' \ne s \} \quad \mbox{and} \quad B = \! \!\! \!\bigcup_{s_l \in \mathcal S_L(s)} \{s'': s'' \preceq s_l\}.
\end{equation*}
We argue that $A \subseteq B$. Choose $s_1 \in A$ and suppose for the sake of contradiction that $s_1 \not\in B$. In particular, this means that $s_1 \not\in \mathcal S_L(s)$ because $\mathcal S_L(s) \subseteq B$. By Definition \ref{defn:lower}, it follows that there must exist $s_2$ such that $s_1 \preceq s_2 \preceq s$ where $s_2 \ne s_1$ and $s_2 \ne s$. It now follows that $s_2 \not\in \mathcal S_L(s)$ because if it were, then $s_1 \preceq s_2$ would imply that $s_1$ is an element of $B$. This argument can be repeated to produce other states $s_3, s_4, \ldots$, each \emph{different} from the rest, such that
\begin{equation}
s_1 \preceq s_2 \preceq s_3 \preceq s_4 \preceq \ldots \preceq s,
\label{chain}
\end{equation}
where each state $s_k$ is not an element $\mathcal S_L(s)$. However, because $\mathcal S$ is a finite set, eventually we reach a point where we cannot produce another state to satisfy Definition \ref{defn:lower} and we will have that the final state, call it $s_K$, is an element of $\mathcal S_L$. Here, we reach a contradiction because (\ref{chain}) (specifically the fact that $s_1 \preceq s_K$) implies that $s_1 \in B$. Thus, $s_1 \in B$ and we have shown that $A \subseteq B$.

Because the value of $s$ was increased, a violation of monotonicity must have occurred during the observation of $S_t^n$, implying that $S_t^n \in A$. Therefore, $S_t^n \in B$, and we know that $S_t^n \preceq s'$ for some $s' \in \mathcal S_L(s)$. Since $\widebar{V}_t^{n-1}$ is monotone over $\mathcal S$ and $\widebar{V}_t^n(s) = z_t^n(S_t^n)$, we can write
\begin{equation*}
\widebar{V}_t^{n-1}(S_t^n) \le \widebar{V}_t^{n-1}(s') \le \widebar{V}_t^{n-1}(s) < z_t^n(S_t^n) = \widebar{V}_t^n(s),
\end{equation*}
meaning that $\Pi_M$ acts on $s'$ and we have
\begin{equation*}
\widebar{V}_t^{n}(s')= z_t^n(S_t^n) = \widebar{V}_t^n(s),
\end{equation*}
the desired result.
\end{proof}
\alphaprod*
\begin{proof}We first notice that the product inside the limit is nonnegative because the stepsizes $\alpha_t^n \le 1$. Also the sequence is monotonic; therefore, the limit exists. Now,
\begin{equation*}
\begin{aligned}
\prod_{n=1}^m \bigl(1-\alpha_t^n(s)\bigr) = \exp \left[ \sum_{n = 1}^m \log \bigl(1-\alpha_t^n(s)\bigr)\right] \le \exp \left[ -\sum_{n=1}^m\alpha_t^n(s)  \right],
\end{aligned}
\end{equation*}
where the inequality follows from $\log(1-x) \le -x$. Since
$$\sum_{n=1}^\infty \alpha_t^n(s) = \infty \quad a.s.,$$
the result follows by appropriately taking limits.
\end{proof}
\XW*
\begin{proof}
First, we state an inequality needed later in the proof. Since $n \ge \tilde{N}_t^{k}(s)$, we know that $n \ge N_{t+1}^k$ by the induction hypothesis for $k$. Now by the induction hypothesis for $t+1$, we know that $\widebar{V}_{t+1}^n(s) \le U_{t+1}^k(s)$ holds. Therefore, using $(ii)$ of Lemma \ref{Hprops}, we see that
\begin{equation}
\bigl(H\widebar{V}^n\bigr)_t(s) \le \bigl(HU^k\bigr)_t(s).
\label{ineq:1}
\end{equation}
To show the statement of the lemma, we induct forwards on $n$.
\subsubsection*{Base case, $n=\tilde{N}_t^{k}(s)$} By the induction hypothesis for $k$ (which we can safely use in this proof because of the placement of the lemma after the induction hypothesis), we have that $\widebar{V}_t^{\tilde{N}_t^{k}(s)}(s) \le U_t^k(s)$. Combined with the fact that
\begin{equation*}
W_t^{\tilde{N}_t^{k}(s),\,\tilde{N}_t^{k}(s)}(s)=0 \textnormal{ and } U_t^k(s) = X_t^{\tilde{N}_t^{k}(s)} (s),
\end{equation*}
we see that the statement of the lemma holds for the base case.

\subsubsection*{Induction hypothesis, $n$} Suppose the statement of the lemma holds for $n$.

\subsubsection*{Inductive step from $n$ to $n+1$} Suppose $S_t^{n+1}=s$, meaning a direct update happened on iteration $n+1$. Thus, we have that
\begin{align}
\widebar{V}_t^{n+1}(s) &=  z_t^{n+1}(s) \nonumber\\
&= \bigl(1-\alpha_t^{n}(s)\bigr) \, \widebar{V}_t^n(s) + \alpha_t^{n}(s) \, \hat{v}_t^{n+1}(s) \nonumber \\
&= \bigl(1-\alpha_t^{n}(s)\bigr) \, \widebar{V}_t^n(s) + \alpha_t^{n}(s) \, \Bigl[\bigl(H\widebar{V}^n\bigr)_t(s) + w_t^{n+1}(s)\Bigr] \nonumber \\
&\le \bigl(1-\alpha_t^{n}(s)\bigr) \, \bigl(X_t^{n}(s) + W_t^{n,\tilde{N}_t^{k}(s)}(s)\bigr) + \alpha_t^{n}(s) \, \Bigl[\bigl(H\widebar{V}^n\bigr)_t(s) + w_t^{n+1}(s)\Bigr] \label{indhyp}\\
&\le \bigl(1-\alpha_t^{n}(s)\bigr) \, \bigl(X_t^{n}(s) + W_t^{n,\tilde{N}_t^{k}(s)}(s)\bigr) + \alpha_t^{n}(s) \, \Bigl[\bigl(HU^k\bigr)_t(s) + w_t^{n+1}(s)\Bigr] \label{ineq:2}\\
&= X_t^{n+1}(s) + W_t^{n+1,\tilde{N}_t^{k}(s)} (s).\nonumber
\end{align}
where (\ref{indhyp}) follows from the induction hypothesis for $n$, (\ref{ineq:2}) follows from (\ref{ineq:1}).
Now let us consider the second case, that $S_t^{n+1} \ne s$. This means that the stepsize $\alpha_t^{n+1}(s)=0$ and thus,
\begin{equation}
\begin{aligned}
X_t^{n+1}(s) &=X_t^{n}(s),\\
W_t^{n+1,\tilde{N}_t^{k}(s)}(s) &= W_t^{n,\tilde{N}_t^{k}(s)}(s).
\end{aligned}
\label{eq:1}
\end{equation}
Because $s \in \mathcal S_t^-$ and $n+1 \ge \tilde{N}_t^{k}(s) \ge N^\Pi$ (induction hypothesis for $k$), we know that the projection operator did not increase the value of $s$ on this iteration (though a decrease is possible). Hence,
\begin{equation*}
\widebar{V}_t^{n+1}(s) \le \widebar{V}_t^n(s) \le X_t^{n}(s) + W_t^{n,\tilde{N}_t^{k}(s)}(s) \le X_t^{n+1}(s) + W_t^{n+1,\tilde{N}_t^{k}(s)}(s),
\end{equation*}
by the induction hypothesis for $n$ and (\ref{eq:1}).
\end{proof}

\replmono*
\begin{proof}
The proof of the inductive step of Proposition \ref{mono_cond_one} can be used, but we must take care of a special case.
Consider any $t < T$ and any $V \in \mathbb R^D$ such that $V_{t+1}$ is monotone over $\mathcal S$. We wish to show that $(HV)_t(s) \le (HV)_t(s')$ for any two states $s \preceq s'$. If $s =(x,y)$ and $s' = (x',y')$ with $0<x \le x'$ and $y \le y'$, then the inductive step of the proof of Proposition \ref{mono_cond_one} can be applied. Similarly, if $x = x' = 0$, then automatic regeneration of the system happens in both cases and we can again apply the inductive step in the proof of Proposition \ref{mono_cond_one}.

However, we now notice that Part (i) of Proposition \ref{mono_cond_one} is not satisfied whenever we have $s=(0,y)$ and $s'=(x,y')$ for $y,\,y' \in \mathcal Y$ and $x \in \mathcal X$ with $x >0$ and $y \le y'$ (due to the automatic regeneration of the system after failure). When in state $s$, regardless of the action taken, the next state is always $(X_\textnormal{max},Y_\textnormal{max})$, where $Y_\textnormal{max} = (Y^i_\textnormal{max})_{i=1}^n$. The contribution that we incur is $-F-r(0,y)$. On the other hand, if we take action $a_t = 1$ at the current state of $s'$, the next state is also $(X_\textnormal{max},Y_\textnormal{max})$, but the contribution incurred is $P-r(x,y')$. As $r(\cdot,\cdot)$ is nonincreasing, we have that $P-r(x,y_2) \ge -F-r(0,y_1)$. Since there is an action that leads $s'$ to the same next state with a higher contribution, we have shown that $(HV)_t(s) \le (HV)_t(s')$.
\end{proof}
\storagemono*
\begin{proof}
Consider any $t < T$ and any $V \in \mathbb R^D$ such that $V_{t+1}$ is monotone over $\mathcal S$. Let $s = (r,e,p,d)$ and $s' = (r',e',p',d')$ be two states in $\mathcal S$ with $s \preceq s'$. The main difficulty we need to work around is the fact that the feasible sets $\mathcal X(s)$ and $\mathcal X(s')$ are different. The goal is to show $(HV)_t(s) \le (HV)_t(s')$, which we can rewrite as
\begin{equation}
\begin{aligned}
\max_{x \in \mathcal X(s)} \Bigl [ C_t(s,x) &+\textbf{E}\bigl[ V_{t+1}(S_{t+1}) \,| \, S_t=s, \, x_t=x \bigr] \Bigr ]\\ &\le 
\max_{x \in \mathcal X(s')} \Bigl [ C_t(s',x)+\textbf{E}\bigl[ V_{t+1}(S_{t+1}) \,| \, S_t=s', \, x_t=x \bigr] \Bigr ],
\label{eq:storage_ineq}
\end{aligned}
\end{equation}
and let $x^* = (\xstart[ed]\!,\xstart[md]\!,\xstart[rd]\!,\xstart[er]\!,\xstart[rm])^\mathsf{T}$ be an optimal solution to the left--hand--side maximization. There are now two cases. First, consider the case where $x^* \in \mathcal X(s')$. Since $x^*$ is an element of $\mathcal X(s)$, we know that by the constraints (\ref{eq:constraint1}) and (\ref{eq:constraint2}),
\[
d' + \xstart[rm] - \xstart[md] \ge d + \xstart[rm] - \xstart[md] \ge 0,
\]
which implies that $C_t(s,x^*) \le C_t(s',x^*)$. Now, using the fact that the transition function is monotone and the independence of $W_{t+1}$ from $S_t$, the same argument as in the proof of Proposition \ref{mono_cond_one} gives us 
\[
\textbf{E}\bigl[ V_{t+1}(S_{t+1}) \,| \, S_t=s, \, x_t=x^* \bigr] \le \textbf{E}\bigl[ V_{t+1}(S_{t+1}) \,| \, S_t=s', \, x_t=x^* \bigr].
\]
Therefore, we have a feasible solution that has an objective value at least as good as the optimal objective value from the left hand side, which is enough to conclude (\ref{eq:storage_ineq}). Next, we consider the case where $x^* \not\in \mathcal X(s')$. Since $s \preceq s'$, the only two possible violations by $x^*$ are constraints (\ref{eq:constraint2}) and (\ref{eq:constraint6}) in the definition of $\mathcal X(s')$. Let $\delta^d = d'-d$ and $\delta^r = r'-r$. It follows that the decision $\bar{x}$ defined by
\[\bar{x} = \bigl(\xstart[ed]\!,\,\xstart[md]+\delta^d\!,\,\xstart[rd]\!,\,\xstart[er]-\delta^r\!,\,\xstart[rm]\bigr)^\mathsf{T}\]
is in the feasible set $\mathcal X(s')$. In addition, we see that the choice of this specific $\bar{x}$ guarantees that $C_t(s,x^*) \le C_t(s',\bar{x})$ and $r + \phi^\mathsf{T} x^* \le r' + \phi^\mathsf{T} \bar{x}$. The monotonicity of (\ref{eq:storage_trans}) and the independence of $W_{t+1}$ from $S_t$ and $x_t$ allow us to conclude 
\[
\textbf{E}\bigl[ V_{t+1}(S_{t+1}) \,| \, S_t=s, \, x_t=x^* \bigr] \le \textbf{E}\bigl[ V_{t+1}(S_{t+1}) \,| \, S_t=s', \, x_t=\bar{x} \bigr].
\]
As in the previous case, the existence of $\bar{x}$ implies (\ref{eq:storage_ineq}).
\end{proof}

\clearpage
\bibliographystyle{abbrvnat}
\bibliography{/Users/drjiang/Documents/Dropbox/Princeton/Princeton_Research/Bibtex/Bib}

\begin{thebibliography}{53}
\providecommand{\natexlab}[1]{#1}
\providecommand{\url}[1]{\texttt{#1}}
\expandafter\ifx\csname urlstyle\endcsname\relax
  \providecommand{\doi}[1]{doi: #1}\else
  \providecommand{\doi}{doi: \begingroup \urlstyle{rm}\Url}\fi

\bibitem[Asamov and Powell(2015)]{Asamov2015}
T.~Asamov and W.~B. Powell.
\newblock {Regularized decomposition of high-dimensional multistage stochastic
  programs with Markov uncertainty}.
\newblock \emph{arXiv preprint arXiv:1505.02227}, 2015.

\bibitem[Ayer et~al.(1955)Ayer, Brunk, and Ewing]{Ayer1955}
M.~Ayer, H.~D. Brunk, and G.~M. Ewing.
\newblock {An empirical distribution function for sampling with incomplete
  information}.
\newblock \emph{The Annals of Mathematical Statistics}, 26\penalty0
  (4):\penalty0 641--647, 1955.

\bibitem[Barlow et~al.(1972)Barlow, Bartholomew, Bremner, and Brunk]{Barlow}
R.~E. Barlow, D.~J. Bartholomew, J.~M. Bremner, and H.~D. Brunk.
\newblock \emph{{Statistical inference under order restrictions: The theory and
  application of isotonic regression}}.
\newblock Wiley, New York, 1972.

\bibitem[Bertsekas(2007)]{Bertsekas2007}
D.~P. Bertsekas.
\newblock \emph{{Dynamic Programming and Optimal Control, Vol. II}}.
\newblock Athena Scientific, Belmont, MA, 4 edition, 2007.

\bibitem[Bertsekas(2011)]{Bertsekas2011}
D.~P. Bertsekas.
\newblock {Approximate policy iteration: A survey and some new methods}.
\newblock \emph{Journal of Control Theory and Applications}, 2011.

\bibitem[Bertsekas and Tsitsiklis(1996)]{Bertsekas1996}
D.~P. Bertsekas and J.~N. Tsitsiklis.
\newblock \emph{{Neuro--Dynamic Programming}}.
\newblock Athena Scientific, Belmont, MA, 1996.

\bibitem[Birge(1985)]{Birge1985}
J.~R. Birge.
\newblock {Decomposition and partitioning methods for multistage stochastic
  linear programs}, 1985.

\bibitem[Breiman(1992)]{Breiman1992}
L.~Breiman.
\newblock \emph{{Probability}}.
\newblock Society of Industrial and Applied Mathematics, Philadelphia, PA,
  1992.

\bibitem[Brunk(1955)]{Brunk1955}
H.~D. Brunk.
\newblock {Maximum likelihood estimates of monotone parameters}.
\newblock \emph{The Annals of Mathematical Statistics}, pages 607--616, 1955.

\bibitem[Carmona and Ludkovski(2010)]{Carmona2010}
R.~Carmona and M.~Ludkovski.
\newblock {Valuation of energy storage: An optimal switching approach}.
\newblock \emph{Quantitative Finance}, 10\penalty0 (4):\penalty0 359--374,
  2010.

\bibitem[Dette et~al.(2006)Dette, Neumeyer, and Pilz]{Dette2006}
H.~Dette, N.~Neumeyer, and K.~F. Pilz.
\newblock {A simple nonparametric estimator of a strictly monotone regression
  function}.
\newblock \emph{Bernoulli}, 12\penalty0 (3):\penalty0 469--490, 2006.

\bibitem[Ekstr\"{o}m(2004)]{Ekstrom2004}
E.~Ekstr\"{o}m.
\newblock {Properties of American option prices}.
\newblock \emph{Stochastic Processes and their Applications}, 114\penalty0
  (2):\penalty0 265--278, 2004.

\bibitem[Feldstein and Rothschild(1974)]{Feldstein1974}
M.~S. Feldstein and M.~Rothschild.
\newblock {Towards an economic theory of replacement investment}.
\newblock \emph{Econometrica: Journal of the Econometric Society}, 42\penalty0
  (3):\penalty0 393--424, 1974.

\bibitem[George and Powell(2006)]{George2006}
A.~P. George and W.~B. Powell.
\newblock {Adaptive stepsizes for recursive estimation with applications in
  approximate dynamic programming}.
\newblock \emph{Machine Learning}, 65\penalty0 (1):\penalty0 167--198, 2006.

\bibitem[Hsih(2010)]{Hsih2010}
K.~W. Hsih.
\newblock \emph{{Optimal Dosing Applied to Glycemic Control of Type 2
  Diabetes}}.
\newblock Senior thesis, Princeton University, 2010.

\bibitem[Jiang and Powell(2015)]{Jiang2013a}
D.~R. Jiang and W.~B. Powell.
\newblock {Optimal hour-ahead bidding in the real-time electricity market with
  battery storage using approximate dynamic programming}.
\newblock \emph{INFORMS Journal on Computing}, 27\penalty0 (3):\penalty0
  525--543, 2015.

\bibitem[Kaplan and Violante(2014)]{Kaplan2014}
G.~Kaplan and G.~L. Violante.
\newblock {A model of the consumption response to fiscal stimulus payments}.
\newblock \emph{Econometrica}, 82\penalty0 (4):\penalty0 1199--1239, 2014.

\bibitem[Kim and Powell(2011)]{Kim2011}
J.~H. Kim and W.~B. Powell.
\newblock {Optimal energy commitments with storage and intermittent supply}.
\newblock \emph{Operations Research}, 59\penalty0 (6):\penalty0 1347--1360,
  2011.

\bibitem[Kleywegt et~al.(2002)Kleywegt, Shapiro, and Homem-de
  Mello]{Kleywegt2002}
A.~J. Kleywegt, A.~Shapiro, and T.~Homem-de Mello.
\newblock {The sample average approximation method for stochastic discrete
  optimization}.
\newblock \emph{SIAM Journal on Optimization}, 12\penalty0 (2):\penalty0
  479--502, 2002.

\bibitem[Kurt and Kharoufeh(2010)]{Kurt2010a}
M.~Kurt and J.~P. Kharoufeh.
\newblock {Monotone optimal replacement policies for a Markovian deteriorating
  system in a controllable environment}.
\newblock \emph{Operations Research Letters}, pages 1--17, 2010.

\bibitem[Kurt et~al.(2011)Kurt, Denton, Schaefer, Shah, and Smith]{Kurt2011}
M.~Kurt, B.~T. Denton, A.~J. Schaefer, N.~D. Shah, and S.~A. Smith.
\newblock {The structure of optimal statin initiation policies for patients
  with type 2 diabetes}.
\newblock \emph{IIE Transactions on Healthcare Systems Engineering}, 1\penalty0
  (1):\penalty0 49--65, 2011.

\bibitem[Luenberger(1998)]{Luenberger1998}
D.~G. Luenberger.
\newblock \emph{{Investment Science}}.
\newblock Oxford University Press, New York, 1998.

\bibitem[Mammen(1991)]{Mammen1991}
E.~Mammen.
\newblock {Estimating a smooth monotone regression function}.
\newblock \emph{The Annals of Statistics}, 19\penalty0 (2):\penalty0 724--740,
  1991.

\bibitem[Mason et~al.(2012)Mason, England, Denton, Smith, Kurt, and
  Shah]{Mason2012}
J.~E. Mason, D.~A. England, B.~T. Denton, S.~A. Smith, M.~Kurt, and N.~D. Shah.
\newblock {Optimizing statin treatment decisions for diabetes patients in the
  presence of uncertain future adherence}.
\newblock \emph{Medical Decision Making}, 32\penalty0 (1):\penalty0 154--166,
  2012.

\bibitem[Mason et~al.(2014)Mason, Denton, Shah, and Smith]{Mason2014}
J.~E. Mason, B.~T. Denton, N.~D. Shah, and S.~A. Smith.
\newblock {Optimizing the simultaneous management of blood pressure and
  cholesterol for type 2 diabetes patients}.
\newblock \emph{European Journal of Operational Research}, 233\penalty0
  (3):\penalty0 727--738, 2014.

\bibitem[McCall(1970)]{McCall1970}
J.~J. McCall.
\newblock {Economics of information and job search}.
\newblock \emph{The Quarterly Journal of Economics}, 84\penalty0 (1):\penalty0
  113--126, 1970.

\bibitem[Mukerjee(1988)]{Mukerjee1988}
H.~Mukerjee.
\newblock {Monotone Nonparametric Regression}.
\newblock \emph{The Annals of Statistics}, 16\penalty0 (2):\penalty0 741--750,
  1988.

\bibitem[M\"{u}ller(1997)]{Muller1997}
A.~M\"{u}ller.
\newblock {How does the value function of a Markov decision process depend on
  the transition probabilities?}
\newblock \emph{Mathematics of Operations Research}, 22\penalty0 (4):\penalty0
  872--885, 1997.

\bibitem[Nadaraya(1964)]{Nadaraya1964}
E.~A. Nadaraya.
\newblock {On estimating regression}.
\newblock \emph{Theory of Probability \& Its Applications}, 9\penalty0
  (1):\penalty0 141--142, 1964.

\bibitem[Nascimento and Powell(2009)]{Nascimento2009a}
J.~M. Nascimento and W.~B. Powell.
\newblock {An optimal approximate dynamic programming algorithm for the lagged
  asset acquisition problem}.
\newblock \emph{Mathematics of Operations Research}, 34\penalty0 (1):\penalty0
  210--237, 2009.

\bibitem[Nascimento and Powell(2010)]{Nascimento2010}
J.~M. Nascimento and W.~B. Powell.
\newblock {Dynamic programming models and algorithms for the mutual fund cash
  balance problem}.
\newblock \emph{Management Science}, 56\penalty0 (5):\penalty0 801--815, 2010.

\bibitem[Ormoneit and Sen(2002)]{Ormoneit2002}
D.~Ormoneit and A.~Sen.
\newblock {Kernel-based reinforcement learning}.
\newblock \emph{Machine Learning}, 49\penalty0 (2-3):\penalty0 161--178, 2002.

\bibitem[Papadaki and Powell(2002)]{Papadaki2002}
K.~P. Papadaki and W.~B. Powell.
\newblock {Exploiting structure in adaptive dynamic programming algorithms for
  a stochastic batch service problem}.
\newblock \emph{European Journal of Operational Research}, 142\penalty0
  (1):\penalty0 108--127, 2002.

\bibitem[Papadaki and Powell(2003{\natexlab{a}})]{Papadaki}
K.~P. Papadaki and W.~B. Powell.
\newblock {A discrete online monotone estimation algorithm}.
\newblock \emph{Operational Research Working Papers, LSEOR 03.73},
  2003{\natexlab{a}}.

\bibitem[Papadaki and Powell(2003{\natexlab{b}})]{Papadaki2003}
K.~P. Papadaki and W.~B. Powell.
\newblock {An adaptive dynamic programming algorithm for a stochastic
  multiproduct batch dispatch problem}.
\newblock \emph{Naval Research Logistics}, 50\penalty0 (7):\penalty0 742--769,
  2003{\natexlab{b}}.

\bibitem[Pereira and Pinto(1991)]{Pereira1991a}
M.~V.~F. Pereira and L.~M. V.~G. Pinto.
\newblock {Multi-stage stochastic optimization applied to energy planning}.
\newblock \emph{Mathematical Programming}, 52\penalty0 (1-3):\penalty0
  359--375, 1991.

\bibitem[Pierskalla and Voelker(1976)]{Pierskalla1976}
W.~P. Pierskalla and J.~A. Voelker.
\newblock {A survey of maintenance models: the control and surveillance of
  deteriorating systems}.
\newblock \emph{Naval Research Logistics Quarterly}, 23\penalty0 (3):\penalty0
  353--388, 1976.

\bibitem[Powell(2011)]{Powell2011}
W.~B. Powell.
\newblock \emph{{Approximate Dynamic Programming: Solving the Curses of
  Dimensionality}}.
\newblock Wiley, 2nd edition, 2011.

\bibitem[Powell et~al.(2004)Powell, Ruszczynski, and Topaloglu]{Powell2004}
W.~B. Powell, A.~Ruszczynski, and H.~Topaloglu.
\newblock {Learning algorithms for separable approximations of discrete
  stochastic optimization problems}.
\newblock \emph{Mathematics of Operations Research}, 29\penalty0 (4):\penalty0
  814--836, 2004.

\bibitem[Puterman(1994)]{Puterman}
M.~L. Puterman.
\newblock \emph{{Markov Decision Processes: Discrete Stochastic Dynamic
  Programming}}.
\newblock Wiley, New York, 1994.

\bibitem[Ramsay(1998)]{Ramsay1998}
J.~O. Ramsay.
\newblock {Estimating smooth monotone functions}.
\newblock \emph{Journal of the Royal Statistical Society: Series B (Statistical
  Methodology)}, 60\penalty0 (2):\penalty0 365--375, 1998.

\bibitem[Ross(1983)]{Ross1983}
S.~M. Ross.
\newblock \emph{{Introduction to Stochastic Dynamic Programming}}.
\newblock Academic Press, New York, 1983.

\bibitem[Rust(1987)]{Rust1987}
J.~Rust.
\newblock {Optimal replacement of GMC bus engines: An empirical model of Harold
  Zurcher}.
\newblock \emph{Econometrica: Journal of the Econometric Society}, 55\penalty0
  (5):\penalty0 999--1033, 1987.

\bibitem[Salas and Powell(2013)]{Salas2013}
D.~Salas and W.~B. Powell.
\newblock {Benchmarking a scalable approximation dynamic programming algorithm
  for stochastic control of multidimensional energy storage problems}.
\newblock \emph{(working paper)}, 2013.

\bibitem[Scott(2009)]{Scott1992}
D.~W. Scott.
\newblock \emph{{Multivariate Density Estimation: Theory, Practice, and
  Visualization}}, volume 383.
\newblock 2009.

\bibitem[Scott and Powell(2012)]{Scott2012}
W.~Scott and W.~B. Powell.
\newblock {Approximate dynamic programming for energy storage with new results
  on instrumental variables and projected Bellman errors}.
\newblock \emph{(working paper)}, 2012.

\bibitem[Secomandi(2010)]{Secomandi2010}
N.~Secomandi.
\newblock {Optimal commodity trading with a capacitated storage asset}.
\newblock \emph{Management Science}, 56\penalty0 (3):\penalty0 449--467, 2010.

\bibitem[Smith and McCardle(2002)]{Smith2002}
J.~E. Smith and K.~F. McCardle.
\newblock {Structural properties of stochastic dynamic programs}.
\newblock \emph{Operations Research}, 50\penalty0 (5):\penalty0 796--809, 2002.

\bibitem[Stockey and {Lucas, Jr.}(1989)]{Stockey1989}
N.~Stockey and R.~E. {Lucas, Jr.}
\newblock \emph{{Recursive Methods in Economic Dynamics}}.
\newblock Harvard University Press, Cambridge, Massachusetts and London,
  England, 1989.

\bibitem[Sutton and Barto(1998)]{Sutton1998}
R.~Sutton and A.~Barto.
\newblock \emph{{Reinforcement Learning: An Introduction}}.
\newblock 1998.

\bibitem[Topaloglu and Powell(2003)]{Topaloglu2003}
H.~Topaloglu and W.~B. Powell.
\newblock {An algorithm for approximating piecewise linear concave functions
  from sample gradients}.
\newblock \emph{Operations Research Letters}, 31\penalty0 (1):\penalty0 66--76,
  2003.

\bibitem[Tsitsiklis(1994)]{Tsitsiklis1994a}
J.~N. Tsitsiklis.
\newblock {Asynchronous stochastic approximation and Q-learning}.
\newblock \emph{Machine Learning}, 16\penalty0 (3):\penalty0 185--202, 1994.

\bibitem[Watkins and Dayan(1992)]{Watkins1992}
C.~J. Watkins and P.~Dayan.
\newblock {Q-learning}.
\newblock \emph{Machine Learning}, 8\penalty0 (3-4):\penalty0 279--292, 1992.

\end{thebibliography}
\end{document}